\documentclass[letterpaper,11pt,reqno]{amsart}
\usepackage[T1]{fontenc}
\usepackage[utf8]{inputenc}
\usepackage{lmodern}
\usepackage[english]{babel}
\usepackage{amsmath,a4wide}
\usepackage{xfrac}
\usepackage{esint}
\usepackage{stmaryrd,mathrsfs,bm,amsthm,mathtools,yfonts,amssymb,color,braket,booktabs,graphicx,graphics,amsfonts}
\usepackage{latexsym,microtype,indentfirst,hyperref}
\usepackage{xcolor}
\usepackage{courier}

\newtheorem{theorem}{Theorem}[section]
\newtheorem{corollary}[theorem]{Corollary}

\newtheorem{lemma}[theorem]{Lemma}
\newtheorem{proposition}[theorem]{Proposition}

\theoremstyle{definition}
\newtheorem{definition}[theorem]{Definition}
\newtheorem{remark}[theorem]{Remark}
\newtheorem*{notation}{Notation}
\newtheorem{example}[theorem]{Example}
\newtheorem{assumption}[theorem]{Assumption}

\newtheorem{theoremletter}{Theorem}

\numberwithin{equation}{section}
\numberwithin{subsection}{section}

\newcommand{\Na}{\mathbb{N}} 
\newcommand{\Q}{\mathbb{Q}} 
\newcommand{\R}{\mathbb{R}} 
\newcommand{\Sf}{\mathbb{S}} 

\newcommand{\eps}{\varepsilon} 
\newcommand{\spt}{\mathrm{spt}} 
\newcommand{\dist}{\mathrm{dist}} 
\newcommand{\Lip}{\mathrm{Lip}} 
\newcommand{\Ha}{\mathcal{H}} 
\newcommand{\Leb}{\mathcal{L}} 

\newcommand{\pa}{\partial}
\newcommand{\clos}{{\rm clos}}
\newcommand{\Int}{{\rm int}\,}

\newcommand{\mres}{\mathbin{\vrule height 1.6ex depth 0pt width 
0.13ex\vrule height 0.13ex depth 0pt width 1.3ex}}

\newcommand{\abs}[1]{\lvert#1\rvert} 
\newcommand{\Abs}[1]{\left\lvert#1\right\rvert} 
\newcommand{\norm}[1]{\left\lVert#1\right\rVert} 

\newcommand{\V}{\mathbf{V}} 
\newcommand{\RV}{\mathbf{RV}}
\newcommand{\IV}{\mathbf{IV}} 
\newcommand{\var}{\mathbf{var}} 
\newcommand{\E}{\mathcal{E}} 
\newcommand{\op}{\mathcal{OP}}
\newcommand{\cA}{\mathcal{A}}
\newcommand{\cB}{\mathcal{B}}
\newcommand{\bE}{\mathbf{E}}
\renewcommand{\d}{\mathrm{d}} 
\newcommand{\bG}{\mathbf{G}} 
\newcommand{\cR}{\mathcal{R}}

\newcommand{\ssubset}{\subset\joinrel\subset}

\title[Brakke flow with fixed boundary conditions]{An existence theorem for Brakke flow \\ with fixed boundary conditions}

\date{\today}

\author{Salvatore Stuvard}
\address{Department of Mathematics, The University of Texas at Austin, 2515 Speedway, Stop C1200, Austin TX 78712-1202, United States of America}
\email{stuvard@math.utexas.edu}

\author{Yoshihiro Tonegawa}
\address{Department of Mathematics, Tokyo Institute of Technology, 2-12-1 Ookayama, Meguro-ku, Tokyo 152-8551, Japan}
\email{tonegawa@math.titech.ac.jp}

\begin{document}

\begin{abstract}

Consider an arbitrary closed, countably $n$-rectifiable set in a strictly convex $(n+1)$-dimensional domain, and suppose that the set has
finite $n$-dimensional Hausdorff measure and the complement
is not connected. Starting from this given set, we show that there exists a
non-trivial Brakke flow with fixed boundary data for all times. As $t \uparrow \infty$, the flow sequentially converges to non-trivial solutions of Plateau's problem in the setting of stationary varifolds. \\

\textsc{Keywords:} mean curvature flow, varifolds, Plateau's problem, minimal surfaces.\\

\textsc{AMS Math Subject Classification (2020):} 53E10 (primary), 49Q20, 49Q05.

\end{abstract}
\maketitle

\section{Introduction}
A time-parametrized family $\{\Gamma(t)\}_{t\geq 0}$ of $n$-dimensional surfaces in $\mathbb R^{n+1}$ (or in an open domain $U \subset \R^{n+1}$) is called 
a \emph{mean curvature flow} (abbreviated hereafter
as MCF) if the velocity of motion of $\Gamma(t)$ is equal to the mean curvature of $\Gamma(t)$ at each point and time. 
The aim of the present paper is to establish a global-in-time existence theorem for the 
MCF $\{\Gamma(t)\}_{t\geq 0}$ starting from a given surface $\Gamma_0$ while keeping the boundary of $\Gamma(t)$ fixed for all times $t \geq 0$. In particular, we are interested in the case when the initial surface $\Gamma_0$ is not smooth. Typical MCF under consideration in this setting may look like a moving network with multiple junctions for $n=1$, 
or a moving cluster of bubbles for $n=2$, and they may undergo various topological changes as they 
evolve. Due to the presence of singularities, we work in the framework of the generalized, measure-theoretic notion of MCF introduced by Brakke and since known as the Brakke flow \cite{Brakke,Ton1}. A global-in-time existence result for a Brakke flow \emph{without} fixed boundary conditions was established by Kim and the second-named author in 
\cite{KimTone} by reworking \cite{Brakke} thoroughly. The major challenge of the present work is to devise a modification to
the approximation scheme in \cite{KimTone} which preserves the boundary data. 

\smallskip

Though somewhat technical, in order to clarify the setting of the problem at this point, we state the assumptions on the initial surface $\Gamma_0$ and the domain $U$ hosting its evolution. Their validity will be assumed throughout the paper. 
\begin{assumption} \label{ass:main}
Integers $n\geq 1$ and $N\geq 2$ are fixed, and ${\rm clos}\,A$ denotes the topological closure of $A$ in $\mathbb R^{n+1}$. 
\begin{itemize}
\item[(A1)] $U \subset \R^{n+1}$ is a strictly convex bounded domain with boundary $\partial U$ of class $C^2$.
\smallskip
\item[(A2)] $\Gamma_0 \subset U$ is a relatively closed, countably $n$-rectifiable set with finite $n$-dimensional Hausdorff measure.
\smallskip
\item[(A3)] $E_{0,1},E_{0,2},\ldots,E_{0,N}$ are non-empty, open, and mutually disjoint subsets of $U$ such that $U\setminus\Gamma_0=\bigcup_{i=1}^N E_{0,i}$.
\smallskip
\item[(A4)] $\partial\Gamma_0 := ({\rm clos}\,\Gamma_0) \setminus U$
is not empty, and for each $x \in \partial\Gamma_0$ there exist at least two indexes $i_1 \ne i_2$ in $\{1,\ldots,N\}$ such that $x \in \clos\left(\clos(E_{0,i_j}) \setminus (U \cup \pa \Gamma_0)\right)$ for $j=1,2$.
\end{itemize}
\end{assumption}
\noindent
Since $N\geq 2$, we implicitly assume that $U\setminus \Gamma_0$ is not connected. 
When $n=1$, $\Gamma_0$ could be for instance a union of Lipschitz curves joined 
at junctions, with ``labels'' from $1$ to $N$ being assigned to each connected component of $U\setminus\Gamma_0$. 
If one defines $F_i:=({\rm clos}\,E_{0,i})\setminus(U\cup\partial
\Gamma_0)$ for $i=1,\ldots,N$, one can check that each $F_i$ is relatively open in $\partial U$,
$F_1,\ldots,F_N$ are mutually disjoint, and $\cup_{i=1}^N F_i=\partial U\setminus
\partial\Gamma_0$. The assumption (A4) is equivalent to the requirement that each $x\in \partial \Gamma_0$ is in
$\partial F_{i_1}\cap\partial F_{i_2}$ for some indices $i_1\neq i_2$. 
The main result of the present paper can then be roughly stated as follows.
\begin{theoremletter}
Under the assumptions (A1)-(A4), there exists a MCF $\{\Gamma(t)\}_{t\geq 0}$ such that
\[
\Gamma(0) = \Gamma_0\,, \qquad \mbox{and} \qquad \pa \Gamma(t) := ({\rm clos}\, \Gamma(t)) \setminus U = \pa \Gamma_0 \quad \mbox{for all $t\geq 0$}\,.
\]
For all $t>0$, $\Gamma(t)$ remains within the convex hull of 
$\Gamma_0\cup\partial\Gamma_0$. 
\end{theoremletter}
\noindent
More precisely, $\{\Gamma(t)\}_{t\geq 0}$ is a MCF in the sense that 
$\Gamma(t)$ coincides with the slice, at time $t$, of the space-time support of a Brakke flow $\{V_t\}_{t\geq 0}$ starting from $\Gamma_0$. The method adopted to produce the evolving generalized surfaces $\Gamma(t)$ actually gives us more. Indeed, we show the existence of $N$ families $\{E_i(t)\}_{t \geq 0}$ ($i = 1,\ldots,N$) of evolving open sets such that $E_i(0)=E_{0,i}$ for every $i$, and $\Gamma(t)=U\setminus\cup_{i=1}^N
E_i(t)$ for all $t \geq 0$. At each time $t\geq 0$, the sets $E_1(t),\ldots,E_N(t)$ are mutually disjoint and form 
a partition of $U$. Moreover, for each fixed $i$ the Lebesgue measure of $E_i(t)$ is a continuous function of time, so that the evolving $\Gamma(t)$ do not exhibit arbitrary instantaneous loss of mass. See Theorems \ref{thm:main} and \ref{thm:main2} for the full statement. 

\smallskip

It is reasonable to expect that the flow $\Gamma(t)$ converges, as 
$t\rightarrow\infty$, to 
a minimal surface in $U$ with boundary $\pa \Gamma_0$. We are not able to prove such a result in full generality; nonetheless, we can show the following
\begin{theoremletter}
There exists a sequence of times $\{t_k\}_{k=1}^\infty$ with $\lim_{k\to\infty} t_k = \infty$ such that the corresponding varifolds $V_k := V_{t_k}$ converge to a \emph{stationary} integral varifold $V_\infty$ in $U$ such that $({\rm clos}\,(\spt\|V_\infty\|)) \setminus U = \pa \Gamma_0$.
\end{theoremletter}
See Corollary \ref{main:cor} for a precise statement. The limit $V_{\infty}$ is a solution to Plateau's problem with boundary $\pa \Gamma_0$, in the sense that it has the prescribed boundary in the topological sense specified above and it is minimal in the sense of varifolds. We warn the reader that $V_{\infty}$ may not be area-minimizing. Furthermore, the flow may converge to different limit varifolds along different diverging sequences of times in all cases when uniqueness of a minimal surface with the prescribed boundary is not guaranteed. The possibility to use Brakke flow in order to select solutions to Plateau's problem in classes of varifolds seems an interesting byproduct of our theory. See Section \ref{propla} for further discussion on these points.

\smallskip  

Next, we discuss closely related results. 
While there are several works on the global-in-time existence of MCF, there are relatively 
few results on the existence of MCF with fixed boundary conditions. When $\Gamma_0$ is a smooth graph over a bounded domain $\Omega$ in $\R^n$, global-in-time existence follows from the classical work of Lieberman \cite{Lieb}. Furthermore, under the assumption that $\Omega$ is mean convex, convergence of the flow to the unique solution to the minimal surfaces equation in $\Omega$ with the prescribed boundary was established by Huisken in \cite{Hu1}; see also the subsequent generalizations to the Riemannian setting in \cite{Priw,spruck}. The case of network flows with fixed endpoints and a single triple junction was extensively
studied in \cite{MNT,MMN}. For other configurations and related works on the network flows,
see the survey paper \cite{MNPS} and references therein. In the case when $N=2$ (which
does not allow triple junctions in general), a 
powerful approach is the level set method \cite{CGG,ES}. Existence and uniqueness in this setting were established in \cite{SZ1}, and the
asymptotic limit as $t\rightarrow\infty$ was studied in \cite{SZ2}. Recently, White \cite{Wh1} proved the existence of a Brakke flow with prescribed smooth boundary in the sense of integral flat chains ${\rm mod}(2)$. The proof uses the elliptic
regularization scheme discovered by Ilmanen \cite{Ilm1}, which allows one to obtain a Brakke flow with additional good regularity and compactness properties; see also \cite{SW} for an application of elliptic regularization within the framework of flat chains with coefficients in suitable finite groups to the long-time existence and short-time regularity of unconstrained MCF starting from a general surface cluster. Observe that the homological constraint used by White prevents the flow to develop interior junction-type singularities of odd order (namely, junctions which are locally diffeomorphic to the union of an odd number of half-hyperplanes), because these singularities are necessarily boundary points ${\rm mod}(2)$. As a consequence, the flows obtained in \cite{Wh1} may differ greatly from those produced in the present paper. This is not surprising, as solutions to Brakke flow may be highly non-unique. A complete characterization of the topological changes that the evolving surfaces can undergo with either of the two approaches is, in fact, an interesting open question. It is worth noticing that analogous generic non-uniqueness holds true also for Plateau's problem: in that context, different definitions of the key words \emph{surfaces, area, spanning} in its formulation lead to solutions with dramatically different regularity properties, thus making each model a better or worse predictor of the geometric complexity of physical soap films; see e.g. the survey papers \cite{David_Plateau,HP_Plateau} and the references therein, as well as the more recent works \cite{DGM,MSS,KMS,KMS2,KMS3,DLHMS_linear,DLHMS_nonlinear}. It is then interesting and natural to investigate different formulations for Brakke flow as well.

\medskip

{\bf Acknowledgments.} The work of S.S. was supported by the NSF grants DMS-1565354, DMS-RTG-1840314 and DMS-FRG-1854344. Y.T. was partially supported by JSPS Grant-in-aid for scientific research 18H03670, 19H00639 and 17H01092.

\section{Definitions, Notation, and Main Results}

\subsection{Basic notation}
The ambient space we will be working in is Euclidean space $\R^{n+1}$. We write $\R^+$ for $[0,\infty)$.
For $A\subset\mathbb R^{n+1}$, ${\rm clos}\,A$ (or $\overline A$) is the topological closure of $A$ in 
$\mathbb R^{n+1}$ (and not in $U$), ${\rm int}\,A$ is the set of interior points of $A$ 
and ${\rm conv}\,A$ is the convex hull of $A$. The standard Euclidean inner product between vectors in $\R^{n+1}$ is denoted $x \cdot y$, and $\abs{x} := \sqrt{x \cdot x}$. If $L,S \in \mathscr{L}(\R^{n+1};\R^{n+1})$ are linear operators in $\R^{n+1}$, their (Hilbert-Schmidt) inner product is $L \cdot S := {\rm trace}(L^T \circ S)$, where $L^T$ is the transpose of $L$ and $\circ$ denotes composition. The corresponding (Euclidean) norm in $\mathscr{L}(\R^{n+1};\R^{n+1})$ is then $\abs{L} := \sqrt{L \cdot L}$, whereas the operator norm in $\mathscr{L}(\R^{n+1};\R^{n+1})$ is $\|L\| := \sup\left\lbrace \abs{L(x)} \, \colon \, \mbox{$x\in\R^{n+1}$ with $\abs{x}\leq 1$}  \right\rbrace$. If $u,v \in \R^{n+1}$ then $u \otimes v \in \mathscr{L}(\R^{n+1};\R^{n+1})$ is defined by $(u \otimes v)(x) := (x \cdot v)\, u$, so that $\| u \otimes v \| = \abs{u}\,\abs{v}$. The symbol $U_{r}(x)$ (resp.~$B_r(x)$) denotes the open (resp.~closed) ball in $\R^{n+1}$ centered at $x$ and having radius $r > 0$. The Lebesgue measure of a set $A \subset \R^{n+1}$ is denoted $\Leb^{n+1}(A)$ or $|A|$.  If $1 \leq k \leq n+1$ is an integer, $U_r^k(x)$ denotes the open ball with center $x$ and radius $r$ in $\R^k$. We will set $\omega_k := \Leb^k(U_1^k(0))$. The symbol $\Ha^k$ denotes the $k$-dimensional Hausdorff measure in $\R^{n+1}$, so that $\Ha^{n+1}$ and $\Leb^{n+1}$ coincide as measures. \\

\smallskip

A Radon measure $\mu$ in $U\subset\mathbb R^{n+1}$ is always also regarded as a linear functional on the space $C_c(U)$ of continuous and compactly supported functions on $U$, with the pairing denoted $\mu(\phi)$ for $\phi \in C_c(U)$. The restriction of $\mu$ to a Borel set $A$ is denoted $\mu\, \mres_A$, so that $(\mu \,\mres_A)(E) := \mu(A \cap E)$ for any $E \subset U$. The support of $\mu$ is denoted $\spt\,\mu$, and it is the relatively closed subset of $U$ defined by
\[
\spt\,\mu := \left\lbrace x \in U \, \colon \, \mu(B_r(x)) > 0 \mbox{ for every $r > 0$} \right\rbrace\,.
\]
The upper and lower $k$-dimensional densities of a Radon measure $\mu$ at $x \in U$ are
\[
\theta^{*k}(\mu,x) := \limsup_{r \to 0^+} \frac{\mu(B_r(x))}{\omega_k\, r^k} \,, \qquad \theta^k_*(\mu,x) := \liminf_{r \to 0^+} \frac{\mu(B_r(x))}{\omega_k\, r^k}\,,
\]
respectively. If $\theta^{*k}(\mu,x) = \theta^k_*(\mu,x)$ then the common value is denoted $\theta^k(\mu,x)$, and is called the {\it $k$-dimensional density} of $\mu$ at $x$. For $1 \leq p \leq \infty$, the space of $p$-integrable (resp.~locally $p$-integrable) functions with respect to $\mu$ is denoted $L^p(\mu)$ (resp.~$L^p_{loc}(\mu)$). For a set $E \subset U$, $\chi_E$ is the characteristic function of $E$. If $E$ is a set of finite perimeter in $U$, then $\nabla \chi_E$ is the associated Gauss-Green measure in $U$, and its total variation $\|\nabla \chi_E\|$ in $U$ is the perimeter measure; by De Giorgi's structure theorem, $\| \nabla \chi_E\| = \Ha^n \mres_{\pa^* E}$, where $\pa^* E$ is the reduced boundary of $E$ in $U$. 
\subsection{Varifolds}
The symbol $\bG(n+1,k)$ will denote the Grassmannian of (unoriented) $k$-dimensional linear planes in $\R^{n+1}$. Given $S \in \bG(n+1,k)$, we shall often identify $S$ with the orthogonal projection operator onto it. The symbol $\V_k(U)$ will denote the space of $k$-dimensional {\it varifolds} in $U$, namely the space of Radon measures on $\bG_k(U) := U \times \bG(n+1,k)$ (see \cite{Allard,Simon} for a comprehensive treatment of varifolds). To any given $V \in \V_k(U)$ one associates a Radon measure $\|V\|$ on $U$, called the {\it weight} of $V$, and defined by projecting $V$ onto the first factor in $\bG_k(U)$, explicitly:
\[
\|V\|(\phi) := \int_{\bG_k(U)} \phi(x) \, dV(x,S) \qquad \mbox{for every $\phi \in C_c(U)$}\,.
\] 
A set $\Gamma \subset \R^{n+1}$ is {\it countably $k$-rectifiable} if it can be covered by countably many Lipschitz images of $\R^k$ into $\R^{n+1}$ up to a $\Ha^k$-negligible set. We say that $\Gamma$ is (locally) {\it $\Ha^k$-rectifiable} if it is $\Ha^k$-measurable, countably $k$-rectifiable, and $\Ha^k(\Gamma)$ is (locally) finite. If $\Gamma \subset U$ is locally $\Ha^k$-rectifiable, and $\theta \in L^{1}_{loc}(\Ha^k \mres_\Gamma)$ is a positive function on $\Gamma$, then there is a $k$-varifold canonically associated to the pair $(\Gamma,\theta)$, namely the varifold $\var(\Gamma,\theta)$ defined by
\begin{equation} \label{varGammatheta}
\var(\Gamma,\theta)(\varphi) := \int_\Gamma \varphi(x, T_x \Gamma) \, \theta(x)\, d\Ha^k(x) \qquad \mbox{for every } \varphi \in C_c(\bG_k(U))\,,
\end{equation}
where $T_x\Gamma$ denotes the approximate tangent plane to $\Gamma$ at $x$, which exists $\Ha^k$-a.e. on $\Gamma$. Any varifold $V \in \V_k(U)$ admitting a representation as in \eqref{varGammatheta} is said to be \emph{rectifiable}, and the space of rectifiable $k$-varifolds in $U$ is denoted by 
${\bf RV}_k(U)$. If $V = \var(\Gamma,\theta)$ is rectifiable and $\theta(x)$ is an integer at $\Ha^k$-a.e. $x \in \Gamma$, then we say that $V$ is an \emph{integral} $k$-dimensional varifold in $U$: the corresponding space is denoted $\IV_k(U)$. 

\subsection{First variation of a varifold}
If $V \in \V_k(U)$ and $f \colon U \to U'$ is $C^1$ and proper, then we let $f_\sharp V \in \V_k(U')$ denote the push-forward of $V$ through $f$. Recall that the weight of $f_\sharp V$ is given by
\begin{equation}\label{pushfd}
\|f_\sharp V\|(\phi) = \int_{\bG_{k}(U)} \phi \circ f(x) \, \abs{\Lambda_k \nabla f(x) \circ S} \, dV(x,S) \qquad \mbox{for every }\, \phi \in C_{c}(U')\,,
\end{equation}
where
\[
\abs{\Lambda_k \nabla f(x) \circ S} := \abs{\nabla f(x) \cdot v_1 \, \wedge \, \ldots \,\wedge\, \nabla f(x) \cdot v_k} \quad \mbox{for any orthonormal basis $\{ v_1, \ldots, v_k \}$ of $S$}
\]
is the Jacobian of $f$ along $S \in \bG(n+1,k)$.
Given a varifold $V \in \V_k(U)$ and a vector field $g \in C^1_c(U; \R^{n+1})$, the {\it first variation} of $V$ in the direction of $g$ is the quantity
\begin{equation}
\label{defFV}
\delta V(g) := \left.\frac{d}{dt}\right|_{t=0} \|(\Phi_t)_\sharp V\|(\tilde U)\,,
\end{equation}
where $\Phi_t(\cdot) = \Phi(t,\cdot)$ is any one-parameter family of diffeomorphisms of $U$ defined for sufficiently small $|t|$ such that $\Phi_0 = {\rm id}_U$ and $\pa_t \Phi(0,\cdot) = g(\cdot)$. The $\tilde U$ is chosen so that ${\rm clos}\,\tilde U\subset U$ is compact and ${\rm spt}\,g\subset \tilde U$, and the definition of \eqref{defFV} 
does not depend on the choice of $\tilde U$. 
It is well known that $\delta V$ is a linear and continuous functional on $C^1_c(U; \R^{n+1})$, and in fact that
\begin{equation}
\label{defFV1}
\delta V(g) = \int_{\bG_k(U)} \nabla g(x) \cdot S \, dV(x,S) \qquad \mbox{for every $g \in C^1_c(U;\R^{n+1})$}\,,
\end{equation}
where, after identifying $S \in \bG(n+1,k)$ with the orthogonal projection operator $\R^{n+1} \to S$,
\[
\nabla g \cdot S = {\rm trace}(\nabla g^T \circ S) = \sum_{i,j=1}^{n+1} S_{ij} \, \frac{\partial g_i}{\partial x_j} = {\rm div}^S g\,.
\]
If $\delta V$ can be extended to a linear and continuous functional on
$C_c(U;\R^{n+1})$, we say that $V$ has {\it bounded first variation} in $U$. In this case, 
$\delta V$ is naturally associated with 
a unique $\R^{n+1}$-valued measure on $U$ by means of the Riesz representation theorem.
If such a measure is absolutely continuous with respect to the weight $\|V\|$, then there exists a $\|V\|$-measurable and locally $\|V\|$-integrable vector
field $h(\cdot,V)$ such that 
\begin{equation} \label{def:generalized mean curvature}
\delta V(g) = - \int_{U} g(x) \cdot h(x,V) \, d\|V\|(x) \qquad \mbox{for every $g \in C_c(U,\R^{n+1})$}
\end{equation}
by the Lebesgue-Radon-Nikod\'ym differentiation theorem. The vector field $h(\cdot,V)$ is called the {\it generalized mean curvature vector} of $V$. In particular, if $\delta V(g)=0$ 
for all $g\in C_c^1(U;\mathbb R^{n+1})$, $V$ is called {\it stationary}, and this is 
equivalent to $h(\cdot,V)=0$ $\|V\|$-almost everywhere. For any $V\in {\bf IV}_k(U)$ with 
bounded first variation, {\it Brakke's perpendicularity theorem} \cite[Chapter 5]{Brakke}
says that 
\begin{equation}
\label{BPT}
S^{\perp}(h(x,V))=h(x,V) \qquad \mbox{for $V$-a.e. $(x,S) \in {\bf G}_k(U)$}\,.
\end{equation}
Here, $S^{\perp}$ is the projection onto the orthogonal complement of $S$ in $\R^{n+1}$. 
This means that the generalized mean curvature vector is perpendicular to the 
approximate tangent plane almost everywhere. 

\smallskip

Other than the first variation $\delta V$ discussed
above, we shall also use a {\it weighted first variation}, defined as follows. For given $\phi\in C^1_c(U;\mathbb R^+)$, $V\in {\bf V}_k(U)$, and $g \in C^1_c(U;\R^{n+1})$, we modify \eqref{defFV} to introduce the $\phi$-weighted first variation of $V$ in the direction of $g$, denoted $\delta(V,\phi)(g)$, by setting
\begin{equation} \label{defFV_modified}
\delta(V,\phi)(g) := \left.\frac{d}{dt}\right|_{t=0} \| (\Phi_t)_\sharp V \|(\phi)\,,
\end{equation}
where $\Phi_t$ denotes the one-parameter family of diffeomorphisms of $U$ induced by $g$ as above. Proceeding as in the derivation of \eqref{defFV1}, one then obtains the expression
\begin{equation}
\label{defFV2}
\delta(V,\phi)(g)=\int_{{\bf G}_k(U)} \phi(x)\, \nabla g(x)\cdot S\,dV(x,S)+
\int_U g(x)\cdot\nabla\phi(x)\,d\|V\|(x)\,.
\end{equation}
Using $\phi\nabla g=
\nabla(\phi g)- g\otimes\nabla\phi$ in \eqref{defFV2} and \eqref{defFV1}, we obtain
\begin{equation}
\label{defFV3}
\begin{split}
\delta(V,\phi)(g)&=\delta V(\phi g)+\int_{{\bf G}_k(U)} g(x)\cdot(\nabla\phi(x)-S(
\nabla\phi(x)))\,dV(x,S) \\
&=\delta V(\phi g)+\int_{{\bf G}_k(U)} g(x)\cdot S^{\perp}(\nabla\phi(x))\,dV(x,S)\,. 
\end{split}
\end{equation}
If $\delta V$ has generalized mean curvature $h(\cdot,V)$, then we may use \eqref{def:generalized mean curvature} in \eqref{defFV3} to obtain
\begin{equation}
\label{defFV4}
\delta(V,\phi)(g)=-\int_U \phi(x)g(x)\cdot h(x,V) \, d\|V\|(x)+\int_{{\bf G}_k(U)} g(x)\cdot S^{\perp}
(\nabla\phi(x))\,
dV(x,S).
\end{equation}

The definition of Brakke flow requires considering weighted first variations in the direction of the mean curvature. Suppose $V\in {\bf IV}_k(U)$, $\delta V$ is locally bounded and absolutely continuous
with respect to $\|V\|$ and $h(\cdot,V)$ is locally square-integrable with respect to $\|V\|$. 
In this case, it is natural from the expression \eqref{defFV4} to define for $\phi\in C_c^1
(U;\mathbb R^+)$
\begin{equation}
\label{defFV5}
\delta(V,\phi)(h(\cdot,V)):=\int_U \lbrace -\phi(x)|h(x,V)|^2+h(x,V)\cdot\nabla\phi(x) \rbrace\,d\|V\|(x).
\end{equation}
Observe that here we have used \eqref{BPT} in order to replace the term $h(x,V)\cdot S^{\perp}(\nabla\phi(x))$ with $h(x,V)\cdot \nabla\phi(x)$.

\subsection{Brakke flow}
To motivate a weak formulation of the MCF, note that a smooth family of $k$-dimensional
surfaces $\{\Gamma(t)\}_{t\geq 0}$ in $U$ is a MCF if and only if the following inequality
holds true for all $\phi = \phi (x,t)\in C_c^1(U\times[0,\infty);\mathbb R^+)$:
\begin{equation}
\label{smMCF1}
\frac{d}{dt}\int_{\Gamma(t)}\phi\,d\mathcal H^k \leq \int_{\Gamma(t)}  
\left\lbrace -\phi\,|h(\cdot,\Gamma(t))|^2+\nabla\phi\cdot h(\cdot,\Gamma(t))
+\frac{\partial\phi}{\partial t} \right\rbrace \,d\mathcal H^k \,.
\end{equation}
In fact, the ``only if''
part holds with equality in place of inequality. For a more comprehensive treatment of the Brakke flow, 
see \cite[Chapter 2]{Ton1}. Formally, if $\partial\Gamma(t)\subset
\partial U$ is fixed in time, with $\phi=1$, we also obtain
\begin{equation}
\label{smMCF2}
\frac{d}{dt}\mathcal H^k(\Gamma(t))  \leq -\int_{\Gamma(t)}|h(x,\Gamma(t))|^2\,d\mathcal H^k(x)\,,
\end{equation}
which states the well-known fact that the $L^2$-norm of the mean curvature represents the dissipation of area along the MCF. Motivated by \eqref{smMCF1} and \eqref{smMCF2}, and for the purposes of this paper, we give the following definition.
\begin{definition} \label{def:Brakke_bc}
We say that a family of varifolds $\{V_t\}_{t\geq 0}$ in $U$ is a {\it Brakke flow with fixed
boundary} $\Sigma\subset\partial U$ if all of the following hold: 
\begin{enumerate}
\item[(a)]
For a.e.~$t\geq 0$, $V_t\in{\bf IV}_k(U)$;
\item[(b)]
For a.e.~$t\geq 0$, $\delta V_t$ is bounded and absolutely continuous with respect to $\|V_t\|$;
\item[(c)] The generalized mean curvature $h(x,V_t)$ (which exists for a.e.~$t$ by (b)) satisfies for all $T>0$
\begin{equation}
\|V_T\|(U)+\int_0^{T}dt\int_U|h(x,V_t)|^2\,d\|V_t\|(x)\leq \|V_0\|(U);
\label{brakineq2}
\end{equation}
\item[(d)]
For all $0\leq t_1<t_2<\infty$ and $\phi\in C_c^1(U\times\R^+;\mathbb R^+)$,
\begin{equation}
\label{brakineq}
\|V_{t}\|(\phi(\cdot,t))\Big|_{t=t_1}^{t_2}\leq \int_{t_1}^{t_2}\delta(V_t,\phi(\cdot,t))(h(\cdot,V_t))+\|V_t\|\big(\frac{\partial\phi}{\partial t}(\cdot,t)\big)\,dt\,,
\end{equation}
having set $\|V_{t}\|(\phi(\cdot,t))\Big|_{t=t_1}^{t_2} := \|V_{t_2}\|(\phi(\cdot,t_2)) - \|V_{t_1}\|(\phi(\cdot,t_1))$;
\item[(e)]
For all $t\geq 0$, $ ({\rm clos}\,({\rm spt}\,\|V_t\|))\setminus U=\Sigma$.
\end{enumerate}
\end{definition}
\noindent
In this paper, we are interested in the $n$-dimensional Brakke flow in particular. 
Formally, by integrating \eqref{smMCF2} from $0$ to $T$, we obtain the analogue of 
\eqref{brakineq2}. By integrating \eqref{smMCF1} from $t_1$ to $t_2$, we also obtain 
the analogue of \eqref{brakineq} via the expression \eqref{defFV5}. We recall that the
closure is taken with respect to the topology of $\mathbb R^{n+1}$ while the support 
of $\|V_t\|$ is in $U$. Thus (e) geometrically
means that ``the boundary of $V_t$ (or $\|V_t\|$) is $\Sigma$''.

\subsection{Main results}  \label{sec:main}

The main existence theorem of a Brakke flow with fixed boundary is the following.
\begin{theorem} \label{thm:main}
Suppose that $U,\Gamma_0$, and $E_{0,1},\ldots,E_{0,N}$ satisfy 
Assumption \ref{ass:main} (A1)-(A4). Then, there exists a Brakke flow $\{V_t\}_{t\geq 0}$
with fixed boundary $\partial\Gamma_0$, and $\|V_0\|=\mathcal H^n\mres_{\Gamma_0}$.
If $\mathcal H^n(\Gamma_0\setminus \cup_{i=1}^N \partial^* E_{0,i})=0$, 
we have $\lim_{t\downarrow 0}\|V_t\|=\mathcal H^n\mres_{\Gamma_0}$. 
\end{theorem}
\noindent
Since we are assuming that $\partial\Gamma_0\neq \emptyset$, we have $V_t\neq 0$ for all $t>0$.
If the union of the \emph{reduced boundaries} of the initial partition in $U$ coincides with $\Gamma_0$ modulo $\Ha^n$-negligible sets (note that the assumptions $(A2)$ and $(A3)$ in Assumption \ref{ass:main} imply that $\Gamma_0 = U\cap\bigcup_{i=1}^N \pa E_{0,i}$), then
the claim is that
the initial condition is satisfied continuously as measures. Otherwise, an instantaneous loss of measure may 
occur at $t=0$. As far as the regularity is concerned, under the additional
assumption that $\{V_t\}_{t > 0}$ is a unit density flow, partial regularity theorems of \cite{Brakke,Kasai-Tone,Ton-2}
show that $V_t$ is a smooth MCF for a.e.~time
and a.e.~point in space, just like \cite{KimTone}, see \cite[Theorem 3.6]{KimTone} for the precise statement. No claim of the uniqueness is made here, but the next Theorem \ref{thm:main2}
gives an additional structure to $V_t$ in the form of ``moving partitions''
starting from $E_{0,1},\ldots,E_{0,N}$. 
\begin{theorem} \label{thm:main2}
Under the same assumption of Theorem \ref{thm:main} and in addition to $\{V_t\}_{t\geq 0}$, for each $i=1,\dots,N$ there exists a one-parameter family $\{E_i(t)\}_{t \geq 0}$ of open sets $E_{i}(t) \subset U$ with the following properties. Let $\Gamma(t):=U\setminus\cup_{i=1}^N E_i(t)$.
\begin{enumerate}

\item $E_{i}(0) = E_{0,i}$ $\forall i=1,\dots,N$;

\item $\forall t \geq 0$, the sets $\{E_i(t)\}_{i=1}^N$ are mutually disjoint; 

\item $\forall\tilde U\ssubset U$ and $\forall t\geq 0$, $\mathcal H^n(\Gamma(t)\cap \tilde U)<\infty$;

\item $\forall t\geq 0$, $\Gamma(t)=U\cap \cup_{i=1}^N \partial( E_i(t))$;

\item $\forall t\geq 0$, $\Gamma(t)\subset {\rm conv}(\Gamma_0\cup\partial\Gamma_0)$;

\item $\forall t\geq 0$ and $\forall i=1,\ldots,N$, $E_i(t)\setminus {\rm conv}(\Gamma_0
\cup\partial\Gamma_0)=E_{0,i}\setminus {\rm conv}(\Gamma_0
\cup\partial\Gamma_0)$;

\item $\forall t\geq 0$, $\partial\Gamma(t):=({\rm clos}\,\Gamma(t))\setminus U=\partial\Gamma_0$;

\item $\forall t\geq 0$ and $\forall i=1,\ldots,N$, $\|\nabla\chi_{E_i(t)}\| \leq \|V_t\|$
and $\sum_{i=1}^N\|\nabla\chi_{E_i(t)}\|\leq 2\|V_t\|$;
 
 \item Fix $i=1,\ldots,N$ and $U_r(x)\ssubset U$, and define $g(t):=
 \mathcal L^{n+1}(U_r(x)\cap E_i(t))$. Then, $g\in C^0([0,\infty))\cap C^{0,\frac12}((0,\infty))$;
 
\item For each $i=1,\ldots,N$, $\chi_{E_i(t)}\in C([0,\infty);L^1(U))$; 

\item Let $\mu$ be the product measure of $\|V_t\|$ and $dt$ defined on $U\times\R^+$, i.e. $d\mu:=d\|V_t\|dt$. Then, $\forall t>0$, we have 
\begin{equation*}
{\rm spt}\,\|V_t\|\subset 
\{x\in U\,:\, (x,t)\in {\rm spt}\,\mu\}=\Gamma(t).
\end{equation*}

\end{enumerate} 

\end{theorem}
\noindent
The claims (1)-(4) imply that $\{E_i(t)\}_{i=1}^N$ is an $\mathcal{L}^{n+1}$-partition of $U$, and that $\Gamma(t)$
has empty interior in particular. The claim (5) is an expected property for the MCF, and, by (11), ${\rm spt}\,\|V_t\|$ is also in the same convex hull. (7) says that $\Gamma(t)$ has the fixed
boundary $\partial\Gamma_0$.
In general, the reduced boundary of the partition and $\|V_t\|$ may not match, but the 
latter is bounded from below by the former as in (8). By (10), the Lebesgue measure of each 
$E_i(t)$ changes continuously in time, so that arbitrary sudden 
loss of measure of $\|V_t\|$ is not allowed. 
The statement in (11) says that the time-slice of the support of $\mu$ at time $t$ contains the support of $\|V_t\|$ and is equal to the topological boundary of the moving
partition. \\

As a corollary of the above, we deduce the following.

\begin{corollary}\label{main:cor}
There exist a sequence $\{t_k\}_{k=1}^\infty$ with $\lim_{k\rightarrow\infty} t_k=\infty$ and a varifold $V \in \IV_n(U)$ such that $V_{t_k} \to V$ in the sense of varifolds. The varifold $V$ is stationary. Furthermore, there is a mutually disjoint family $\{E_i\}_{i=1}^N$ of open subsets of $U$ such that 
\begin{enumerate}
\item $\forall i=1,\ldots,N$, $ \| \nabla \chi_{E_i} \|  \leq \|V\|$ and $\sum_{i=1}^N
\|\nabla\chi_{E_i}\|  \leq 2\|V\|$;
\item $\forall i=1,\ldots,N$, $E_i \setminus {\rm conv (\Gamma_0\cup\pa\Gamma_0)} = E_{0,i} \setminus {\rm conv}(\Gamma_0\cup\pa\Gamma_0)$;
\item $U \setminus \bigcup_{i=1}^N E_i = \spt\|V\|$, and $0 < \Ha^n (U \setminus \bigcup_{i=1}^N E_i) \leq \|V\|(U) \leq \Ha^n(\Gamma_0)$;
\item $({\rm clos}\,(\spt\|V\|))\setminus U= (\clos(U \setminus \bigcup_{i=1}^N E_i)) \setminus U = \partial\Gamma_0$.
\end{enumerate}
\end{corollary}

The varifold $V$ in Corollary \ref{main:cor} is a solution to Plateau's problem in $U$ in the class of stationary varifolds satisfying the topological constraint $({\rm clos}\,(\spt \|V\|)) \setminus U = \pa \Gamma_0$. This is an interesting byproduct of our construction, above all considering that $\pa \Gamma_0$ enjoys in general rather poor regularity (in particular, it may have infinite $(n-1)$-dimensional Hausdorff measure, and also it may not be countably $(n-1)$-rectifiable). Even though the \emph{topological} boundary condition specified above seems natural in this setting, other notions of spanning may be adopted: for instance, in Proposition \ref{final spanning} we show that a \emph{strong homotopic spanning condition} in the sense of \cite{HP16,DGM} is preserved along the flow and in the limit if it is satisfied at the initial time $t=0$. We postpone further discussion and questions concerning the application to Plateau's problem to Section \ref{propla}.  

\subsection{General strategy and structure of the paper}

The general idea behind the proof of Theorems \ref{thm:main} and \ref{thm:main2}
 is to suitably modify the time-discrete approximation scheme introduced in \cite{KimTone,Brakke}. There, one constructs a time-parametrized flow of open partitions which is piecewise constant in time. We will call \emph{epoch} any time interval during which the approximating flow is constant. The open partition at a given epoch is constructed from the open partition at the previous epoch by applying two operations, which we call \emph{steps}. The first step is a small
Lipschitz deformation of partitions with the effect of
``regularizing singularities'' by ``locally minimizing the area of the boundary of partitions'' at a small scale. This deformation is defined in such a way that, if the boundary of partitions is regular (relative to a certain length scale), then
the deformation reduces to the identity.
The second step consists of flowing the boundary of partitions by a suitably defined ``approximate
mean curvature vector''. The latter is computed by smoothing the surface measures via convolution with a localized heat kernel. Note that, typically, the boundary of open partitions has bounded $n$-dimensional measure, but the unit-density varifold associated to it may not have bounded first variation. In \cite{KimTone}, a time-discrete approximate MCF is obtained by alternating these two steps, epoch after epoch. In the present work, we 
need to fix the boundary $\partial\Gamma_0$. The rough idea to achieve this is to perform an ``exponentially 
small''
truncation of the approximate mean curvature vector near $\partial\Gamma_0$, so that
the boundary cannot move in the ``polynomial time scale'' defining an epoch with respect to a certain length
scale. We also need to make sure that the time-discrete movement does not push the 
boundary of open partitions to the outside of $U$. To prevent this, in addition to
the two steps (Lipschitz deformation and motion by smoothed and truncated mean curvature
vector), we add another ``retraction to $U$'' step to be performed in each epoch. All these 
operations have to come with suitable estimates on the surface measures, in order to have convergence of the approximating flow when we let the epoch time scale approach zero. The final goal is to show that this limit flow is indeed a Brakke flow with fixed boundary $\pa\Gamma_0$ as in Definition \ref{def:Brakke_bc}.

\smallskip

The rest of the paper is organized as follows. Section \ref{sec:prelim} lays the foundations to the technical construction of the approximate flow by proving the relevant estimates to be used in the Lipschitz deformation and flow by smoothed mean curvature steps, and by defining the boundary truncation of the mean curvature. Both the discrete approximate flow and its ``vanishing epoch'' limit are constructed in Section \ref{sec:limit flow}. In Section \ref{sec:Brakke} we show that the one-parameter family of measures obtained in the previous section satisfies conditions (a) to (d) in Definition \ref{def:Brakke_bc}. The boundary condition (e) is, instead, proved in Section \ref{sec:bb}, which therefore also contains the proofs of Theorems \ref{thm:main} and \ref{thm:main2}. Finally, Section \ref{propla} is dedicated to the limit $t \to \infty$: hence, it contains the proof of Corollary \ref{main:cor}, as well as a discussion of related results and open questions concerning the application of our construction to Plateau's problem.

\section{Preliminaries} \label{sec:prelim}

In this section we will collect the preliminary results that will play a pivotal role in the construction of the time-discrete approximate flows. Some of the results are straightforward adaptations of the corresponding ones in \cite{KimTone}: when that is the case, we shall omit the proofs, and refer the reader to that paper. 

\subsection{Classes of test functions and vector fields}

Define, for every $j \in \Na$, the classes $\cA_j$ and $\cB_j$ as follows:

\begin{equation} \label{classA}
\begin{split}
\cA_j := \{ \phi \in C^2(\R^{n+1}; \R^+) \, \colon \, &\phi(x) \leq 1,\; \abs{\nabla \phi(x)} \leq j\, \phi(x), \\ &\|\nabla^2\phi(x) \| \leq j \, \phi(x)\,\, \mbox{for every $x \in \R^{n+1}$} \}\,,
\end{split}
\end{equation}

\begin{equation} \label{classB}
\begin{split}
\cB_j := \{ g \in C^2(\R^{n+1}; \R^{n+1}) \, \colon \, &|g(x)| \leq j,\,\, \norm{\nabla g(x)} \leq j\, , \\ &\|\nabla^2 g(x) \| \leq j \, \, \mbox{for every $x \in \R^{n+1}$ and $\|g\|_{L^2} \leq j$} \}\,.
\end{split}
\end{equation}

The properties of functions $\phi \in \cA_j$ and vector fields $g \in \cB_j$ are precisely as in \cite[Lemma 4.6, Lemma 4.7]{KimTone}, and we record them in the following lemma for future reference.

\begin{lemma} \label{l:class properties}

Let $x,y \in \R^{n+1}$ and $j \in \Na$. For every $\phi \in \cA_j$, the following properties hold:
\begin{align}
\phi(x) & \leq \phi(y) \exp(j\, \abs{x-y})\,, \label{e:Gronwall} \\
\abs{\phi(x) - \phi(y)} &\leq j \, \abs{x-y} \phi(y) \exp(j\, \abs{x-y})\,, \label{e:1st_order} \\
\abs{\phi(x) - \phi(y) - \nabla\phi(y) \cdot (x-y)} &\leq j\, \abs{x-y}^2 \phi(y) \exp(j\, \abs{x-y}) \label{e:2nd_order}\,.
\end{align}

Also, for every $g \in \cB_j$:
\begin{equation} \label{e:vectorfield}
\abs{g(x) - g(y)} \leq j\, \abs{x-y}\,.
\end{equation}

\end{lemma}

\subsection{Open partitions and admissible functions}

Let $\tilde U\subset \R^{n+1}$ be a bounded open set. 
Later, $\tilde U$ will be an open set
which is very close to $U$ in Assumption \ref{ass:main}.  

\begin{definition}\label{def:op}

For $N \geq 2$, an \emph{open partition} of $\tilde U$ in $N$ elements is a finite and ordered collection $\E = \{E_{i}\}_{i=1}^N$ of subsets $E_{i} \subset \tilde U$ such that:
\begin{itemize}
\item[(a)] $E_1,\dots,E_N$ are open and mutually disjoint;
\item[(b)] $\Ha^n(\tilde U \setminus \bigcup_{i=1}^N E_i) < \infty$;
\item[(c)] $\bigcup_{i=1}^N \partial E_i\subset \tilde U$ is countably $n$-rectifiable.
\end{itemize}
The set of all open partitions of $\tilde U$ of $N$ elements will be denoted $\op^N(\tilde U)$. 
\end{definition}
\noindent
Note that some of the $E_i$ may be empty. Condition (b) implies that
\begin{equation}\label{etop}
\tilde U \setminus \bigcup_{i=1}^N E_i = \bigcup_{i=1}^N \partial E_i\,,
\end{equation}
and thus that $\bigcup_{i=1}^N \pa E_i$ is $\Ha^n$-rectifiable and each $E_i$ is in fact an open set with finite perimeter in $\tilde U$. By De Giorgi's structure theorem, the reduced boundary $\partial^*E_i$ is $\Ha^n$-rectifiable: nonetheless, the reduced boundary $\partial^*E_i$ may not coincide in general with the topological boundary $\partial E_i$, which makes condition (c) not redundant. We keep the following for 
later use. The proof is straightforward. 
\begin{lemma} \label{difpa}
Suppose $\E=\{E_i\}_{i=1}^N\in \op^N(\tilde U)$ and $f:\R^{n+1}\to\R^{n+1}$ is a
$C^1$ diffeomorphism. Then we have $\{f(E_i)\}_{i=1}^N\in \op^N(f(\tilde U))$. 
\end{lemma}
\begin{notation}

Given $\E \in \op^N(\tilde U)$, we will set
\begin{equation} \label{e:interior boundary}
\partial\E := \var\left(\bigcup_{i=1}^N \partial E_i , 1\right) \in \IV_{n}(\R^{n+1})\,.
\end{equation}
Here, to avoid some possible confusion, we emphasize that we want to consider
$\partial\E$ as a varifold on $\R^{n+1}$ when we construct approximate MCF. 
On the other hand, note
that we still consider the relative topology of $\tilde U$, as $\pa E_i\subset \tilde U$ here. In particular, writing $\Gamma=\cup_{i=1}^N\pa E_i$, we have $\|\pa\E\|=\Ha^n\mres_{\Gamma}$, and 
\[
\partial\E(\varphi) = \int_{\Gamma} \varphi(x,T_x\,\Gamma)\, d\Ha^n(x) \qquad \mbox{for every $\varphi \in C_{c}(\bG_n(\R^{n+1}))$}\,,
\]
where $T_x\,\Gamma \in \bG(n+1,n)$ is the approximate tangent plane to $\Gamma$ at $x$, which exists and is unique at $\Ha^n$-a.e. $x \in \Gamma$ because of Definition \ref{def:op}(c).
\end{notation}

\begin{definition} \label{def:E-admissible}

Given $\E=\{E_i\}_{i=1}^N \in \op^N(\tilde U)$ and a closed set $C\ssubset \tilde U$, a function $f \colon \R^{n+1} \to \R^{n+1}$ is \emph{$\E$-admissible} in $C$ if it is Lipschitz continuous and satisfies the following. Let $\tilde{E}_i := \Int(f(E_i))$ for $i = 1,\dots,N$. Then:
\begin{itemize}
\item[(a)] $\{x\,:\,x\neq f(x)\}\cup \{f(x)\,:\,x\neq f(x)\}\subset C$;
\item[(b)] $\{\tilde{E}_i\}_{i=1}^N$ are mutually disjoint;
\item[(c)] $\tilde U \setminus \bigcup_{i=1}^N \tilde{E}_i \subset f(\bigcup_{i=1}^N \partial E_i)$.

\end{itemize}

\end{definition}

\begin{lemma} \label{l:preserving partitions}

Let $\E = \{E_i\}_{i=1}^N \in \op^N(\tilde U)$ be an open partition of $\tilde U$ in $N$ elements,
$C\ssubset \tilde U$, and let $f$ be $\E$-admissible in $C$. If we define $\tilde\E := \{\tilde E_{i}\}_{i=1}^N$ with $\tilde E_i := \Int(f(E_i))$, then $\tilde\E \in \op^N(\tilde U)$.

\end{lemma}

\begin{proof}
We check that $\tilde \E$ satisfies properties (a)-(c) in Definition \ref{def:op}.
By Definition \ref{def:E-admissible}(a) and (b), it is clear that $\tilde E_1,\ldots,
\tilde E_N$ are open and mutually disjoint subsets of $\tilde U$, which gives (a).
In order to prove (b), we use Definition \ref{def:E-admissible}(c) and the area formula to compute:
\[
\Ha^n\Big( \tilde U \setminus \bigcup_{i=1}^N \tilde E_i \Big)
\leq \Ha^n\Big(f(\bigcup_{i=1}^N \partial E_i)\Big) 
\leq \Lip(f)^n \, \Ha^n\Big( \bigcup_{i=1}^N \partial E_i \Big) < \infty\,,
\]
where we have used Definition \ref{def:op}(b) and \eqref{etop}. This also shows 
$\tilde U\setminus\bigcup_{i=1}^N \tilde E_i=\bigcup_{i=1}^N\partial\tilde E_i$.
Finally, we prove property (c). Observe that, since $\bigcup_{i=1}^N \partial E_i$ is countably $n$-rectifiable, also the set $f(\bigcup_{i=1}^N \partial E_i)$ is countably $n$-rectifiable. Since any subset of a countably $n$-rectifiable set is countably $n$-rectifiable, also $\bigcup_{i=1}^N \partial\tilde E_i$ is countably $n$-rectifiable.
\end{proof}

\begin{notation}
If $\E \in \op^N(\tilde U)$ and $f \in \Lip(\R^{n+1}; \R^{n+1})$ is $\E$-admissible in $C$
for some $C\ssubset \tilde U$, then the open partition $\tilde \E \in \op^N(\tilde U)$ will be denoted $f_{\star}\E$. 
\end{notation}

\subsection{Area reducing Lipschitz deformations}

\begin{definition}\label{def:Lip_def}
For $\E = \{E_i\}_{i=1}^N \in \op^N(\tilde U)$, $j \in \Na$ and a closed set $C \ssubset\tilde U$, define $\bE(\E, C, j)$ to be the set of all $\E$-admissible functions $f$ in $C$ such that:

\begin{itemize}

\item[(a)] $\abs{f(x) - x} \leq \sfrac{1}{j^2}$ for every $x \in C$;

\item[(b)] $\Leb^{n+1}(\tilde E_i \triangle E_i) \leq \sfrac1j$ for all $i = 1,\dots,N$, where $\tilde E_i = \Int(f(E_i))$, and where $E \triangle F := \left[ E \setminus F \right] \cup \left[ F \setminus E \right]$ is the symmetric difference of the sets $E$ and $F$;

\item[(c)] $\| \partial f_{\star}\E\|(\phi) \leq \|\partial\E\|(\phi)$ for all $\phi \in \cA_j$. Here, $f_{\star}\E = \{\tilde E_i \}_{i=1}^N$ and $\|\partial\E\|$ is the weight of the multiplicity one varifold associated to the open partition $\E$.

\end{itemize}

\end{definition}
\noindent
The set $\bE(\E,C,j)$ is not empty, as it contains the identity map. 

\begin{definition} \label{def:excess}
Given $\E \in \op^N(\tilde U)$ and $j$, and given a closed set $C \ssubset \tilde U$, we define 
\begin{equation} \label{e:excess}
\begin{split}
\Delta_j\|\partial\E\|(C) :&= \inf_{f \in \bE(\E,C,j)} \left\lbrace    \|\partial f_{\star}\E\|(C) - \|\partial \E\|(C) \right\rbrace \\ &= \inf_{f \in \bE(\E,C,j)} \left\lbrace    \|\partial f_{\star}\E\|(\R^{n+1}) - \|\partial \E\|(\R^{n+1}) \right\rbrace\,.
\end{split}
\end{equation}
\end{definition}
\noindent
Observe that it always holds $\Delta_j\|\partial\E\|(C) \leq 0$, since the identity map $f(x)=x$ belongs to $\bE(\E,C,j)$. The quantity $\Delta_j\|\partial\E\|(C)$ measures the extent to which $\|\partial\E\|$ can be reduced by acting with area reducing Lipschitz deformations in $C$. 

\subsection{Smoothing of varifolds and first variations}

We let $\psi \in C^{\infty}(\R^{n+1})$ be a radially symmetric function such that
\begin{equation} \label{e:cutoff_smoothing}
\begin{split}
& \psi(x) = 1 \mbox{ for } \abs{x} \leq 1/2\,, \qquad \psi(x) = 0 \mbox{ for } \abs{x} \geq 1\,, \\
& 0 \leq \psi(x) \leq 1\,, \quad \abs{\nabla \psi(x)} \leq 3\,, \quad \|\nabla^2\psi(x)\| \leq 9 \mbox{ for all } x \in \R^{n+1}\,, 
\end{split}
\end{equation}
and we define, for each $\eps \in \left( 0, 1 \right)$, 
\begin{equation} \label{e:smoothing_kernel}
\hat\Phi_\eps(x) := \frac{1}{(2\pi \eps^2)^{\frac{n+1}{2}}} \, \exp\left( - \frac{\abs{x}^2}{2\eps^2} \right)\,, \quad \Phi_\eps(x) := c(\eps)\, \psi(x) \, \hat\Phi_\eps(x)\,,
\end{equation}
where the constant $c(\eps)$ is chosen in such a way that
\begin{equation} \label{e:normalization_kernel}
\int_{\R^{n+1}} \Phi_\eps(x) \, dx = 1\,.
\end{equation}

The function $\Phi_\eps$ will be adopted as a convolution kernel for the definition of the smoothing of a varifold. We record the properties of $\Phi_\eps$ in the following lemma (cf. \cite[Lemma 4.13]{KimTone}).

\begin{lemma} \label{l:properties_kernel}
There exists a constant $c = c(n)$ such that, for $\eps \in \left( 0, 1 \right)$, we have:

\begin{align} 
\abs{\nabla \Phi_\eps(x)} &\leq \frac{\abs{x}}{\eps^2} \, \Phi_\eps(x) + c\, \chi_{B_1 \setminus B_{1/2}}(x) \, \exp(-\eps^{-1})\,, \label{e:1st_bound} \\
\|\nabla^2 \Phi_\eps(x)\| &\leq \frac{\abs{x}^2}{\eps^4} \, \Phi_\eps(x) + \frac{c}{\eps^2}\, \Phi_\eps(x) + c\,  \chi_{B_1 \setminus B_{1/2}}(x) \, \exp(-\eps^{-1})\,. \label{e:2nd_bound}
\end{align}
\end{lemma}

Next, we use the convolution kernel $\Phi_\eps$ in order to define the smoothing of a varifold and its first variation. Recall that, given a Radon measure $\mu$ on $\R^{n+1}$, the smoothing of $\mu$ by means of the kernel $\Phi_\eps$ is defined to be the Radon measure $\Phi_\eps \ast \mu$ given by
\begin{equation} \label{e:smoothing_measure}
(\Phi_\eps \ast \mu)(\phi) := \mu(\Phi_\eps \ast \phi) = \int_{\R^{n+1}} \int_{\R^{n+1}} \Phi_\eps(x-y) \, \phi(y) \, dy \, d\mu(x) \qquad \mbox{for every } \phi \in C_c(\R^{n+1}) \,.
\end{equation}

The definition of smoothing of a varifold $V$ is the equivalent of \eqref{e:smoothing_measure} when regarding $V$ as a Radon measure on $\bG_n(\R^{n+1})$, keeping in mind that the operator $(\Phi_\eps \ast)$ acts on a test function $\varphi \in C_c(\bG_n(\R^{n+1}))$ by convolving only the space variable. Explicitly, we give the following definition.

\begin{definition} \label{def:smoothing_varifold}
Given $V \in \V_n(\R^{n+1})$, we let $\Phi_\eps \ast V \in \V_n(\R^{n+1})$ be the varifold defined by
\begin{equation} \label{e:smoothing_varifold}
(\Phi_\eps \ast V) (\varphi) := V (\Phi_\eps \ast \varphi) = \int_{\bG_n(\R^{n+1})} \int_{\R^{n+1}} \Phi_\eps(x-y) \, \varphi(y,S)  \, dy \, dV(x,S) 
\end{equation} 
for every $\varphi \in C_c(\bG_n(\R^{n+1}))$. 
\end{definition}

Observe that, given a Radon measure $\mu$ on $\R^{n+1}$, one can identify the measure $\Phi_\eps \ast \mu$ with a $C^{\infty}$ function by means of the Hilbert space structure of $L^2(\R^{n+1}) = L^2(\mathcal L^{n+1})$. Indeed, for any $\phi \in C_c(\R^{n+1})$ we have that
\[
(\Phi_\eps \ast \mu) (\phi) = \langle \Phi_\eps \ast \mu\,|\, \phi \rangle_{L^2(\R^{n+1})}\,,
\] 
where $\Phi_\eps \ast \mu \in C^\infty(\R^{n+1})$ is defined by 
\[
(\Phi_{\eps} \ast \mu)(x) := \int_{\R^{n+1}} \Phi_\eps(x-y) \, d\mu(y)\,.
\]

These considerations suggest the following definition for the smoothing of the first variation of a varifold.

\begin{definition} \label{def:smoothing_first_var}
Given $V \in \V_n(\R^{n+1})$, the smoothing of $\delta V$ by means of the convolution kernel $\Phi_\eps$ is the vector field $\Phi_\eps \ast \delta V \in C^{\infty}(\R^{n+1}; \R^{n+1})$ defined by
\begin{equation} \label{e:smoothing_first_var}
 (\Phi_\eps \ast \delta V) (x) := \int_{\bG_n(\R^{n+1})} S(  \nabla \Phi_\eps(y-x)  ) \, dV(y,S)\,, 
\end{equation}
in such a way that
\begin{equation} \label{e:why_smoothing_first_var}
 \delta V (\Phi_\eps \ast g) = \langle \Phi_\eps \ast \delta V \, | \, g \rangle_{L^2(\R^{n+1})} \qquad \mbox{for every } g \in C^{1}_c(\R^{n+1}; \R^{n+1})\,.
\end{equation}

\end{definition}

\begin{lemma} \label{l:smoothing}
For $V \in \V_n(\R^{n+1})$, we have
\begin{align} 
\Phi_\eps \ast \|V \| &= \| \Phi_\eps \ast V \| \,, \label{e:comm1} \\
\Phi_\eps \ast \delta V &= \delta (\Phi_\eps \ast V) \,. \label{e:comm2}
\end{align}
Moreover, if $\|V\|(\R^{n+1}) < \infty$ then
\begin{equation} \label{smoothing mass estimate}
\| \Phi_\eps \ast V \| (\R^{n+1}) \leq \|V\|(\R^{n+1})   \,.
\end{equation}
\end{lemma}

\begin{proof}

The identities \eqref{e:comm1} and \eqref{e:comm2} are proved in \cite[Lemma 4.16]{KimTone}. Concerning \eqref{smoothing mass estimate}, we observe that for any $\varphi \in C_c(\bG_n(\R^{n+1}))$ with $\| \varphi \|_{0} \leq 1$, setting $\tau_z(x) := x-z$, it holds:
\[
\begin{split}
(\Phi_\eps \ast V) (\varphi) &= \int_{\bG_n(\R^{n+1})} \int_{\R^{n+1}} \Phi_\eps(x-y) \, \varphi(y,S)  \, dy \, dV(x,S)  \\
&=  \int_{\bG_n(\R^{n+1})} \int_{\R^{n+1}} \Phi_\eps(z) \, \varphi(x-z,S) \, dz \, dV(x,S) \\
&= \int_{\R^{n+1}} \Phi_\eps(z) \int_{\bG_n(\R^{n+1})} \varphi(\tau_z(x),S) \, dV(x,S)\, dz \leq \|V\|(\R^{n+1})\,.
\end{split}
\]
Taking the supremum among all functions $\varphi \in C_c(\bG_n(\R^{n+1}))$ with $\|\varphi\|_0 \leq 1$ completes the proof.
\end{proof}

\subsection{Smoothed mean curvature vector}

\begin{definition} \label{def:smcv}
Given $V \in \V_n(\R^{n+1})$ and $\eps \in \left( 0, 1 \right)$, the \emph{smoothed mean curvature vector} of $V$ is the vector field $h_\eps(\cdot, V) \in C^\infty(\R^{n+1}; \R^{n+1})$ defined by
\begin{equation} \label{e:smcv}
h_\eps(\cdot,V) := - \Phi_\eps \ast \left(  \frac{\Phi_\eps \ast \delta V}{\Phi_\eps \ast \|V \| + \eps}  \right)\,.
\end{equation}

\end{definition}

We will often make use of \cite[Lemma 5.1]{KimTone} with $\Omega \equiv 1$ (and $c_1 = 0$). For the reader's convenience, we provide here the statement.

\begin{lemma} \label{l:smc estimates}
For every $M > 0$, there exists a constant $\eps_1 \in \left( 0,1 \right)$, depending only on $n$ and $M$ such that the following holds. Let $V \in \V_{n}(\R^{n+1})$ be an $n$-dimensional varifold in $\R^{n+1}$ such that $\|V\|(\R^{n+1}) \leq M$, and, for every $\eps \in \left( 0, \eps_1 \right)$, let $h_{\eps}(\cdot, V)$ be its smoothed mean curvature vector. Then:
\begin{equation} \label{e:h in L infty}
\abs{h_\eps(x, V)} \leq 2 \, \eps^{-2}\,,
\end{equation}

\begin{equation} \label{e:nabla h in L infty}
\|\nabla h_\eps(x,V)\| \leq 2\, \eps^{-4}\,,
\end{equation}

\begin{equation} \label{e:nabla2 h in L infty}
\| \nabla^2 h_\eps(x,V) \| \leq 2\, \eps^{-6}\,.
\end{equation}

\end{lemma}

\subsection{The cut-off functions $\eta_j$}

In this subsection we construct the cut-off functions which will later be used to truncate the smoothed mean curvature vector in order to produce time-discrete approximate flows which \emph{almost} preserve the boundary $\pa\Gamma_0$.

\smallskip

Given a set $E \subset \R^{n+1}$ and $s > 0$, $(E)_s$ denotes the $s$-neighborhood of $E$, namely the open set
\[
(E)_s := \bigcup_{x \in E} U_{s}(x)\,.
\] 
We shall also adopt the convention that $(E)_0 = E$. \\

\smallskip

Let $U$ and $\Gamma_0$ be as in Assumption \ref{ass:main}.

\begin{definition} \label{D and K sets}

We define for $j\in\mathbb N$:
\begin{equation} \label{D sets}
D_{j} := \left\lbrace x \in U \, \colon \, \dist(x, \partial U) \geq \frac{2}{j^{\sfrac{1}{4}}}   \right\rbrace\,.
\end{equation}
Observe that $D_j$ is not empty for all $j$ sufficiently large (depending on $U$).

Also, we define the sets
\begin{equation} \label{K sets}
K_j := \left( \Gamma_0 \setminus D_j \right)_{1/j^{\sfrac{1}{4}}}\,, \qquad \tilde K_j := \left( \Gamma_0 \setminus D_j \right)_{2/j^{\sfrac{1}{4}}}\,, \quad \mbox{and} \quad \hat K_j := \left( \Gamma_0 \setminus D_j  \right)_{3/j^{\sfrac18}}\,,
\end{equation}
so that $K_j \subset \tilde K_j \subset \hat K_j$.
\end{definition}

\begin{definition} \label{def:etaj}
Let $\psi \colon \left(0, \infty \right) \to \R$ be a smooth function satisfying the following properties:
\begin{itemize}
\item[(a)] $0 \leq \psi(t) \leq 1$ for every $t>0$, $\psi(t) = t$ for $t \in \left( 0, 1/2 \right]$, $t/2 \leq \psi(t) \leq t$ for $t \in \left[ 1/2, 3/2 \right]$, $\psi(t) = 1$ for $t \geq 3/2$;

\item[(b)] $0 \leq \psi'(t) \leq 1$ for every $t > 0$;

\item[(c)] $\abs{\psi''(t)} \leq 2$ for every $t > 0$.
\end{itemize}

For every $j \in \Na$, set
\[
\hat{\d}_j(x) := \dist(x, \R^{n+1} \setminus \left( \Gamma_0 \setminus D_j  \right)_{2/j^{\sfrac18}}) \qquad \mbox{for every $x \in \R^{n+1}$}\,.
\]

Let $\{\phi_\rho\}_{\rho}$, $\rho > 0$, be a standard family of mollifiers: precisely, let 
\[
\phi(w) := 
\begin{cases}
A_n \, \exp\left( \frac{1}{\abs{w}^2 - 1} \right) & \mbox{if $\abs{w} < 1$}\\
0 & \mbox{otherwise}\,,
\end{cases}
\]
for a suitable normalization constant $A_n$ chosen in such a way that $\int_{\R^{n+1}} \phi(w) \, dw = 1$, and define $\phi_\rho(z) := \rho^{-(n+1)}\, \phi(z/\rho)$. Then, set $\rho_j := 1/(j^{\sfrac14})$, and $\d_j := \phi_{\rho_j} \ast \hat{\d}_j$. We finally define
\begin{equation} \label{e:etaj}
\eta_j(x) := \psi\left( \exp\left( - j^{\sfrac14} (\d_j(x) - j^{-\sfrac14})  \right) \right)\,.
\end{equation}
\end{definition}

\begin{lemma} \label{l:etaj} 
There exists $J = J(n)$ such that the following properties hold for all $j \geq J$:
\begin{enumerate}
\item $\eta_j \equiv 1$ on $\R^{n+1} \setminus \hat K_j$;

\item $0 < \eta_j \leq \exp(-j^{\sfrac18})$ on $\tilde K_j$;

\item $\eta_j \in \cA_{j^{\sfrac34}}$.
\end{enumerate}
\end{lemma}

\begin{proof}
For the proof of (1), if $x \notin \hat{K}_j$ then $\hat\d_j(x) = 0$. Moreover, since $\rho_j = j^{-\sfrac{1}{4}} < j^{-\sfrac{1}{8}}$, evidently $\hat\d_j(y) = 0$ for all $y \in B_{\rho_j}(x)$. This implies that
\[
\d_j(x) = (\phi_{\rho_j} \ast \hat\d_j)(x) = \int_{B_{\rho_j}(x)} \phi_{\rho_j}(x-y)\, \hat\d_j(y) \, dy = 0\,.
\]
Hence, $\eta_j(x) = \psi(e) = 1$ because of property (a) of $\psi$ in Definition \ref{def:etaj}. 

\smallskip

Next, we prove (2). Let $x \in \tilde K_j$, so that there exists $z \in \Gamma_0 \setminus D_j$ such that $\abs{x-z} < 2\,j^{-\sfrac14}$. If $y \in B_{\rho_j}(x)$, then $\abs{y-z} < 3\, j^{-\sfrac14}$ by the definition of $\rho_j$, and thus, for $j$ suitably large,
\[
\hat\d_j(y) = \dist(y, \R^{n+1} \setminus \left( \Gamma_0 \setminus D_j  \right)_{2/j^{\sfrac18}}) \geq 2 j^{-\sfrac{1}{8}} - 3\, j^{-\sfrac14}\,,
\]
which in turn implies
\[
\d_j(x) = (\phi_{\rho_j} \ast \hat\d_j)(x) = \int_{B_{\rho_j}(x)} \phi_{\rho_j}(x-y)\, \hat\d_j(y) \geq 2 j^{-\sfrac{1}{8}} - 3\, j^{-\sfrac14}\,.
\]
Hence, setting $t := \exp\left( - j^{\sfrac14} (\d_j(x) - j^{-\sfrac14})\right)$ we have that $0 < t \leq \exp(4 - 2\, j^{\sfrac{1}{8}}) \leq 1/2$ for $j$ large enough. Hence, by property (a) of $\psi$ in Definition \ref{def:etaj}:
\[
\eta_j(x) = \psi(t) = t \leq \exp(4 - 2\, j^{\sfrac{1}{8}}) \qquad \mbox{for every $x \in \tilde K_j$}\,.
\]
In particular, up to taking larger values of $j$, we see that 
\[
0 < \eta_j(x) \leq e^{-j^{\sfrac18}} \qquad \mbox{for every $x \in \tilde K_j$}\,.
\]

\smallskip

Finally, we prove (3). To this aim, we compute the gradient of $\eta_j$: at any point $x$, we have
\[
\nabla\eta_j =- j^{\sfrac14}\, \psi'(t) \, t\, \nabla\d_j\,.
\]
Using that $t = \psi(t)$ for $0 \leq t \leq 1/2$, $\psi'(t) = 0$ for $t \geq 3/2$, and that $\abs{t} = t \leq 2\,\psi(t)$ for $t \in \left[ 1/2,3/2\right]$, together with the fact that $\abs{\psi'} \leq 1$, we can estimate
\begin{equation} \label{e:gradient etaj}
\abs{\nabla\eta_j} \leq 2\,j^{\sfrac14} \, \abs{\nabla\d_j}\, \eta_j \leq 2\, j^{\sfrac14}\,\eta_j\,, 
\end{equation}
where we have used that $\nabla\d_j(x) = \phi_{\rho_j} \ast \nabla \hat{d}_j (x)$, so that
\[
\abs{\nabla \d_j (x)} \leq \int_{B_{\rho_j}(x)} \phi_{\rho_j}(x-y) \, \abs{\nabla \hat\d_j (y)} \leq 1\,.
\]

In particular, $\abs{\nabla \eta_j} \leq j^{\sfrac34}\, \eta_j$ as soon as $j \geq 4$. Next, we compute the Hessian of $\eta_j$
\[
\nabla^2\eta_j = j^{\sfrac12} \, t \, \left( t\,\psi''(t) + \psi'(t) \right) \nabla\d_j \otimes \nabla\d_j - j^{\sfrac14} \, \psi'(t) \, t \, \nabla^2\d_j\,,
\]
from which we estimate
\[
\| \nabla^2\eta_j\| \leq 100 \, j^{\sfrac12} \, \eta_j + j^{\sfrac14} \, \eta_j \, \|\nabla^2\d_j\|\,.
\]

Now, observe that 
\[
\begin{split}
\| \nabla^2\d_j\| & \leq \int_{B_{\rho_j}(x)} \| \nabla \phi_{\rho_j} (x-y) \otimes \nabla \hat\d_j(y)\| \, dy \leq \int_{B_{\rho_j}} \abs{\nabla \phi_{\rho_j} (z)} \, dz  \\
& = \rho_j^{-1} \int_{B_1} \abs{\nabla\phi(w)} \, dw = C(n)\, \rho_j^{-1}\,.
\end{split}
\]

Hence, recalling that $\rho_j = j^{-\sfrac14}$, we conclude the estimate
\begin{equation} \label{e:hessian etaj}
\|\nabla^2\eta_j\| \leq C(n)\, j^{\sfrac12} \, \eta_j
\end{equation}
for a constant $C$ depending only on $n$. Thus, we conclude $\eta_j \in \cA_{j^{\sfrac34}}$ for $j$ sufficiently large. 
\end{proof}

\subsection{$L^2$ approximations}

In this subsection, we collect a few estimates of the error terms deriving from working with smoothed first variations and smoothed mean curvature vectors. They will be critically important to deduce the convergence of the discrete approximation algorithm. The first estimate is a modification of \cite[Proposition 5.3]{KimTone}. We let $\eta_j$ be the cut-off function as in Definition \ref{def:etaj}, corresponding to $U$ and $\Gamma_0$, and we will suppose that $j \geq J(n)$, in such a way that the conclusions of Lemma \ref{l:etaj} are satisfied.

\begin{proposition} \label{p:prop53}
For every $M > 0$, there exists $\eps_2 \in \left( 0, 1 \right)$ depending only on $n$ and $M$ such that the following holds. For any $j \geq J(n)$, $g \in \cB_j$, $V \in \V_n(\R^{n+1})$ with $\|V\|(\R^{n+1}) \leq M$, $\eps \in \left( 0, \eps_2 \right)$ with 
\begin{equation} \label{e:eps_smallness}
j \leq \frac12 \, \eps^{-\frac16}\,,
	\end{equation}
we have for $h_\eps(\cdot) = h_\eps(\cdot,V)$:
\begin{equation} \label{e:prop53}
\Abs{\int_{\R^{n+1}}    h_\eps \cdot \eta_j\,g \, d\|V\| + \int_{\R^{n+1}}   \, (\Phi_\eps \ast \delta V) \cdot \eta_j\,g \, dx      } \leq \eps^{\frac14} \, \left(  \int_{\R^{n+1}} \eta_j   \frac{\abs{\Phi_\eps \ast \delta V}^2}{\Phi_\eps \ast \|V\| + \eps} \, dx \right)^{\frac12}\,.
\end{equation}
\end{proposition}
Given the validity of \eqref{e:why_smoothing_first_var}, we see that \eqref{e:prop53} measures the deviation from the identity \eqref{def:generalized mean curvature}. The difference with \cite[Proposition 5.3]{KimTone} is that there, in place of 
$\eta_j g$ (left-hand side of \eqref{e:prop53}) and $\eta_j$ (right-hand side
of \eqref{e:prop53}), we have $g$ and $\Omega$, respectively. We note that $g\,\eta_j$ satisfies
$|(g\,\eta_j)(x)|\leq j\eta_j(x)$ and $\|\nabla(g\,\eta_j)(x)\|\leq 2\,j^{\sfrac74}\eta_j(x)$: using these, the modification of the proof is 
straightforward, and thus we omit the details. 

\smallskip

The following is \cite[Proposition 5.4]{KimTone}.
\begin{proposition} \label{p:prop54}
There exists a constant $\eps_3 \in \left( 0, 1 \right)$ depending only on $n$ and $M$ with the following property. Given any $V \in \V_n(\R^{n+1})$ with $\|V\|(\R^{n+1}) \leq M$, $j \in \Na$, $\phi \in \cA_j$, and $\eps \in \left( 0, \eps_3 \right)$ satisfying \eqref{e:eps_smallness}, we have:
\begin{align}
\label{e:fv along h vs h in L2}
\Big|  \delta V(\phi\, h_\eps)  +   \int_{\R^{n+1}} \phi \, \frac{\abs{\Phi_\eps \ast \delta V}^2}{\Phi_\eps \ast \|V\| + \eps} \, dx   \Big| & \leq  \eps^{\frac{1}{4}} \, \left( \int_{\R^{n+1}} \phi \, \frac{\abs{\Phi_\eps \ast \delta V}^2}{\Phi_\eps \ast \|V\| + \eps} \, dx + 1  \right), \\
\label{e:L2 norm of h vs approx}
\int_{\R^{n+1}} \abs{h_\eps}^2 \, \phi \, d\|V\| & \leq  (1+\eps^{\frac14}) \int_{\R^{n+1}}  \phi \, \frac{\abs{\Phi_\eps \ast \delta V}^2}{\Phi_\eps \ast \|V\| + \eps} \, dx + \eps^{\frac14}\,.
\end{align}
\end{proposition}
Note that formula \eqref{e:fv along h vs h in L2} estimates the deviation from the identity \eqref{def:generalized mean curvature} with $g = h(\cdot,V)$.

\smallskip

The next statement is \cite[Proposition 5.5]{KimTone}.
The proof is a straightforward modification, using \eqref{e:prop53}.

\begin{proposition} \label{p:prop55}
For every $M > 0$, there exists $\eps_4 \in \left( 0, 1 \right)$ depending only on $n$ and $M$ with the following property. For any $j \geq J(n)$, $g \in \cB_j$, $V \in \V_n(\R^{n+1})$ with $\|V\|(\R^{n+1}) \leq M$, $\eps \in \left( 0, \eps_4 \right)$ satisfying \eqref{e:eps_smallness}, it holds
\begin{equation} \label{e:prop55}
\Abs{  \int_{\R^{n+1}}   h_\eps \cdot \eta_j\,g \, d\|V\| + \delta V (\eta_j \, g) } \leq \eps^{\frac14} \left( 1 +    \left(  \int_{\R^{n+1}} \eta_j\,   \frac{\abs{\Phi_\eps \ast \delta V}^2}{\Phi_\eps \ast \|V\| + \eps} \, dx     \right)^{\frac12}  \right)   \,.
\end{equation}
\end{proposition}

\subsection{Curvature of limit varifolds}

The next Proposition \ref{p:prop56} corresponds to \cite[Proposition 5.6]{KimTone} when there is
no boundary.
\begin{proposition} \label{p:prop56}
Suppose that $\{V_{j_\ell}\}_{\ell=1}^{\infty} \subset \V_n(\R^{n+1})$ and $\{\eps_{j_\ell}\}_{\ell=1}^{\infty} \subset \left( 0, 1 \right)$ are such that:
\begin{enumerate}
\item $\sup_{\ell} \|V_{j_\ell}\|(\R^{n+1}) < \infty$,
\item $\liminf_{\ell \to \infty}  \int_{\R^{n+1}} \eta_{j_\ell}\,  \frac{\abs{\Phi_{\eps_{j_\ell}} \ast \delta V_{j_\ell}}^2}{\Phi_{\eps_{j_\ell}} \ast \|V_{j_\ell}\| + \eps_{j_\ell}} \, dx < \infty  $,
\item $\lim_{\ell \to \infty} \eps_{j_\ell} = 0$ and $j_\ell\leq \varepsilon_{j_\ell}^{-\frac16}/2$. 
\end{enumerate}
Then, there exists a subsequence $\{j'_\ell\}\subset\{j_\ell\}$ such that $V_{j'_{\ell}} \to V \in \V_{n}(\R^{n+1})$ in the sense of varifolds, and $V$ has a generalized mean curvature vector $h(\cdot, V)$ in $U$ such that
\begin{equation} \label{e:prop56}
\int_{U} \abs{h(\cdot,V)}^2 \, \phi \, d\|V\| \leq \liminf_{\ell \to \infty} \int_{\R^{n+1}} \eta_{j_\ell}\, \phi \, \frac{\abs{\Phi_{\eps_{j_\ell}} \ast \delta V_{j_\ell}}^2}{\Phi_{\eps_{j_\ell}} \ast \|V_{j_\ell}\| + \eps_{j_\ell}} \, dx
\end{equation}
for every $\phi \in C_c(U; \R^+)$.
\end{proposition}
\begin{proof}
By (1), we may choose a (not relabeled) subsequence $V_{j_\ell}$ converging to $V$
as varifolds on $\R^{n+1}$, and we may assume that the integrals in (2) for this subsequence 
converge to the $\liminf$ of the original sequence. 
Fix $g\in C_c^2(U;\R^{n+1})$. 
For all sufficiently large $\ell$, we have 
$g\,\eta_{j_\ell}=g$ due to Lemma \ref{l:etaj}(1), \eqref{K sets} and \eqref{D sets}. Moreover, we may assume that $g\,\eta_{j_\ell}\in
\mathcal B_{j_\ell}$ due to Lemma \ref{l:etaj}(3). Then, by \eqref{e:prop55}, (2) and (3), we have
\begin{equation}
\label{p56-1}
\delta V(g)=\lim_{\ell\rightarrow\infty}
\delta V_{j_\ell}(g\,\eta_{j_{\ell}})=-\lim_{\ell\rightarrow\infty}\int_{\R^{n+1}}
h_{\eps_{j_\ell}}(\cdot,V_{j_\ell})\cdot \eta_{j_\ell}\,g\,d\|V_{j_\ell}\|.
\end{equation}
Since $\eta_{j_\ell}\in \mathcal A_{j_\ell}$ in particular, by the Cauchy-Schartz inequality 
and \eqref{e:L2 norm of h vs approx}, we have
\begin{equation}
\label{p56-2}
\delta V(g)\leq \Big(\liminf_{\ell\rightarrow\infty} \int_{\R^{n+1}} \frac{|\Phi_{\eps_{j_\ell}}\ast
\delta V_{j_\ell}|^2\,\eta_{j_\ell}}{\Phi_{\eps_{j_\ell}}\ast\|V_{j_\ell}\|+\eps_{j_\ell}}\,dx\Big)^{\sfrac12}
\Big(\int_{\R^{n+1}} |g|^2\,d\|V\|\Big)^{\sfrac12}.
\end{equation}
This shows that $\delta V$ is absolutely continuous with respect to $\|V\|$ on
$U$ and $h(\cdot, V)$ satisfies
\begin{equation}
\label{p56-3}
\int_{U}|h(\cdot,V)|^2\,d\|V\|\leq 
\liminf_{\ell\rightarrow\infty} \int_{\R^{n+1}} \frac{|\Phi_{\eps_{j_\ell}}\ast
\delta V_{j_\ell}|^2\,\eta_{j_\ell}}{\Phi_{\eps_{j_\ell}}\ast\|V_{j_\ell}\|+\eps_{j_\ell}}\,dx.
\end{equation}
Given $\phi\in C^2_c(U;\R^+)$ ($C_c$ case is by 
approximation), let $i\in \mathbb N$ be
arbitrary and consider $\hat\phi:=\phi+i^{-1}$. For all sufficiently large $\ell$, 
we have $g\,\eta_{j_\ell}\hat\phi\in\mathcal B_{j_\ell}$ and $\eta_{j_\ell}\hat\phi\in \mathcal A_{j_\ell}$
(we may assume $|\hat\phi|<1$ without loss of generality).  
Thus the same computation above with $g\,\eta_{j_\ell}\hat \phi$ yields
\begin{equation}
\label{p56-4}
\int_{\R^{n+1}}h\cdot g\,\hat\phi\,d\|V\|\leq 
\Big(\liminf_{\ell\rightarrow\infty} 
\int_{\R^{n+1}} \frac{|\Phi_{\eps_{j_\ell}}\ast
\delta V_{j_\ell}|^2\,\eta_{j_\ell}\hat\phi}{\Phi_{\eps_{j_\ell}}\ast\|V_{j_\ell}\|+\eps_{j_\ell}}\,dx\Big)^{\sfrac12}
\Big(\int_{\R^{n+1}}|g|^2\hat\phi\,d\|V\|\Big)^{\sfrac12}.
\end{equation}
We let then $i\rightarrow\infty$ in \eqref{p56-4} to replace $\hat \phi$ by $\phi$,
and finally we approximate $h(\cdot,V)$ by $g$ to obtain \eqref{e:prop56}. 
\end{proof}

\subsection{Motion by smoothed mean curvature with boundary damping}

We aim at proving the following proposition: it contains the perturbation estimates for a varifold $V$ which is moved by a vector field consisting of a boundary damping of its smoothed mean curvature for a time $\Delta t$. 

\begin{proposition} \label{p:motion by smc}
There exists $\eps_5 \in \left( 0,1 \right)$, depending only on $n$, $M$ and $U$ such that the following holds. Suppose that:
\begin{enumerate}

\item $V \in \V_n(\R^{n+1})$ satisfies $\spt\,\|V\| \subset \left( U \right)_{1}$ and $\|V\|(\R^{n+1}) \leq M$;

\item $j \geq J(n)$ and $\eta_j$ is as in Definition \ref{def:etaj};

\item $\eps \in \left( 0, \eps_5 \right)$ satisfies \eqref{e:eps_smallness};

\item $\Delta t \in \left[ 2^{-1} \eps^{\kappa}, \eps^\kappa\right]$, with 
\[
\kappa = 3n + 20 \,.
\]
\end{enumerate}
  Define
\[
f(x) := x + \eta_j(x) h_\eps(x,V) \Delta t\,.
\]  
Then, for every $\phi \in \cA_j$ we have the following estimates.

\begin{equation} \label{e:smc1}
\left| \frac{\|f_\sharp V\|(\phi) - \|V\|(\phi)}{\Delta t} - \delta(V,\phi)(\eta_j h_{\eps}(\cdot,V)) \right| \leq \eps^{\kappa - 10}\,,
\end{equation}

\begin{equation} \label{e:smc2}
\frac{\|f_\sharp V\|(\R^{n+1}) - \|V\|(\R^{n+1})}{\Delta t} + \frac{1}{4} \int_{\R^{n+1}} \eta_j\, \frac{\abs{\Phi_\eps \ast \delta V}^2}{\Phi_\eps \ast \|V\| + \eps} \, dx \leq 2\,\eps^{\sfrac14}\,.
\end{equation}

Furthermore, if also $\|f_\sharp V\|(\R^{n+1}) \leq M$, then we have

\begin{equation} \label{e:smc3}
\abs{\delta(V,\phi)(\eta_j \, h_\eps(\cdot, V)) - \delta(f_\sharp V, \phi)(\eta_j \, h_\eps(\cdot, f_\sharp V))} \leq \eps^{\kappa-2n-18}\,,
\end{equation}

\begin{equation} \label{e:smc4}
\left|  \int_{\R^{n+1}} \eta_j \frac{\abs{\Phi_\eps \ast \delta V}^2}{\Phi_\eps \ast \|V\| + \eps} \, dx - \int_{\R^{n+1}} \eta_j \frac{\abs{\Phi_\eps \ast \delta (f_\sharp V)}^2}{\Phi_\eps \ast \|f_\sharp V\| + \eps} \, dx       \right| \leq \eps^{\kappa-3n-18}\,.
\end{equation}

\end{proposition}

\begin{proof}
We want to estimate the following quantity

\[
A := \|f_\sharp V \|(\phi) - \|V\|(\phi) - \delta(V,\phi)(\eta_j h_\eps(\cdot,V)) \, \Delta t = \|f_\sharp V \|(\phi) - \|V\|(\phi) - \delta(V,\phi)(F)\,,
\]
where $F(x) := \eta_j(x) h_\eps(x,V) \Delta t = f(x) - x$. By \eqref{pushfd} and \eqref{defFV2}, we have that
\[
A = \int_{\bG_{n}(\R^{n+1})} \{ \phi(f(x)) \, \abs{\Lambda_n\nabla f(x) \circ S} - \phi(x) - \phi(x) \, \nabla F \cdot S - F \cdot \nabla \phi \} \, dV(x,S)\,,
\]
which can be written as
\[
A = I_1 + I_2 + I_3\,,
\]
with 
\begin{align*}
I_1 :&= \int_{\bG_{n}(\R^{n+1})} \left( \phi(f(x)) - \phi(x) \right) \, \left( \abs{\Lambda_n\nabla f(x) \circ S} -1 \right) \, dV(x,S)\,, \\
I_2 :&= \int_{\bG_n(\R^{n+1})} \phi(x)\, \left( \abs{\Lambda_n\nabla f(x) \circ S} - 1 - \nabla F \cdot S \right) \, dV(x,S)\,, \\
I_3 :&= \int_{\bG_{n}(\R^{n+1})} \phi(f(x)) - \phi(x) - \nabla\phi(x) \cdot F(x)\, dV(x,S)\,.
\end{align*}

Choose $\eps_5 \leq \min\{\eps_1, \eps_3\}$, so that the conclusions of Lemma \ref{l:smc estimates} and Proposition \ref{p:prop54} hold with $\eps \in \left( 0, \eps_5 \right)$. In order to estimate the size of the various integrands appearing in the definition of $I_1, I_2$ and $I_3$, we first observe that, by \eqref{e:h in L infty} and our assumption on $\Delta t$,

\begin{equation} \label{F in L infty}
\abs{F(x)} = \abs{\eta_j h_\eps(\cdot, V) \Delta t} \leq 2\, \eps^{\kappa - 2}\,.
\end{equation}

Furthermore, using \eqref{e:h in L infty}, \eqref{e:nabla h in L infty}, \eqref{e:eps_smallness}, and the fact that $\eta_j \in \cA_j$ we obtain

\begin{equation} \label{nabla F in L infty}
\| \nabla F \| \leq \Delta t\, \left( \eta_j \|\nabla h_\eps\| + \| h_\eps \otimes \nabla \eta_j\| \right) \leq \eps^\kappa \left( 2\,\eps^{-4} + 2\, j \, \eps^{-2} \right) \leq 3\, \eps^{\kappa - 4}\,.
\end{equation}

Since $\phi \in \cA_j$, we can use the results of Lemma \ref{l:class properties} to estimate:

\begin{align} 
\abs{\phi(f(x)) - \phi(x)} &\overset{\eqref{e:1st_order}}{\leq} j^{} \abs{F(x)} \phi(x) \exp\left( j^{} \abs{F(x)} \right) \leq \eps^{\kappa - 3}\,, \label{e:test1}\\
\abs{\phi(f(x)) - \phi(x) - \nabla\phi(x) \cdot F(x)} &\overset{\eqref{e:2nd_order}}{\leq} j^{} \abs{F(x)}^2 \phi(x) \exp\left( j^{} \abs{F(x)} \right) \leq \eps^{\kappa-5} \, \Delta t\,. \label{e:test2}
\end{align}

Analogously, using that $f(x) = x + F(x)$, so that
\[
\abs{\Lambda_n\nabla f(x) \circ S} = \abs{({\rm Id} + \nabla F(x)) \cdot v_1 \wedge \ldots \wedge ({\rm Id} + \nabla F(x)) \cdot v_n}
\]
for any orthonormal basis $\{v_1,\ldots,v_n \}$ of $S$, we can Taylor expand the tangential Jacobian and deduce the estimates

\begin{align}
\Big|\abs{\Lambda_n\nabla f(x) \circ S} - 1\Big| &\leq c(n) \, \|\nabla F \| \overset{\eqref{nabla F in L infty}}{\leq} c(n) \, \eps^{\kappa - 4} \leq c(n) \, \Delta t\, \eps^{-4} \leq \Delta t\, \eps^{-5}\,, \label{e:Jacobian1}\\
\Big|\abs{\Lambda_n\nabla f(x) \circ S} - 1 - \nabla F \cdot S\Big| &\leq c(n) \, \|\nabla F\|^2 \overset{\eqref{nabla F in L infty}}{\leq} c(n) \eps^{2\,\kappa - 8} \leq \eps^{k-9} \, \Delta t\,, \label{e:Jacobian2}
\end{align}
modulo choosing a smaller value of $\eps$ if necessary. Putting all together, we can finally conclude the proof of \eqref{e:smc1}:

\begin{equation} \label{smc1_final}
\abs{A} \leq \abs{I_1} + \abs{I_2} + \abs{I_3} \leq \left( \eps^{\kappa-8} + \eps^{\kappa-9} + \eps^{\kappa-5} \right) \, \Delta t \, \|V\|(\R^{n+1}) \leq \eps^{\kappa-10} \Delta t\,.
\end{equation}

\smallskip

In order to prove \eqref{e:smc2}, we use \eqref{e:smc1} with $\phi(x) \equiv 1$, which implies that
\begin{equation} \label{smc1 implies smc2}
\frac{\|f_\sharp V\|(\R^{n+1}) - \|V\|(\R^{n+1})}{\Delta t} \leq \delta V(\eta_j h_{\eps}(\cdot, V)) + \eps^{\kappa-10}\,.
\end{equation}

On the other hand, since $\eta_j \in \cA_j$ we can apply \eqref{e:fv along h vs h in L2} to further estimate

\begin{equation} \label{first variation estimate}
\delta V(\eta_j h_{\eps}) \leq - (1- \eps^{\sfrac14}) \left(  \int_{\R^{n+1}} \eta_j \, \frac{\abs{\Phi_\eps \ast \delta V}^2}{\Phi_\eps \ast \|V\| + \eps} \, dx  \right) + \eps^{\sfrac14}\,,
\end{equation}

so that \eqref{e:smc2} follows by choosing $\eps$ so small that $1 - \eps^{\sfrac14} \geq 1/4$.

\smallskip

Finally, we turn to the proof of \eqref{e:smc3} and \eqref{e:smc4}. In order to simplify the notation, let us write $\hat{V}$ instead of $f_\sharp V$. Using the same strategy as in \cite[Proof of Proposition 5.7]{KimTone}, we can estimate
\[
\abs{\Phi_\eps \ast \|\hat V\|(x) - \Phi_\eps \ast \|V\|(x)} \leq I_1 + I_2\,,
\]
where 
\[
I_1 = \int \abs{\Phi_\eps(f(y) - x) - \Phi_\eps(y-x)}\, \abs{\Lambda_n \nabla f(y) \circ S} \, dV(y,S)\,,
\]
and
\[
I_2 = \int \Phi_\eps(y-x) \, \abs{\abs{\Lambda_n \nabla f \circ S} -1 }\, dV(y,S)\,.
\]

The first term can be estimated by observing that for some point $\hat y$ on the segment $\left[ y-x, f(y)-x\right]$,
\[
\begin{split}
\abs{\Phi_\eps(f(y) - x) - \Phi_\eps(y-x)} &\leq \abs{\nabla \Phi_\eps(\hat y)} \, \abs{F(y)} \\ 
&\overset{\eqref{e:1st_bound}}{\leq} \abs{F(y)} \, \left( \eps^{-2} \abs{\hat y} \Phi_\eps(\hat y) + c\, \chi_{B_1 \setminus B_{1/2}}(\hat y) \, \exp(-\eps^{-1}) \right)\\
&\overset{\eqref{F in L infty}}{\leq} c(n)\, \eps^{\kappa-n-5}\, \chi_{B_2(x)}(y)\,,
\end{split}
\]
and using that
\[
\abs{\Lambda_n\nabla f(y) \circ S} \leq 1 + \eps^{\kappa - 5}
\]
because of \eqref{e:Jacobian1}, so that
\[
I_1 \leq \eps^{\kappa-n-6} \, \|V\|(B_2(x))\,.
\]
Concerning the second term in the sum, we can use \eqref{e:Jacobian1} again to estimate
\[
I_2 \leq c(n)\, \eps^{-n-1}\, \eps^{\kappa - 5 }\, \|V\|(B_{1}(x))\,.
\]
Putting the two estimates together, we see that
\begin{equation} \label{smoothed measures}
\abs{\Phi_\eps \ast \|\hat V\|(x) - \Phi_\eps \ast \|V\|(x)} \leq \eps^{\kappa-n-7} \, \|V\|(B_2(x))\,.
\end{equation}
Analogous calculations lead to
\begin{equation} \label{smoothed first variations}
\abs{\Phi_\eps \ast \delta \hat V(x) - \Phi_\eps \ast \delta V(x)} \leq \eps^{\kappa-n-9} \, \|V\|(B_2(x))\,.
\end{equation}
The rough estimates also give
\begin{equation} \label{rough estimates}
\abs{\Phi_\eps \ast \delta V(x)}\,, \abs{\Phi_\eps \ast \delta \hat V(x)} \leq \eps^{-n-4} \, \|V\|(B_{2}(x))\,.
\end{equation}
The estimates \eqref{smoothed measures}, \eqref{smoothed first variations}, and \eqref{rough estimates} immediately yield
\begin{equation} \label{comparison1}
\left| \frac{\Phi_\eps \ast \delta\hat V}{\Phi_\eps \ast \|\hat V\| + \eps} - \frac{\Phi_\eps \ast \delta V}{\Phi_\eps \ast \|V\| + \eps}     \right| \leq \eps^{\kappa - n - 10}\, \|V\|(B_2(x)) + \eps^{\kappa-2n-13} \, \|V\|(B_2(x))^2\,,
\end{equation}
as well as
\begin{equation} \label{comparison2}
\left|  \frac{\abs{\Phi_\eps \ast \delta\hat V}^2}{\Phi_\eps \ast \|\hat V\| + \eps}  - \frac{\abs{\Phi_\eps \ast \delta V}^2}{\Phi_\eps \ast \|V\| + \eps}   \right| \leq \eps^{\kappa - 2n - 15} \, \|V\|(B_2(x))^2 + \eps^{\kappa-3n-17} \, \|V\|(B_2(x))^3\,.
\end{equation}

Observe that, since $\spt\|V\|\subset \left( U \right)_1$, the right-hand side of estimates \eqref{comparison1} and \eqref{comparison2} is zero whenever $\dist(x, {\rm clos}(U)) > 3$. Hence,
\eqref{comparison2} and the monotonicity of the mass $\|V\|(B_2(x)) \leq M$ imply that
\[
\begin{split}
&\left|   \int_{\R^{n+1}} \eta_j \frac{\abs{\Phi_\eps \ast \delta V}^2}{\Phi_\eps \ast \|V\| + \eps} \, dx - \int_{\R^{n+1}} \eta_j \frac{\abs{\Phi_\eps \ast \delta (f_\sharp V)}^2}{\Phi_\eps \ast \|f_\sharp V\| + \eps} \, dx       \right| \\ &\hspace{2cm}\leq \left(\eps^{\kappa - 2n - 15} \, M^2 + \eps^{\kappa-3n-17} \, M^3 \right)\, \int_{\left( U \right)_3} \eta_j(x) \, dx \leq \eps^{\kappa-3n-18} 
\end{split}
\]
by possibly choosing a smaller value of $\eps$ (depending on $U$ and $M$). This proves \eqref{e:smc4}. 

Finally, we prove \eqref{e:smc3}. By \eqref{e:smcv}, \eqref{comparison1}, and
the properties of $\Phi_\eps$, we deduce that
\begin{align}\label{heV}
\norm{\nabla^lh_\eps(V) - \nabla^lh_\eps(\hat V)} &\leq \eps^{\kappa-2n-14-2l}(M+M^2)
\end{align}
for $l=0,1,2$. We can conclude using \eqref{heV}, \eqref{F in L infty}-\eqref{e:Jacobian1} and
suitable interpolations that:
\begin{align*}
&\abs{\delta (V,\phi)(\eta_j \, h_\eps(V)) - \delta(\hat V, \phi)(\eta_j \, h_\eps(\hat V))} \\
&\qquad = \Big| \int_{\bG_n(\R^{n+1})} \left\lbrace \phi \, \nabla (\eta_j \, h_{\eps}(V) ) \cdot S + \eta_j \, h_{\eps}(V) \cdot \nabla \phi\right\rbrace dV(x,S) \\
&\qquad \qquad - \int_{\bG_{n}(\R^{n+1})} \big\{ \phi \circ f \, \left[ \nabla(\eta_j \, h_\eps(\hat V)) \right] \circ f  \cdot (\nabla f \circ S) \\
&\qquad \qquad \qquad \qquad \qquad  + (\eta_j \, h_\eps(\hat V)) \circ f \cdot (\nabla \phi \circ f)\big\} \abs{\Lambda_n\nabla f \circ S} \, dV(x,S)   \Big| \\
&\qquad\leq \eps^{\kappa-2n-18}\,. \qedhere
\end{align*}
\end{proof}

\section{Existence of limit measures} \label{sec:limit flow}

\subsection{The construction of the approximate flows}

Suppose $U$ and $\Gamma_0$ are as in Assumption \ref{ass:main}. Together with the sets $D_j, K_j, \tilde K_j, \hat K_j$ introduced in Definition \ref{D and K sets}, for $k = 0,1,\ldots$, we set
\[
D_{j,k} := \left\lbrace x \in U \, \colon \, \dist(x, \partial U) \geq \frac{1}{j^{\sfrac14}} - k \, \exp(-j^{\sfrac18}) \right\rbrace \,.
\]
Once again, here the indices $j$ and $k$ are chosen in such a way that the corresponding sets $D_{j,k}$ are non-empty proper subsets of $U$. Observe that we have the elementary inclusions $D_{j,0} \subset D_{j,k} \subset D_{j,k'}$ for every $0 \leq k \leq k'$, and that $D_j \subset D_{j,k}$ for every $k$.

\smallskip 

Before proceeding with the construction of the time-discrete approximate flows, we need to introduce a suitable new class of test functions. Since $U$ is an open and bounded convex domain with boundary $\pa U$ of class $C^2$, there exists a neighborhood $\left( \pa U \right)_{s_0}$ such that, denoting $\d_{U}(x) := \dist(x, \R^{n+1} \setminus U)$ for $x \in \left( \pa U \right)_{s_0} \cap U$ the distance function from the boundary, the vector field $\nu_{U}(x) := - \nabla \d_{U}(x)$ is a $C^1$ extension to $\left( \pa U \right)_{s_0}^{-} := \left( \pa U \right)_{s_0} \cap U$ of the exterior unit normal vector field to $\pa U$.

\begin{definition} \label{radially increasing functions}
Define the tubular neighborhood of $ \pa U$ and the vector field $\nu_{U}$ as above. Given an open set $W$, a function $\phi \in C^1(\R^{n+1}; \R^+)$ is said to be non decreasing in $W$ along the fibers of the normal bundle of $\pa U$ oriented by $\nu_{U}$, or simply \emph{$\nu_{U}$-non decreasing} in $W$, if for every $x \in W \cap \left( \pa U \right)_{s_0}^{-}$ the map
\[
t \mapsto \phi(x + t \, \nu_{U} (x))
\]
is monotone non decreasing for $t$ such that $x + t \, \nu_{U} (x) \in W \cap \left( \pa U\right)_{s_0}^{-}$. For $j \in \Na$, we will set 
\begin{equation} \label{classR}
\mathcal{R}_j := \left\lbrace \phi \in C^1(\R^{n+1}; \R^+) \, \colon \, \phi \mbox{ is $\nu_{U}$-non decreasing in $\R^{n+1} \setminus D_j$}\right\rbrace\,.
\end{equation}
\end{definition}

The following proposition and its proof contain the constructive algorithm which produces the time-discrete approximations of our Brakke flow with fixed boundary.

\begin{proposition} \label{p:induction}
Let $U$, $\E_0 = \{E_{0,i}\}_{i=1}^N \in \op^N(U)$, and $\Gamma_0$ be as in Assumption \ref{ass:main}. There exists a positive integer $J=J(n)$ with the following property. For every $j \geq J(n)$, there exist $\eps_j \in \left( 0, 1 \right)$ satisfying \eqref{e:eps_smallness}, $p_j \in \Na$, and, for every $k \in \{0,1,\ldots,j\,2^{p_j}\}$, a bounded open set $U_{j,k} \subset \R^{n+1}$ with boundary $\partial U_{j,k}$ of class $C^2$ and an open partition $\E_{j,k} = \{E_{j,k,i}\}_{i=1}^N \in \op^N(U_{j,k})$ such that
\begin{equation} \label{e:initial partiation}
U_{j,0} = U \quad \mbox{and} \quad \E_{j,0} = \E_0 \qquad \mbox{for every $j$}\,,
\end{equation}
and such that, setting $\Delta t_j := 2^{-p_j}$, and defining $\Gamma_{j,k} := U_{j,k} \setminus \bigcup_{i=1}^N E_{j,k,i}$, the following holds true:
\begin{enumerate}

\item $\partial U_{j,k} \subset (\partial U)_{k\,\exp(-j^{\sfrac18})}$ and $U_{j,k} \triangle U \subset \left( \pa U  \right)_{k\,\exp(-j^{\sfrac18})}$,

\item $K_j\cap \Gamma_{j,k}\setminus D_{j,k}\subset (\Gamma_0)_{k\,\exp(-j^{\sfrac18})}$,

\item $\Gamma_{j,k}\setminus K_j\subset (D_{j,k})_{j^{-10}}$.
\end{enumerate}

Moreover, we have:

\begin{equation} \label{induction:mass estimate}
\| \partial \E_{j,k} \|(\R^{n+1}) \leq \| \partial \E_0 \|(\R^{n+1}) + k \, \Delta t_j \, \eps_j^{\sfrac16}\,,
\end{equation}

\begin{equation} \label{induction:mean curvature}
\begin{split}
 \frac{\| \partial \E_{j,k}\| (\R^{n+1})  - \| \partial \E_{j,k-1}\| (\R^{n+1})  }{\Delta t_j} &+ \frac{1}{4} \int_{\R^{n+1}} \eta_j \frac{\abs{\Phi_{\eps_j} \ast \delta (\partial \E_{j,k})}^2}{\Phi_{\eps_j} \ast \| \partial \E_{j,k} \| + \eps_j} \, dx \\
&- \frac{(1 - j^{-5})}{\Delta t_j} \, \Delta_j \|\partial \E_{j,k-1}\|(D_j) \leq \eps_{j}^{\sfrac18}\,,
\end{split}
\end{equation}

\begin{equation} \label{induction:mass variation}
\frac{\| \partial \E_{j,k}\| (\phi)  - \| \partial \E_{j,k-1}\| (\phi) }{\Delta t_j} \leq \delta(\partial \E_{j,k}, \phi)(\eta_j\,h_{\eps_j}(\cdot, \partial \E_{j,k})) + \eps_j^{\sfrac18}
\end{equation}

for every $k \in \{1,\ldots,j\,2^{p_j}\}$ and $\phi \in \cA_j \cap \cR_j$.

\end{proposition}

\smallskip

\begin{proof}[{\bf Proof of Proposition \ref{p:induction}}]
Set
\begin{equation} \label{def of M}
M := \| \pa \E_0 \|(\R^{n+1}) + 1 \,,
\end{equation}
let $\kappa = 3n+20$ as in Proposition \ref{p:motion by smc}, and consider the following set of conditions for $\eps \in \left( 0,1 \right)$:
\begin{equation} \label{epsilon conditions}
\begin{cases}
& \eps < \eps_* := \min\{\eps_1\,,\ldots\,,\eps_5\}\,, \mbox{with $\eps_* = \eps_*(n,U,M)$}\,,\\
&\mbox{\eqref{e:eps_smallness} holds, namely $\eps^{\sfrac16} \leq 1/(2\,j)$}\,,\\
& 2\,\eps^{\kappa-2} \leq j^{-10}\,,\\
&2\, j \, \eps^{-\kappa} \, \exp(-j^{\sfrac18}) \leq 1/(4j^{\sfrac14})\,.
\end{cases}
\end{equation}
Notice that the conditions in \eqref{epsilon conditions} are compatible for large $j$, namely there exists $j_0$ with the property that for every $j \geq j_0$ the set of $\eps \in \left( 0, 1 \right)$ satisfying \eqref{epsilon conditions} is not empty. Letting $J(n)$ be the number provided by Lemma \ref{l:etaj}, for every $j \geq \max\{j_0, J(n)\}$ we choose $\eps_j \in \left( 0, 1 \right)$ such that all conditions in \eqref{epsilon conditions} are met. Observe that $\lim_{j \to \infty} \eps_j = 0$. Then, we choose $p_j \in \Na$ such that
\begin{equation} \label{d:time step}
\Delta t_j := \frac{1}{2^{p_j}} \in \left( 2^{-1} \, \eps_j^{\kappa}, \eps_{j}^{\kappa} \right] \,.
\end{equation}

\smallskip

The argument is constructive, and it proceeds by means of an induction process on $k \in \{0,1,\ldots,j\, 2^{p_j}\}$. We set $U_{j,0} := U$ and $\E_{j,0} := \E_0$. Properties (1), (2), (3), as well as the estimate in \eqref{induction:mass estimate} are then trivially satisfied, given the definition of $M$ and since $U_{j,0}=U$, $\Gamma_0\setminus D_{j,0}\subset
\Gamma_0$ and $\Gamma_0\setminus K_j\subset \Gamma_0 \cap D_j \subset D_{j,0}$. Next, let $k \geq 1$, and assume we obtained the open partition $\E_{j,k-1} = \{E_{j,k-1,i}\}_{i=1}^N$ of $U_{j,k-1}$ satisfying (1), (2), (3), and \eqref{induction:mass estimate} with $k-1$ in place of $k$. We will now produce $U_{j,k}$ and $\E_{j,k} = \{E_{j,k,i}\}_{i=1}^N$ satisfying the same conditions with $k$. At the same time, we will also show that each inductive step satisfies \eqref{induction:mean curvature} and \eqref{induction:mass variation}. Before proceeding, let us record the inductive assumptions for $U_{j,k-1}$ and $\Gamma_{j,k-1}:=U_{j,k-1}\cap\cup_{i=1}^N\partial E_{j,k-1,i}$ in the following set of equations:
\begin{equation}
\label{ind0}
\partial U_{j,k-1}\subset (\partial U)_{(k-1)\exp(-j^{\sfrac18})}\, \quad \mbox{and} \quad U_{j,k-1} \triangle U \subset \left(  \pa U \right)_{{(k-1)\,\exp(-j^{\sfrac18})}}\,,
\end{equation}
\begin{equation}
\label{ind1}
K_j\cap \Gamma_{j,k-1}\setminus D_{j,k-1}\subset (\Gamma_0)_{(k-1)\exp(-j^{\sfrac18})}\,,
\end{equation}
\begin{equation}
\label{ind2}
\Gamma_{j,k-1}\setminus K_j\subset (D_{j,k-1})_{j^{-10}}\,,
\end{equation}
\begin{equation} \label{indmass}
\| \pa \E_{j,k-1} \|(\R^{n+1}) \leq \| \pa \E_0 \|(\R^{n+1}) + (k-1) \, \Delta t_j \, \eps_j^{\sfrac16}\,.
\end{equation}

\smallskip

{\bf Step 1: area reducing Lipschitz deformation.} First notice that $D_{j,k-1} \subset U_{j,k-1}$. Indeed, the definition of $D_{j,k-1}$, \eqref{ind0}, and the choice of $\eps_j$ imply that $D_{j,k-1} \cap (U_{j,k-1} \triangle U) = \emptyset$, so that our claim readily follows from $D_{j,k-1} \subset U$. In particular,
$D_j\subset D_{j,k-1}\subset U_{j,k-1}$. Hence, we can choose $f_{1} \in \bE(\E_{j,k-1},D_j,j)$ such that, setting $\E_{j,k}^\star := (f_1)_{\star}\E_{j,k-1}$ ($\in \op^N(U_{j,k-1})$ by Lemma \ref{l:preserving partitions}), we have
\begin{equation} \label{e:almost minimizing}
\| \partial \E_{j,k}^\star\|(\R^{n+1}) - \| \partial \E_{j,k-1}\|(\R^{n+1}) \leq (1 - j^{-5})\, \Delta_{j}\|\partial \E_{j,k-1}\|(D_j) \, \footnote{Recall that $\Delta_{j}\|\partial \E_{j,k-1}\|(D_j) \leq 0$}\,.    
\end{equation}
Set $\Gamma_{j,k}^\star := U_{j,k-1} \cap \bigcup_{i=1}^N \partial E_{j,k,i}^\star$, and note that
\begin{equation}
\label{starinv}
\Gamma_{j,k}^{\star}\setminus D_j=\Gamma_{j,k-1}\setminus D_j
\end{equation}
and 
\begin{equation} \label{mass estimate step 1}
\|\partial \E_{j,k}^\star\|(\phi) \leq \|\partial \E_{j,k-1}\|(\phi) \qquad \mbox{for every $\phi \in \cA_j$} \,.
\end{equation}

{\bf Step 2:  retraction.}  Outside of $D_{j,k-1}$, we perform a suitable retraction procedure so that 
$\Gamma_{j,k}^\star\setminus (D_{j,k-1}\cup K_j)$ is retracted to $\partial D_{j,k-1}$. 
This retraction step is not needed for $k=1$, since $\Gamma_{j,1}^\star \cap D_{j,0}^c = \Gamma_{j,0} \cap D_{j,0}^c$, and $\Gamma_{j,0}\setminus K_j\subset D_{j,0}$ already. 

Define
\begin{equation}
\label{ind4}
A_{j,k}:=\{x\in \partial (D_{j,k-1})_{j^{-10}}\,:\, {\rm dist}\,(x,\Gamma_0\setminus
D_j)>1/(2j^{1/4})\}\,,
\end{equation}
and observe that $\left. f_{1} \right|_{A_{j,k}} = \left. {\rm id} \right|_{A_{j,k}}$, so that $A_{j,k} \cap E_{j,k,i}^\star = A_{j,k} \cap {\rm int}(f_{1}(E_{j,k-1,i})) = A_{j,k} \cap E_{j,k-1,i}$ for every $i = 1,\ldots,N$. In particular, $\Gamma_{j,k}^\star \cap A_{j,k} = \Gamma_{j,k-1} \cap A_{j,k}$.

We claim the validity of the following
\begin{lemma}
\label{ind5}
We have $A_{j,k}\cap \Gamma_{j,k}^\star=\emptyset$. Moreover, for any $x\in \partial A_{j,k}$
(the boundary as a subset of $\partial (D_{j,k-1})_{j^{-10}}$), we have ${\rm dist}\,(x,\Gamma_{j,k}^\star)
\geq j^{-10}$. 
\end{lemma}
\begin{proof}
By the discussion above, $A_{j,k} \cap \Gamma_{j,k}^\star = A_{j,k} \cap \Gamma_{j,k-1}$. By \eqref{ind2}, $A_{j,k}\cap \Gamma_{j,k-1}\setminus K_j=\emptyset$. If $x\in A_{j,k}\cap 
\Gamma_{j,k-1}\cap K_j$, then $x\in K_j\cap \Gamma_{j,k-1}\setminus D_{j,k-1}$. Then
by \eqref{ind1}, ${\rm dist}\,(x,\Gamma_0)<(k-1)\exp(-j^{\sfrac18}) \leq 1/(4\, j^{\sfrac14})$, where the last inequality follows from $k \leq j\, 2^{p_j} \leq 2\, j\, \eps_j^{-\kappa}$ and the choice of $\eps_j$. By \eqref{ind4}, we need to 
have some $\tilde x\in \Gamma_0\cap D_j$ such that $|x-\tilde x|<(k-1)\exp(-j^{\sfrac18}) $. 
On the other hand, by the definitions of $D_{j,k-1}$ and $D_j$, $|x-\tilde x|\geq {\rm dist}(A_{j,k},D_j)>1/j^{1/4}$, and we have reached a
contradiction. Thus the first claim follows. For the second claim, such point $x$
satisfies ${\rm dist}\,(x, \Gamma_0\setminus D_j)=1/(2j^{1/4})$.
If there exists 
$\tilde x\in \Gamma_{j,k}^\star$ with $|x-\tilde x|< j^{-10}$, then $\tilde x \in \Gamma_{j,k-1}$, and
${\rm dist}\,(\tilde x,\Gamma_0\setminus D_j)<1/(2j^{1/4})+j^{-10}$, so that $\tilde x\in  K_j\cap
\Gamma_{j,k-1}\setminus D_{j,k-1}$. By \eqref{ind1}, ${\rm dist}\,(\tilde x,\Gamma_0)
\leq (k-1)\exp(-j^{\sfrac18})$ and thus ${\rm dist}\,(x,\Gamma_0)<j^{-10}+(k-1)\exp(-j^{\sfrac18})$. Since
${\rm dist}\,(x,\Gamma_0\setminus D_j)=1/(2j^{1/4})$, this shows
that there exists $\hat x\in \Gamma_0\cap D_j$ such that $|\hat x- x|<j^{-10}+(k-1)\exp(-j^{\sfrac18}) \leq 1/(2\, j^{1/4})$. 
On the other hand, ${\rm dist}\,(\partial (D_{j,k-1})_{j^{-10}},D_j)>1/(j^{1/4})$,
which is a contradiction. Thus we have the second claim. 
\end{proof}

\smallskip

Next, for each point $x\in \partial(D_{j,k-1})_{j^{-10}}$, let $r_0(x)\in \partial D_{j,k-1}$ be the nearest point projection of $x$ onto $\partial D_{j,k-1}$, and set $r_s(x):=sx+(1-s)r_0(x)$
for $s\in (0,1)$.
With this notation, define
\begin{equation*}
{\rm Ret}_{j,k}:=\{r_s(x)\,:\, x\in A_{j,k}, \,\,s\in (0,1)\}.
\end{equation*}
\begin{lemma}\label{ind6}
We have
$(D_{j,k-1})_{j^{-10}}\setminus (K_j\cup D_{j,k-1})\subset {\rm Ret}_{j,k}$.
\end{lemma}

\begin{proof}
For any point $\tilde x\in (D_{j,k-1})_{j^{-10}}\setminus (K_j\cup D_{j,k-1})$, there
exist $s\in (0,1)$ and $x\in \partial(D_{j,k-1})_{j^{-10}}$ such that 
$\tilde x=r_s(x)$. The condition $\tilde x\notin K_j$ means that ${\rm dist}\,(\tilde x,\Gamma_0\setminus
D_j)\geq 1/j^{1/4}$, and then ${\rm dist}\,(x,\Gamma_0\setminus D_j)\geq 1/j^{1/4}
-j^{-10}$. Thus $x\in A_{j,k}$ and $\tilde x\in {\rm Ret}_{j,k}$. 
\end{proof}

The set $A_{j,k}$ is a relatively open subset of $\partial (D_{j,k-1})_{j^{-10}}$. Let 
$A_{j,k,l}\subset A_{j,k}$ be any of the (at most countably many) connected components of $A_{j,k}$ and
define
\begin{equation*}
{\rm Ret}_{j,k,l}:=\{r_s(x)\,:\, x\in A_{j,k,l},\,\, s\in (0,1)\}.
\end{equation*}
\begin{lemma}
\label{supind6}
We have $(A_{j,k,l}\cup (\partial A_{j,k,l})_{j^{-10}})\cap \Gamma_{j,k}^\star=\emptyset$.
\end{lemma}
\begin{proof} The claim follows directly from Lemma \ref{ind5}.
\end{proof}
Lemma \ref{supind6} implies that for each $l$
there exists some $i(l)\in\{1,\ldots,N\}$ such that $E_{j,k,i(l)}^{\star}$ contains $A_{j,k,l} \cup (\partial A_{j,k,l})_{j^{-10}}$. 
For each index $l$, let $i(l)$ be this correspondence. 
We define for each $i=1,\ldots,N$
\begin{equation*}
\tilde E_{j,k,i}:=E_{j,k,i}^{\star}\cup (\cup_{i(l)=i} {\rm Ret}_{j,k,l}).
\end{equation*}
In other words, when $A_{j,k,l}\cup (\partial A_{j,k,l})_{j^{-10}}$ is 
contained in $E_{j,k,i(l)}^{\star}$ with $i(l)=i$, then we replace the open partitions inside 
${\rm Ret}_{j,k,l}$ by $\tilde E_{j,k,i}$. 
For the resulting open
partition $\tilde\E_{j,k} := \{\tilde E_{j,k,i}\}_{i=1}^N \in \op^N(U_{j,k-1})$, define $\tilde \Gamma_{j,k}:=U_{j,k-1}\cap
\cup_{i=1}^N\partial \tilde E_{j,k,i}$.
\begin{lemma}
We have 
\begin{equation}
\label{supind6aeq}
\tilde \Gamma_{j,k}\setminus K_j\subset D_{j,k-1}
\end{equation}
and
\begin{equation}
\label{supind6aeq2}
\tilde\Gamma_{j,k}\setminus D_{j,k-1}= \Gamma_{j,k}^\star\setminus (D_{j,k-1}\cup
{\rm Ret}_{j,k}) 
=\Gamma_{j,k-1}\setminus (D_{j,k-1}\cup
{\rm Ret}_{j,k}).
\end{equation}
\label{supind6a}
\end{lemma}
\begin{proof}
Note that $\tilde \Gamma_{j,k}\cap \overline{{\rm Ret}_{j,k}}\setminus D_{j,k-1}=\emptyset$
since $\partial {\rm Ret}_{j,k}\setminus D_{j,k-1}$ is contained in some open partition by
Lemma \ref{supind6}
and $\tilde\Gamma_{j,k}\cap {\rm Ret}_{j,k}=\emptyset$. If there exists 
$x\in \tilde\Gamma_{j,k}\setminus (K_j\cup D_{j,k-1})$, then $x\notin \overline{{\rm Ret}_{j,k}}$ and thus $x\in \Gamma_{j,k}^{\star}\setminus (K_j\cup D_{j,k-1}) =  \Gamma_{j,k-1}\setminus (K_j\cup D_{j,k-1})$. By \eqref{ind2}, 
$x\in (D_{j,k-1})_{j^{-10}}\setminus(K_j\cup D_{j,k-1})$. By Lemma \ref{ind6}, 
$x\in {\rm Ret}_{j,k}$, which is a contradiction. This proves the first claim. The 
second claim follows from the definition of $\tilde\Gamma_{j,k}$, in the sense that the new partition has no boundary in ${\rm Ret}_{j,k}$, while $\Gamma_{j,k}^\star \setminus
(D_{j,k-1}\cup{\rm Ret}_{j,k})$ is kept intact. The identity in \eqref{starinv} is also used to
obtain the last equality.
\end{proof}
\begin{lemma} \label{l:mass estimate}
For any $\phi \in \cR_j$ we have:
\begin{equation} \label{e:mass estimate after retraction}
\int_{\tilde \Gamma_{j,k}}\phi\, d\mathcal H^n
\leq \int_{\Gamma_{j,k}^\star}\phi\,d\mathcal H^{n}\,.
\end{equation}
\end{lemma}
\begin{proof}
Note that $\tilde\Gamma_{j,k} \triangle  \Gamma_{j,k}^\star \subset (\partial D_{j,k-1}\cap \overline{{\rm Ret}_{j,k}})\cup {\rm Ret}_{j,k}$, and that $\tilde \Gamma_{j,k} \cap {\rm Ret}_{j,k} = \emptyset$. Let 
${\rm Ret}_{j,k,l}$ and $E_{j,k,i(l)}^{\star}$ be as before. For any
$x\in\tilde\Gamma_{j,k}\cap \overline{{\rm Ret}_{j,k,l}}\subset\partial D_{j,k-1}$, 
consider $\tilde x\in \partial (D_{j,k-1})_{j^{-10}}$ such that $r_0(\tilde x)=x$. 
Note that $\tilde x=r_1(\tilde x)\in E_{j,k,i(l)}^{\star}$. If $r_s(\tilde x)\notin \Gamma_{j,k}^\star$ for all
$s\in [0,1)$, then $r_0(\tilde x)=x\in E_{j,k,i(l)}^{\star}$ and we have $x\in \tilde E_{j,k,i(l)}$,
which is a contradiction to $x\in \tilde\Gamma_{j,k}$. Thus there exists $s\in [0,1)$ such that $r_s(\tilde x)
\in \Gamma_{j,k}^\star$. In particular, we see that $\tilde\Gamma_{j,k} \cap\overline{{\rm Ret}_{j,k}}$ is in the image of $\Gamma_{j,k}^\star\cap 
\overline{{\rm Ret}_{j,k}}$ through the normal nearest point projection onto $\partial D_{j,k-1}$. Furthermore, since $r_s(\tilde x) = x + s \, \abs{\tilde x - x} \, \nu_{U}(x)$, and since $\phi$ is $\nu_{U}$-non decreasing in $\R^{n+1} \setminus D_j$, it holds $\phi(x) \leq \phi(r_s(\tilde x))$. Given that the normal nearest point projection onto $\pa D_{j,k-1}$ is a
Lipschitz map with Lipschitz constant $=1$, the desired estimate follows from the area formula. 
\end{proof}

Note that, as a corollary of Lemma \ref{l:mass estimate}, we have that, setting $\tilde\E_{j,k} = \{ \tilde E_{j,k,i}  \}_{i=1}^N$,
\begin{equation} \label{mass estimate step 2}
\| \partial \tilde \E_{j,k} \|(\R^{n+1}) \leq  \| \partial \E_{j,k}^\star \|(\R^{n+1}) \,.
\end{equation}

\smallskip

{\bf Step 3: motion by smoothed mean curvature with boundary damping.} Let $\tilde V_{j,k} = \partial \tilde{\E}_{j,k}$ as defined in \eqref{e:interior boundary}, and compute $h_{\eps_j}(\cdot):=h_{\eps_j}
(\cdot,\tilde V_{j,k})$. Also, let $\eta_j \in \cA_{j^{\sfrac34}}$ be the cut-off function defined in Definition \ref{def:etaj}. Observe that $j$ has been chosen so that the conclusions of Lemma \ref{l:etaj} hold. Define the smooth diffeomorphism $f_{j,k}(x):=x+\eta_j(x)\,h_{\eps_j}(x)\,\Delta t_j$. Observe that the induction hypothesis \eqref{indmass}, together with \eqref{mass estimate step 1} and \eqref{mass estimate step 2}, implies that $\|\tilde V_{j,k} \|(\R^{n+1}) \leq M$ as defined in \eqref{def of M}. Hence, by Lemma \ref{l:etaj}, and using \eqref{e:h in L infty} and the definition of $\Delta t_j$, we can conclude that $\abs{\eta_j \, h_\eps\, \Delta t_j} \leq \exp(-j^{\sfrac18})$ on $\tilde K_j$. By the choice of $\eps_j$, we also have that $|\eta_j\, h_\eps\, \Delta t_j|\leq j^{-10}$ everywhere. 

Set $U_{j,k}:=f_{j,k}(U_{j,k-1})$, $E_{j,k,i}:=f_{j,k}(\tilde E_{j,k,i})$ 
and $\Gamma_{j,k}:=U_{j,k}\cap\cup_{i=1}^N\partial E_{j,k,i}$. 

\begin{lemma} \label{l:step1}
We have
\[
\partial U_{j,k} \subset \left( \partial U \right)_{k\, \exp(-j^{\sfrac18})}\, \quad \mbox{and} \quad U_{j,k} \triangle U \subset \left( \pa U  \right)_{k \, \exp(-j^{\sfrac18})}\,,
\]
namely \eqref{ind0} with $k$ in place of $k-1$ holds true.
\end{lemma}

\begin{proof}
Since $|x-f_{j,k}(x)|\leq \eta_j |h_{\varepsilon_j}|\Delta t_j\leq \exp(-j^{\sfrac18})$ on $K_j$ by Lemma \ref{l:etaj}(2), 
we see with \eqref{ind0} that $f_{j,k}(K_j\cap (\partial U_{j,k-1}\cup U_{j,k-1}\triangle U))\subset (\partial U)_{k\exp(-j^{\sfrac18})}$.
In order to show that also $f_{j,k}((\partial U_{j,k-1}\cup U_{j,k-1}\triangle U)\setminus K_j) \subset (\pa U)_{k\exp(-j^{\sfrac18})}$, 
we next claim that 
\begin{equation} \label{claim:boundary}
\min\{\dist(\partial U_{j,k-1} \setminus K_j, \tilde\Gamma_{j,k})\,, \; \dist( (U_{j,k-1}   \triangle U) \setminus K_j, \tilde\Gamma_{j,k})   \} \geq 1/(4\,j^{\sfrac14})\,.
\end{equation}
To see this, let $x \in (\partial U_{j,k-1} \cup (U_{j,k-1} \triangle U)) \setminus K_j$ and $y \in \tilde \Gamma_{j,k}$. Since $x \in \partial U_{j,k-1} \cup (U_{j,k-1} \triangle U)$, by \eqref{ind0} there is $\tilde x \in \partial U$ such that $\abs{x - \tilde x} \leq (k-1) \exp(-j^{\sfrac18})$. Now, if $y \notin K_j$, then by Lemma \ref{supind6a}, $y\in D_{j,k-1}$. By the definition of $D_{j,k-1}$, $\abs{x - y} \geq \abs{y-\tilde x}-\abs{\tilde x-x}\geq 1/j^{\sfrac14} - 2(k-1)\exp(-j^{\sfrac18})$, so that $\abs{x-y} \geq 1/(4\,j^{\sfrac14})$. The same conclusion clearly holds if $y \in D_{j,k-1}$. Finally, if $y \in K_j \setminus D_{j,k-1}$ then, 
by \eqref{supind6aeq2}, $y\in \Gamma_{j,k-1}\cap K_j\setminus D_{j,k-1}$. 
Then by \eqref{ind1}, $y\in (\Gamma_0)_{(k-1)\exp(-j^{\sfrac18})}\setminus 
D_{j,k-1}$. By the definition of $K_j$, we have $|x-y|\geq j^{-\sfrac14}-
(k-1)\exp(-j^{\sfrac18})>1/(4j^{\sfrac14})$. This proves \eqref{claim:boundary}. 
For any point $x\notin (\tilde \Gamma_{j,k})_{1/4j^{\sfrac14}}$, note that 
\begin{equation*}
|h_{\eps_j}(x,\tilde V_{j,k})|\leq \eps_j^{-1}\int_{\tilde\Gamma_{j,k}}
|\nabla \Phi_{\eps_j}(x-y)|\,d\Ha^n (y)\leq M\exp(-1/\eps_j)<\exp(-j^{\sfrac18})
\end{equation*}
for all sufficiently large $j$. This shows that $f_{j,k}((\partial U_{j,k-1}
\cup U_{j,k-1}\triangle U)\setminus K_j)\subset (\partial U)_{k\exp(-j^{
\sfrac18})}$ and concludes the proof. 

\end{proof}
\begin{lemma} We have
\[f_{j,k} (D_{j,k-1})\cap (K_j\setminus D_{j,k})=\emptyset.\]
\label{inc1}
\end{lemma}
\begin{proof}
Suppose, towards a contradiction, that $x \in f_{j,k} (D_{j,k-1})\cap (K_j\setminus D_{j,k})$.
Since $|\Delta t_j\eta_j h_{\eps_j}|\ll 1/j^{1/4}$ for all points, $\hat x:=f_{j,k}^{-1}(x)$ is in 
$\tilde K_j$ in particular. Then, $|\eta_j(\hat x)\, h_{\eps_j}(\hat x)\,\Delta t_j|\leq \exp(-j^{\sfrac18})$.
This means that $|x-\hat x|\leq \exp(-j^{\sfrac18})$. Since $x\notin D_{j,k}$, we need to have
$\hat x\notin D_{j,k-1}$ by the definition of these sets. But this is a contradiction 
since $x=f_{j,k}(\hat x) \in f_{j,k}(D_{j,k-1})$ and $f_{j,k}$ is bijective. 
\end{proof}
\begin{lemma} \label{l:step2}
We have
\begin{equation}
\label{ind3}
(\Gamma_{j,k}\cap K_j)\setminus D_{j,k}\subset (\Gamma_0)_{k\exp(-j^{\sfrac18})}\,,
\end{equation}
namely \eqref{ind1} with $k$ in place of $k-1$ holds true.
\end{lemma}
\begin{proof}
For any $x\in (\Gamma_{j,k}
\cap K_j)\setminus D_{j,k}$, by Lemma \ref{inc1}, $x\notin f_{j,k}(D_{j,k-1})$ and there exists $\hat x\in \tilde\Gamma_{j,k}\setminus D_{j,k-1}$ such that $f_{j,k}(\hat x)
=x$. By \eqref{supind6aeq} and \eqref{supind6aeq2}, $\hat x\in (\Gamma_{j,k}^\star \cap K_j) \setminus D_{j,k-1} = (\Gamma_{j,k-1}\cap K_j) \setminus D_{j,k-1}$. 
By \eqref{ind1}, $\hat x\in (\Gamma_0)_{(k-1)\exp(-j^{\sfrac18})}$; on the other hand, $\hat x\in K_j$ implies
$|x-\hat x|\leq \exp(-j^{\sfrac18})$. These two estimates together prove \eqref{ind3}. 
\end{proof}

\begin{lemma} \label{incsn}
We have
\begin{equation} \label{l:step3}
\Gamma_{j,k} \setminus K_j \subset (D_{j,k})_{j^{-10}}\,,
\end{equation}
namely \eqref{ind2} with $k$ in place of $k-1$ holds true.
\end{lemma}
\begin{proof}
If $x \in \Gamma_{j,k} \setminus K_j$, then there is $\tilde x \in \tilde \Gamma_{j,k}$ such that $x = f_{j,k}(\tilde x)$. If $\tilde x \notin K_j$, then $x \in D_{j,k-1} \subset D_{j,k}$ by Lemma \ref{supind6a}, and since $\abs{x- \tilde x} < j^{-10}$ by the properties of the diffeomorphism $f_{j,k}$ our claim holds true. Hence, suppose that $\tilde x \in K_j$. Since in this case $\abs{x - \tilde x} \leq \exp(-j^{\sfrac18})$, if $\tilde x \in D_{j,k-1}$ then evidently $x \in D_{j,k}$, and the proof is complete. On the other hand, we claim that it has to be $\tilde x \in D_{j,k-1}$. Indeed, otherwise we would have $\tilde x \in \tilde \Gamma_{j,k} \cap K_j \setminus D_{j,k-1}$, and thus, again by Lemma \ref{supind6a}, $\tilde x \in \Gamma_{j,k}^\star \cap K_j \setminus D_{j,k-1} = \Gamma_{j,k-1} \cap K_j \setminus D_{j,k-1}$. But then, by \eqref{ind1}, there exists $y \in \Gamma_0$ such that $\abs{\tilde x-y} < (k-1)\, \exp(-j^{\sfrac18})$. 
Since $\tilde x \notin D_{j,k-1}$, we have $y \notin D_j$, and therefore $\dist(x, (\Gamma_0 \setminus D_j)) 
\leq |x-\tilde x|+|\tilde x-y|< k\, \exp(-j^{\sfrac18}) < 1/j^{\sfrac14}$. But this contradicts the fact that $x \notin K_j$ and completes the proof.  
\end{proof}

\smallskip

{\bf Conclusion.} Together, Lemmas \ref{l:step1}, \ref{l:step2} and \ref{incsn} complete the induction step from $k-1$ to $k$ for properties (1), (2), (3). Concerning \eqref{induction:mass estimate}, first we observe that, since $f_{j,k}$ is a diffeomorphism,
\begin{equation} \label{e:partition after motion}
\pa \E_{j,k} = \var\left( \bigcup_{i=1}^N  (U_{j,k} \cap \pa E_{j,k,i})\,, \; 1 \right) = \var\left(  f_{j,k}\Big(  \bigcup_{i=1}^N (U_{j,k-1} \cap \pa \tilde E_{j,k,i})  \Big)  \,,  \; 1 \right) = (f_{j,k})_\sharp \pa \tilde \E_{j,k}\,.
\end{equation}
We can then use \eqref{e:smc2} with $V = \pa  \tilde \E_{j,k}$, $M$ as defined in \eqref{def of M}, $\eps = \eps_j$, and $\Delta t = \Delta t_j$ in order to conclude that
\begin{equation} \label{mass estimate step 3}
\| \pa \E_{j,k} \|(\R^{n+1}) \leq 2\, \Delta t_j \, \eps_j^{\sfrac14} + \| \pa \tilde \E_{j,k} \|(\R^{n+1})\,.
\end{equation}
Combining \eqref{mass estimate step 3} with \eqref{mass estimate step 1} and \eqref{mass estimate step 2}, and using that $2 \, \eps_{j}^{\sfrac14} < \eps_{j}^{\sfrac16} $, we get
\begin{equation} \label{e:key mass bound}
\| \pa \E_{j,k} \|(\R^{n+1}) \leq \| \pa \E_{j,k-1} \|(\R^{n+1}) + \Delta t_j \, \eps_j^{\sfrac16}\,,
\end{equation}
which, together with \eqref{indmass}, gives \eqref{induction:mass estimate}. Last, we show that the construction of the induction step satisfies \eqref{induction:mean curvature} and \eqref{induction:mass variation}. Since $\eps_j$ satisfies \eqref{e:eps_smallness} and \eqref{induction:mass estimate} implies 
$\| (f_{j,k})_\sharp \pa \tilde \E_{j,k} \|(\R^{n+1})\leq M$, so that the estimates in \eqref{e:smc3} and \eqref{e:smc4} hold true. Then \eqref{induction:mean curvature} follows from \eqref{e:smc2}, \eqref{e:smc4}, \eqref{mass estimate step 2}
and \eqref{e:almost minimizing}.
Finally, \eqref{induction:mass variation} is a consequence of \eqref{e:smc1}, \eqref{e:smc3}, \eqref{e:mass estimate after retraction} and \eqref{mass estimate step 1}.
\end{proof}

We are now in a position to define an approximate flow of open partitions. As anticipated in the introduction, the flow is piecewise constant in time; the parameter $\Delta t_j$ defined in \eqref{d:time step} is the \emph{epoch length}, namely the length of the time intervals in which the flow is set to be constant.

\begin{definition} 
For every $j \geq \max\{j_0,J(n)\}$, define a family $\E_j(t)$ for $t \in \left[ 0, j \right]$ by setting
\[
\E_j (t) := \E_{j,k}\quad \mbox{if $t \in \left( (k-1) \, \Delta t_j, k \, \Delta t_j \right]$}\,.
\]
\end{definition}

\subsection{Convergence in the sense of measures}

\begin{proposition} \label{p:limit_measure}

Under the assumptions of Proposition \ref{p:induction}, there exist a subsequence $\{j_{\ell} \}_{\ell=1}^{\infty}$ and a one-parameter family of Radon measures $\{\mu_t\}_{t \geq 0}$ on $U$ such that 
\begin{equation} \label{e:limit_measure}
\mu_t(\phi) = \lim_{\ell \to \infty} \| \pa \E_{j_\ell}(t) \| (\phi)
\end{equation}
for all $\phi \in C_{c}(U)$ and $t\in \mathbb R^+$. The limits $\lim_{s\to t+} \mu_s(\phi)$ and $\lim_{s\to t-}\mu_s(\phi)$ exist and satisfy
\begin{equation}\label{muconti}
\lim_{s\to t+} \mu_s(\phi)\leq \mu_t(\phi)\leq \lim_{s\to t-}\mu_s(\phi)
\end{equation}
for all $\phi \in C_c(U;\mathbb R^+)$ and $t\in \mathbb R^+$. Furthermore, $\lim_{s\to t+} \mu_s(\phi)=\lim_{s\to t-}\mu_s(\phi)$ for all $t\in\mathbb R^+\setminus B$, where $B \subset \R^+$ is countable. Finally, for every $T > 0$ we have
\begin{equation} \label{finite total mean curvature}
\limsup_{\ell \to \infty} \int_{0}^T \left( \int_{\R^{n+1}} \eta_{j_\ell} \, \frac{\abs{\Phi_{\eps_{j_\ell}} \ast \delta (\pa \E_{j_\ell}(t))}^2}{\Phi_{\eps_{j_\ell}} \ast \| \pa \E_{j_\ell}(t) \| + \eps_{j_\ell}} \, dx  - \frac{1}{\Delta t_{j_\ell}} \, \Delta_{j_\ell} \| \pa \E_{j_\ell}(t) \| (D_{j_\ell})  \right) \, dt < \infty\,,
\end{equation}
and for a.e. $t \in \R^+$ it holds
\begin{equation} \label{decay of mass reduction}
\lim_{\ell \to \infty} j_\ell^{2(n+1)} \, \Delta_{j_\ell} \| \pa \E_{j_\ell}(t) \|(D_{j_\ell}) = 0\,.
\end{equation}
\end{proposition}

\begin{proof}

Let $2_\Q$ be the set of all non-negative numbers of the form $\frac{i}{2^j}$ for some $i,j \in \Na \cup \{0\}$. $2_\Q$ is countable and dense in $\R^+$. For each fixed $T \in \Na$, the mass estimate in \eqref{induction:mass estimate} implies that
\begin{equation} \label{e:precompactness}
\limsup_{j \to \infty} \sup_{t \in \left[ 0, T \right]} \| \pa \E_{j}(t) \| (\R^{n+1}) \leq \| \pa \E_0 \| (\R^{n+1})\,.
\end{equation}
Therefore, by a diagonal argument we can choose a subsequence $\{j_{\ell}\}$ and a family of Radon measures $\{ \mu_t \}_{t \in 2_\Q}$ on $\R^{n+1}$ such that
\begin{equation} \label{e:convergence 2Q}
\mu_t (\phi) = \lim_{\ell \to \infty} \| \pa \E_{j_\ell}(t) \| (\phi) \qquad \mbox{for every $\phi \in C_{c}(\R^{n+1})$, for every $t \in 2_\Q$}\,.
\end{equation}
Furthermore, with \eqref{e:precompactness}, we also deduce that
\begin{equation} \label{e:limit mass bound}
\mu_{t} (\R^{n+1}) \leq \| \pa \E_0 \|(\R^{n+1}) \qquad \mbox{for every $t \in 2_\Q$}\,.
\end{equation}

Next, let $Z := \{ \phi_q \}_{q \in \Na}$ be a countable subset of $C^2_{c}(U; \R^+)$ which is dense in $C_{c} (U; \R^+)$ with respect to the supremum norm. We claim that the function
\begin{equation} \label{monotone function}
t \in  2_\Q \mapsto g_{q}(t) := \mu_{t}(\phi_q) -  t\, \| \nabla^2 \phi_q \|_{\infty} \, \| \pa \E_0 \| (\R^{n+1})
\end{equation}
is monotone non-increasing. To see this, first observe that since $\phi_q$ has compact support, and since the definition in \eqref{monotone function} depends linearly on $\phi_q$, we can assume without loss of generality that $\phi_q < 1$. For convenience, for $t\leq 0$, we define $g_q(t):=\mu_0(\phi_q)=\|\pa\E_0\|(\phi_q)$. Next, given any $j \geq J(n)$ as in Proposition \ref{p:induction}, for every positive function $\phi$ such that $\eta_j \, \phi \in \cA_j$ we can compute
\begin{equation} \label{monotonicity estimate basic}
\begin{split}
\delta (\pa \E_j(t), \phi) (\eta_j \, h_{\eps_j}) &= \delta (\pa \E_j(t)) (\eta_j \, \phi \, h_{\eps_j}) + \int_{\bG_n(\R^{n+1})} \eta_j(x) \, h_{\eps_j} \cdot S^{\perp} (\nabla \phi (x)) \, d(\pa \E_{j}(t))(x,S) \\
&=: I_1 + I_2
\end{split}
\end{equation}
for every $t \in \left[ 0, j \right]$, and where $h_{\eps_j}(\cdot) = h_{\eps_j}(\cdot, \pa\E_{j}(t))$. By the choice of $\eps_j$, and since $\eta_j \, \phi \in \cA_j$, we can use \eqref{e:fv along h vs h in L2} to estimate
\begin{equation} \label{monotonicity estimate 1}
I_1 \leq \eps_j^{\sfrac14} - \left( 1 - \eps_{j}^{\sfrac14}  \right) \, \int_{\R^{n+1}} \eta_j \, \phi \, \frac{\abs{\Phi_{\eps_j} \ast \delta (\pa\E_{j}(t))}^2}{\Phi_{\eps_j} \ast \| \pa \E_{j}(t) \| + \eps_j} \, dx\,,
\end{equation}
whereas Young's inequality together with \eqref{e:L2 norm of h vs approx} yields
\begin{equation} \label{monotonicity estimate 2}
\begin{split}
I_2 &\leq  \frac12 \, \int_{\R^{n+1}} \eta_j \, \phi \, \abs{h_{\eps_j}}^2 \, d\|\pa \E_{j}(t)\| + \frac{1}{2} \, \int_{\R^{n+1}} \eta_j \, \frac{\abs{S^\perp(\nabla\phi)}^2}{\phi} \, d\| \pa \E_{j}(t) \| \\
&\leq \frac{\eps_j^{\sfrac14}}{2} + \left( \frac12 + \frac{\eps_j^{\sfrac14}}{2}  \right) \, \int_{\R^{n+1}} \eta_j \, \phi \, \frac{\abs{\Phi_{\eps_j} \ast \delta (\pa\E_{j}(t))}^2}{\Phi_{\eps_j} \ast \| \pa \E_{j}(t) \| + \eps_j} \, dx + \frac{1}{2} \, \int_{\R^{n+1}} \eta_j \, \frac{\abs{S^\perp(\nabla\phi)}^2}{\phi} \, d\| \pa \E_{j}(t) \|.
\end{split}
\end{equation}
Plugging \eqref{monotonicity estimate 1} and \eqref{monotonicity estimate 2} into \eqref{monotonicity estimate basic}, we obtain
\begin{equation} \label{monotonicity estimate final}
\delta (\pa \E_j(t), \phi) (\eta_j \, h_{\eps_j}) \leq 2\, \eps_{j}^{\frac14} + \frac12\, \int_{\R^{n+1}} \eta_j \, \frac{\abs{\nabla \phi}^2}{\phi} \, d\| \pa \E_j(t) \|
\end{equation}
for every $t \in \left[ 0, j \right]$ and for every positive function $\phi$ such that $\eta_j \, \phi \in \cA_j$. Now, for every $T \in \Na$, for every $\phi_q \in Z$ with $\phi_q < 1$, and for every sufficiently large $i \in \Na$, choose $j_* \geq \max\{ T, J(n)\}$ so that 
\begin{itemize}
\item[(i)] $\phi_q + i^{-1} \in \cA_{j} \cap \cR_{j}$,
\item[(ii)] $\eta_{j} \, (\phi_q + i^{-1}) \in \cA_{j}$
\end{itemize}
for every $j \geq j_*$. Using that $\eta_j \in \cA_{j^{\sfrac34}}$ for every $j \geq J(n)$ and that $\phi_q = 0$ outside some compact set $K \subset U$, it is easily seen that the two conditions above can be met by choosing $j_*$ sufficiently large, depending on $i$, $\|\phi_q\|_{C^2}$, and $K$. In particular, $j_*$ is so large that $\phi_q \equiv 0$ on $\left( \pa U \right)_{s_0}^{-} \setminus D_{j_*}$, so that $\phi_q + i^{-1}$ is trivially $\nu_{U}$-non decreasing in $\R^{n+1} \setminus D_{j_*}$ because it is constant in there. For any fixed $t_1,t_2 \in \left[ 0, T \right] \cap 2_\Q$ with $t_2 > t_1$, choose a larger $j_*$, so that both $t_1$ and $t_2$ are integer multiples of $1/2^{p_{j_*}}$. Then, both $t_2$ and $t_1$ are integer multiples of $\Delta t_{j_\ell}$ for every $j_\ell \geq j_*$. Hence, for every $j_\ell \geq j_*$ we can apply \eqref{induction:mass variation} repeatedly with $\phi = \phi_q + i^{-1} \in \cA_{j_\ell} \cap \cR_{j_\ell}$ and \eqref{monotonicity estimate final} again with $\phi = \phi_q + i^{-1}$ so that $\eta_{j_\ell} \, \phi \in \cA_{j_\ell}$ in order to deduce
\begin{equation} \label{towards monotonicity 1}
\begin{split}
&\| \pa \E_{j_\ell}(t_2) \| (\phi_q + i^{-1}) - \| \pa\E_{j_\ell}(t_1) \| (\phi_q + i^{-1})\\ &\qquad \qquad \qquad \leq \left( \eps_{j_\ell}^{\sfrac18} + 2\, \eps_{j_\ell}^{\sfrac14}  \right) (t_2 - t_1) + \frac12 \, \int_{t_1}^{t_2} \int_{\R^{n+1}} \eta_{j_\ell} \, \frac{\abs{\nabla \phi_q}^2}{\phi_q + i^{-1}} \, d\|\pa \E_{j_\ell}(t)\| \, dt\,.
\end{split}
\end{equation}
As we let $\ell \to \infty$, the left-hand side of \eqref{towards monotonicity 1} can be bounded from below, using \eqref{e:precompactness} and \eqref{e:convergence 2Q}, as follows:
\begin{equation} \label{lhs lower bound}
\geq \mu_{t_2}(\phi_q) - \mu_{t_1}(\phi_q) - i^{-1} \, \| \pa\E_0 \|(\R^{n+1})\,.
\end{equation}
In order to estimate the right-hand side of \eqref{towards monotonicity 1}, we note that
\begin{equation} \label{trick}
\frac{\abs{\nabla \phi_q}^2}{\phi_q + i^{-1}} \leq \frac{\abs{\nabla\phi_q}^2}{\phi_q} \leq 2\, \| \nabla^2 \phi_q \|_{\infty}\,,
\end{equation}
so that if we plug \eqref{trick} in \eqref{towards monotonicity 1}, use that $\eta_{j_\ell} \leq 1$, let $\ell \to \infty$ by means of \eqref{e:precompactness}, and finally let $i \to \infty$ we conclude
\begin{equation} \label{towards monotonicity 2}
\mu_{t_2}(\phi_q) - \mu_{t_1}(\phi_q) \leq \| \nabla^2 \phi_q \|_{\infty} \, \| \pa \E_0 \|(\R^{n+1}) \, (t_2 - t_1)
\end{equation}
   for every $t_1,t_2 \in \left[0, T \right] \cap 2_\Q$ with $t_2 > t_1$ and for any $\phi_q \in Z$ with $\phi_q < 1$, thus proving that the function defined in \eqref{monotone function} is indeed monotone non-increasing on $[0,T]$. Since $T$ is arbitrary, the same holds on $\mathbb R^+$. 
   
\smallskip   
   
   Define now
   \[
   B := \left\lbrace t \in \mathbb R^+ \, \colon \, \lim_{2_\Q \ni s \to t-} g_{q}(s) > \lim_{2_\Q \ni s \to t+} g_{q}(s) \quad \mbox{for some $q \in \Na$} \right\rbrace\,.
   \] 
   By the monotonicity of each $g_{q}$, $B$ is a countable subset of $\R^+$, and for every $t \in \R^+ \setminus (B \cup 2_\Q)$ we can define $\mu_t(\phi_q)$ for every $\phi_q \in Z$ by 
   \begin{equation} \label{e:mu family extended}
    \mu_t(\phi_q) := \lim_{2_\Q \ni s \to t} \left( g_{q}(s) + s\, \|\nabla^2 \phi_q\|_{\infty} \, \| \pa \E_0\|(\R^{n+1})  \right) = \lim_{2_\Q \ni s \to t} \mu_{s}(\phi_q)\,.
\end{equation}      
   
We claim that 
\begin{equation} \label{mu_t is the correct limit}
\exists \, \lim_{\ell \to \infty} \| \pa \E_{j_\ell}(t) \| (\phi_q) = \mu_t (\phi_q) \qquad \mbox{for every $t \in \R^+ \setminus (B \cup 2_\Q)$\mbox{ and }$\phi_q \in Z$}\,. 
\end{equation}   

Indeed, due to the definition of $\pa \E_{j_\ell}(t)$, there exists a sequence $\{t_\ell\}_{\ell=1}^{\infty} \subset 2_\Q$ with $t_\ell > t$ such that $\lim_{\ell \to \infty} t_\ell = t$ and $\pa \E_{j_\ell}(t) = \pa \E_{j_\ell}(t_\ell)$. For any $s \in 2_\Q$ with $s > t$, and for all suffciently large $\ell$ so that $s > t_\ell$, we deduce from \eqref{towards monotonicity 1} that
\begin{equation} \label{correct limit1}
\| \pa \E_{j_\ell}(s) \| (\phi_q + i^{-1}) \leq \| \pa \E_{j_\ell}(t_\ell) \| (\phi_q + i^{-1}) + {\rm O}(s-t)\,.
\end{equation}
    Taking the $\liminf_{\ell \to \infty}$ and then the $\lim_{i \to \infty}$ on both sides of \eqref{correct limit1} we obtain that
    \begin{equation} \label{correct limit2}
     \mu_{s}(\phi_q) \leq \liminf_{\ell\to\infty} \| \pa \E_{j_\ell}(t_\ell)\| (\phi_q) + {\rm O}(s-t)\,,
    \end{equation}
   so that when we let $s \to t+$ the definition of $\mu_t$ and the fact that $\pa \E_{j_\ell}(t_\ell) = \pa \E_{j_\ell}(t)$ yield
   \begin{equation} \label{correct limit3}
    \mu_t(\phi_q) \leq \liminf_{\ell \to \infty} \| \pa \E_{j_\ell}(t) \| (\phi_q) \,.
   \end{equation}
   An analogous argument provides, at the same time,
   \begin{equation} \label{correct_limit4}
    \limsup_{\ell \to \infty} \| \pa \E_{j_\ell}(t) \| (\phi_q) \leq \mu_t(\phi_q)\,,
   \end{equation}
   so that \eqref{correct limit3} and \eqref{correct_limit4} together complete the proof of \eqref{mu_t is the correct limit}. Since $Z$ is dense in $C_{c}(U ;\R^+)$, \eqref{mu_t is the correct limit} determines the limit measure uniquely, and the convergence holds for every $\phi \in C_c(U)$ at every $t \in \R^+ \setminus B$. On the other hand, since $B$ is countable we can extract a further subsequence of $\{ \pa \E_{j_\ell}(t)\}_{\ell=1}^{\infty}$ converging to a Radon measure $\mu_t$ in $U$ for every $t \geq 0$.
The continuity of $\mu_t(\phi)$ on $\mathbb R^+\setminus B$ follows from the definition of $B$ and
a density argument. The existence of limits and the inequalities \eqref{muconti} can be also deduced from \eqref{towards monotonicity 2}
in the case $\phi=\phi_q$, and by density for $\phi\in C_c(U;\mathbb R^+)$. This completes the proof of the first part of the statement.\\
   
   \smallskip
   
   The claim in \eqref{finite total mean curvature} follows from \eqref{induction:mean curvature}. Finally, \eqref{finite total mean curvature} implies that for each $T > 0$
\begin{equation} 
\lim_{\ell \to \infty} \int_{0}^T - j^{2(n+1)} \, \Delta_{j_\ell} \| \pa \E_{j_\ell}(t) \| (D_{j_\ell}) \, dt \lesssim \lim_{\ell \to \infty} j_{\ell}^{2(n+1)} \, \Delta t_{j_\ell} = 0\,,
\end{equation}   
where in the last identity we have used that 
\[
\Delta t_{j_\ell} \leq \eps_{j_\ell}^{\kappa} \ll j_{\ell}^{-2(n+1)}\,,
\]
given the definition of $\kappa$ and the fact that $\eps_j$ satisfies \eqref{e:eps_smallness}. The proof is now complete.
   \end{proof}

\section{Brakke's inequality, rectifiability and integrality of the limit} \label{sec:Brakke}

In the next proposition we deduce further information concerning the family $\{\mu_t\}_{t \geq 0}$ of measures in $U$ introduced in Proposition \ref{p:limit_measure}.

\begin{proposition} \label{p:integral varifold limit}
Let $\{ \pa \E_{j_\ell}(t)\}$ for $\ell \in \Na$ and $t \geq 0$, and $\{\mu_t\}$ for $t \geq 0$ be as in Proposition \ref{p:limit_measure} satisfying \eqref{e:limit_measure}, \eqref{finite total mean curvature} and \eqref{decay of mass reduction}. Then, we have the following.
\begin{enumerate}

\item For a.e. $t \in \R^+$ the measure $\mu_t$ is integral, namely
there exists an integral varifold $V_t \in \IV_n(U)$ such that $\mu_t = \|V_t\|$.

\item For a.e. $t \in \R^+$, if a subsequence $\{j_{\ell}'\}_{\ell=1}^\infty \subset \{j_\ell\}_{\ell=1}^\infty $ is such that 
\begin{equation} \label{hp:uniform bound on L2 mean curvature}
\sup_{\ell \in \Na} \int_{\R^{n+1}} \eta_{j_\ell'} \, \frac{\abs{\Phi_{\eps_{j_\ell'}} \ast \delta (\pa \E_{j_\ell'}(t))}^2}{\Phi_{\eps_{j_\ell'}} \ast \| \pa \E_{j_\ell'}(t)  \| + \eps_{j_\ell'}} \, dx < \infty\,,
\end{equation}
then $\pa \E_{j_\ell'}(t)$ converges to $V_t\in {\bf IV}_n(U)$ as varifolds in $U$ as $\ell \to \infty$, namely
\begin{equation} \label{p2 of limit varifold}
\lim_{\ell \to \infty} \pa \E_{j_\ell'}(t) (\varphi) = V_t(\varphi) \qquad \mbox{for every $\varphi \in C_{c}(\bG_{n}(U))$}\,.
\end{equation}

\item For a.e. $t \in \R^+$, $V_t$ has generalized mean curvature $h(\cdot, V_t)$ in $U$ which satisfies
\begin{equation} \label{e:lsc of L2 norm mean curvature}
\int_{U} \abs{h(\cdot, V_t)}^2 \, \phi \, d\|V_t\| \leq \liminf_{\ell \to \infty} \int_{\R^{n+1}} \phi \, \eta_{j_\ell} \, \frac{\abs{\Phi_{\eps_{j_\ell}} \ast \delta (\pa \E_{j_\ell}(t))}^2}{\Phi_{\eps_{j_\ell}} \ast \| \pa \E_{j_\ell}(t)  \| + \eps_{j_\ell}} \, dx<\infty 
\end{equation}
for any $\phi \in C_c(U;\R^+)$.

\end{enumerate}
\end{proposition} 

Before proving Proposition \ref{p:integral varifold limit}, we need to state two important results, which are obtained by suitably modifying \cite[Theorem 7.3 \& Theorem 8.6]{KimTone}, respectively.

\begin{theorem}[Rectifiability Theorem] \label{t:rectifiability}
Suppose that $\{U_{j_\ell}\}_{\ell=1}^{\infty}$ are open sets in $\R^{n+1}$, $\{\E_{j_\ell}\}_{\ell=1}^\infty$ are such that $\E_{j_\ell} \in \op^N(U_{j_\ell})$, and $\{\eps_{j_\ell}\}_{l=1}^{\infty} \subset \left(0,1\right)$. Suppose that they satisfy
\begin{enumerate}
\item $\pa U_{j_\ell} \subset \left( \pa U \right)_{1/(4\, j_\ell^{\sfrac14})}$ and $U_{j_\ell} \, \triangle \, U \subset \left( \pa U \right)_{1/(4\, j_{\ell}^{\sfrac14})}$,
\item $\lim_{\ell \to \infty} j_\ell^4\,\eps_{j_\ell} = 0$ and $j_\ell\leq \eps_{j_\ell}^{\sfrac16}/2$,
\item $\sup_{\ell \in \Na} \| \pa \E_{j_\ell} \|(\R^{n+1}) < \infty$,
\item $\liminf_{\ell \to \infty} \int_{\R^{n+1}} \eta_{j_\ell} \, \frac{\abs{\Phi_{\eps_{j_\ell}} \ast \delta (\pa \E_{j_\ell})}^2}{\Phi_{\eps_{j_\ell}} \ast \| \pa \E_{j_\ell} \| + \eps_{j_\ell}}\, dx < \infty$,
\item $\lim_{\ell \to \infty} \Delta_{j_\ell} \|\pa \E_{j_\ell} \|(D_{j_\ell}) = 0$. 
\end{enumerate}

Then, there exist a subsequence $\{j'_\ell\}_{\ell=1}^{\infty}\subset\{j_\ell\}_{\ell=1}^{\infty}$ and a varifold $V \in \V_n(\R^{n+1})$ such that $\pa \E_{j'_\ell} \to V$ in the sense of varifolds, $\spt\, \|V\| \subset {\rm clos}\,U$, and
\begin{equation} \label{e:lower density bound}
\theta^{*n}(\|V\|,x) \geq c_0 > 0 \qquad \mbox{for $\|V\|$ a.e. $x\in U$}\,.
\end{equation}

Here, $c_0$ is a constant depending only on $n$. Furthermore, $V \mres \bG_n(U) \in \RV_n(U)$.

\end{theorem} 

\begin{proof}
The existence of a subsequence $\{  \pa \E_{j'_\ell} \}_{\ell=1}^{\infty}$ converging in the sense of varifolds to $V \in \V_n(\R^{n+1})$ follows from the compactness theorem for Radon measures using assumption (3). The limit varifold $V$ satisfies $\spt\|V\| \subset {\rm clos}\,U$ because of assumption (1). Indeed, since $\spt\| \pa\E_{j_\ell}\| \subset {\rm clos}\,U_{j_\ell}$ by definition of open partition, if $x \in \R^{n+1} \setminus {\rm clos}\,U$ then (1) implies that there is a radius $r > 0$ such that $\| \pa \E_{j'_\ell}\| (U_r(x)) = 0$ for all sufficiently large $\ell$, which in turn gives $\|V\|(U_r(x)) = 0$. Furthermore, the validity of (2), (3), and (4) allows us to apply Proposition \ref{p:prop56} in order to deduce that $\| \delta V \| \mres U$ is a Radon measure. Hence, the rectifiability of the limit varifold in $U$ is a consequence of Allard's rectifiability theorem \cite[Theorem 5.5(1)]{Allard} once we prove \eqref{e:lower density bound}. In turn, the latter can be obtained by repeating \emph{verbatim} the arguments in \cite[Theorem 7.3]{KimTone}. Indeed, the proof in there is local, and for a given $x_0 \in U$ it can be reproduced by replacing $B_1(x_0)$ in \cite[Theorem 7.3]{KimTone} by $B_{\rho}(x_0)$ for sufficiently small $\rho>0$ and large $\ell$
so that $B_{\rho}(x_0)\subset D_{j'_\ell}$ and $\eta_{j'_\ell} = 1$ on $B_{\rho}(x_0)$.
\end{proof}

\begin{theorem}[Integrality Theorem]\label{t:integrality}
Under the same assumptions of Theorem \ref{t:rectifiability}, if the stronger
\begin{itemize}
\item[(5)'] $\lim_{\ell \to \infty} j_\ell^{2(n+1)} \, \Delta_{j_\ell} \| \pa \E_{j_\ell} \|(D_{j_\ell}) = 0$
\end{itemize}
holds, then there is a converging subsequence $\{\pa\E_{j'_\ell}\}_{\ell=1}^{\infty}$ such that the limit varifold $V$ satisfies $V \mres \bG_n(U) \in \IV_n(U)$.
\end{theorem}
Just like Theorem \ref{t:rectifiability}, the claim is local in nature and the proof is the same as 
\cite[Theorem 8.6]{KimTone}. 

\begin{proof}[Proof of Proposition \ref{p:integral varifold limit}]

First, observe that by \eqref{finite total mean curvature} and Fatou's lemma we have
\begin{equation} \label{vlim1}
\liminf_{\ell \to \infty} \int_{\R^{n+1}} \eta_{j_\ell} \, \frac{\abs{\Phi_{\eps_{j_\ell}} \ast \delta (\pa \E_{j_\ell}(t))}^2}{\Phi_{\eps_{j_\ell}} \ast \| \pa \E_{j_\ell} \| + \eps_{j_\ell}} \, dx < \infty
\end{equation}
for a.e. $t \in \R^+$. Furthermore, from \eqref{induction:mass estimate} and the definition of $\pa \E_{j}(t)$ we also have that for every $T < \infty$
\begin{equation} \label{vlim2}
\sup_{\ell \in\Na} \sup_{t \in \left[ 0, T \right]} \| \pa \E_{j_\ell}(t) \| (\R^{n+1}) < \infty\,.
\end{equation}
Let $t \in \R^+$ be such that \eqref{vlim1} and \eqref{decay of mass reduction} hold. We want to show that the sequence $\{ \pa \E_{j_\ell}(t) \}_{\ell=1}^\infty$ satisfies the assumptions of Theorem \ref{t:integrality}. Assumption (1) follows from the construction of the discrete flow in Proposition \ref{p:induction} and the choice of $\eps_{j_\ell}$; (2) follows again from the choice of $\eps_{j_\ell}$, more precisely from \eqref{e:eps_smallness}; (3) and (4) are \eqref{vlim2} and \eqref{vlim1}, respectively; (5)' is \eqref{decay of mass reduction}. Hence, Theorem \ref{t:integrality} implies that, along a further subsequence $\{j_\ell'\}_{\ell=1}^\infty \subset \{j_\ell\}_{\ell=1}^\infty$, $\pa \E_{j_\ell'}(t)$ converges, as $\ell \to \infty$, to a varifold $V_t \in \V_n(\R^{n+1})$ with $\spt \|V_t\| \subset {\rm clos}\,U$ and such that $V_t \mres \bG_n(U) \in \IV_n(U)$. Since the convergence is in the sense of varifolds, the weights converge as Radon measures, and thus $\lim_{\ell \to \infty} \| \pa \E_{j_\ell'}(t) \| = \| V_t \|$: \eqref{e:limit_measure} then readily implies that $\| V_t \| \mres U = \mu_t$ as Radon measures on $U$, thus proving (1). Concerning the statement in (2), let $\{j_\ell'\}_{\ell=1}^\infty$ be a subsequence along which \eqref{hp:uniform bound on L2 mean curvature} holds. Then, any converging further subsequence must converge to a varifold satisfying the conclusion of Theorem \ref{t:integrality}. A priori, two distinct subsequences may converge to different limits. On the other hand, each subsequential limit $V_t$ is a rectifiable varifold when restricted to the open set $U$, and furthermore it satisfies $\|V_t\| \mres U = \mu_t$. Since rectifiable varifolds are uniquely determined by their weight, we deduce that the limit in $U$ is independent of the particular subsequence, and thus \eqref{hp:uniform bound on L2 mean curvature} forces the whole sequence $\pa \E_{j_\ell'}(t)$ to converge to a uniquely determined integral varifold $V_t$ in $U$. Finally, (3) follows from Proposition \ref{p:prop56}.
\end{proof}

A byproduct of the proof of Proposition \ref{p:integral varifold limit} is the existence of a (uniquely defined) integral varifold $V_t \in \IV_{n}(U)$ with weight $\|V_t\| = \mu_t$ for every $t \in \R^+ \setminus Z$, where $\mathcal{L}^1(Z) = 0$. Such a varifold $V_t$ is the limit on $U$ of any sequence $\pa \E_{j_\ell'}(t)$ along which \eqref{hp:uniform bound on L2 mean curvature} holds true. We can now extend the definition of $V_t$ to $t \in Z$ so to have a one-parameter family $\{V_t\}_{t \in \R^+} \subset \V_{n}(U)$ of varifolds satisfying $\|V_t\| = \mu_t$ for every $t \in \R^+$. Such an extension can be defined in an arbitrary fashion: for instance, if $t \in Z$ then we can set $V_t(\varphi) := \int \varphi(x,S) \, d\mu_t(x)$ for every $\varphi \in C_{c}(\bG_n(U))$, where $S$ is any constant plane in $\bG(n+1,n)$.

\medskip

In the next theorem, we show that the family of varifolds $\{V_t\}$ is indeed a Brakke flow in $U$. 
The boundary condition and the initial condition will be discussed in the following section. 

\begin{theorem}[Brakke's inequality] \label{t:Brakke inequality}
For every $T > 0$ we have
\begin{equation} \label{e:mean curvature bound}
\|V_T\|(U)+\int_0^T \int_{U} \abs{h(x,V_t)}^2 \, d\|V_t\|(x) \, dt \leq \Ha^n(\Gamma_0)\,.
\end{equation}
Furthermore, for any $\phi \in C^1_{c}(U \times \R^+ ; \R^+)$ and $0 \leq t_1 < t_2 < \infty$ we have:
\begin{equation} \label{e:Brakke}
\|V_t\|(\phi(\cdot, t))\Big|_{t=t_1}^{t_2} \leq \int_{t_1}^{t_2} \left( \delta (V_t, \phi(\cdot, t))(h(\cdot, V_t)) + \|V_t\|( \frac{\pa \phi}{\pa t}(\cdot, t) )   \right) \, dt\,. 
\end{equation}
\end{theorem}

\begin{proof}

In order to prove \eqref{e:mean curvature bound}, we use \eqref{induction:mass variation} with
$\phi=1$ which belongs to $\mathcal A_j\cap \mathcal R_j$ for all $j$. Assume $T\in 2_{\mathbb Q}$
first. By summing over the
index $k$ and for all sufficiently large $j$, we have
\begin{equation*}
\|\partial \mathcal E_{j}(T)\|(U)-\int_0^T \delta(\partial\mathcal E_{j}(t))
(\eta_{j} h_{\eps_j})\,dt\leq \Ha^n(\Gamma_0)+T\eps_{j}^{\sfrac18}.
\end{equation*}
By \eqref{e:fv along h vs h in L2} and \eqref{e:lsc of L2 norm mean curvature} as well as $\|V_T\|(U)\leq \liminf_{\ell\rightarrow\infty}\|\partial
\mathcal E_{j_\ell}(T)\|(U)$, we obtain \eqref{e:mean curvature bound}. For $T\notin 2_{\mathbb Q}$, 
use \eqref{muconti} to deduce the same inequality. 

\medskip

We now focus on proving the validity of Brakke's inequality \eqref{e:Brakke}. \\

\smallskip

{\bf Step 1.} We will first assume that $\phi$ is independent of $t$, and then extend the proof to the more general case. By an elementary density argument, we can assume that $\phi \in C^\infty_c(U; \R^+)$. Moreover, since the support of $\phi$ is compact and \eqref{e:Brakke} depends linearly on $\phi$, we can also normalize $\phi$ in such a way that $\phi < 1$ everywhere. Then, for all sufficiently large $i \in \mathbb N$, also $\hat \phi := \phi + i^{-1} < 1$ everywhere. Arguing as in the proof of Proposition \ref{p:limit_measure}, we can choose $m \in \mathbb N$ so that $m \geq J(n)$ (see Lemma \ref{l:etaj}) and furthermore
\begin{itemize}
\item[(i)] $\hat \phi \in \cA_j \cap \cR_j$,
\item[(ii)] $\eta_j \, \hat \phi \in \cA_j$
\end{itemize}
for all $j \geq m$. Next, fix $0 \leq t_1 < t_2 < \infty$, and let $\ell$ be such that $j_\ell \geq m$ and $j_\ell \geq t_2$, so that $\pa \E_{j_\ell}(t)$ is certainly well defined for $t \in \left[ t_1, t_2 \right]$. By the condition (i) above, we can apply \eqref{induction:mass variation} with $\hat \phi$ and deduce
\begin{equation} \label{Brakke1}
\| \pa \E_{j_\ell}(t)\|(\hat \phi) - \| \pa \E_{j_\ell} (t - \Delta t_{j_\ell}) \| (\hat \phi) \leq \Delta t_{j_\ell} \, \left(\delta (\pa \E_{j_\ell}(t), \hat \phi)(\eta_{j_\ell} \, h_{\eps_{j_\ell}}(\cdot, \pa \E_{j_\ell}(t))) + \eps_{j_\ell}^{\sfrac18}\right)
\end{equation}
for every $t = \Delta t_{j_\ell}, 2\, \Delta t_{j_\ell}, \ldots, j_\ell\, 2^{p_{j_\ell}} \, \Delta t_{j_\ell}$. Since $\Delta t_{j_\ell} \to 0$ as $\ell \to \infty$, we can assume without loss of generality that $\Delta t_{j_\ell} < t_2 - t_1$, so that there exist $k_1, k_2 \in \mathbb N$ with $k_1 < k_2$ such that $t_1 \in \left( (k_1 - 2) \, \Delta t_{j_\ell}, (k_1 - 1) \, \Delta t_{j_\ell} \right]$ and $t_2 \in \left( (k_2 - 1) \, \Delta t_{j_\ell}, k_2 \, \Delta t_{j_\ell} \right]$. If we sum \eqref{Brakke1} on $t = k \, \Delta t_{j_\ell}$ for $k \in \left[ k_1, k_2 \right] \cap \mathbb N$ we get
\begin{equation} \label{Brakke2}
\| \pa \E_{j_\ell}(t) \|(\hat \phi) \Big|_{t= (k_1 - 1) \, \Delta t_{j_\ell}}^{k_2\,\Delta t_{j_\ell}} \leq \sum_{k=k_1}^{k_2} \Delta t_{j_\ell} \, \left(\delta (\pa \E_{j_\ell}(k\,\Delta t_{j_\ell}), \hat \phi)(\eta_{j_\ell} \, h_{\eps_{j_\ell}}(\cdot, \pa \E_{j_\ell}(k\,\Delta t_{j_\ell}))) + \eps_{j_\ell}^{\sfrac18}\right)\,.
\end{equation}
Since $\hat \phi = \phi + i^{-1}$, we can estimate the left-hand side of \eqref{Brakke2} from below as
\begin{equation} \label{Brakke3}
\| \pa \E_{j_\ell}(t) \|(\hat \phi) \Big|_{t= (k_1 - 1) \, \Delta t_{j_\ell}}^{k_2\,\Delta t_{j_\ell}} \geq \| \pa \E_{j_\ell}(t_2)\| (\phi) - \| \pa \E_{j_\ell}(t_1) \|(\phi) - i^{-1} \| \pa \E_{j_\ell} (t_1) \|(\R^{n+1})\,,
\end{equation}
so that when we let $\ell \to \infty$ we conclude
\begin{equation} \label{Brakke4}
\limsup_{\ell \to \infty} \| \pa \E_{j_\ell}(t) \|(\hat \phi) \Big|_{t= (k_1 - 1) \, \Delta t_{j_\ell}}^{ k_2\,\Delta t_{j_\ell}} \geq \| V_{t} \|(\phi) \Big|_{t=t_1}^{t_2} - i^{-1}\, \| \pa \E_0 \|(\R^{n+1})\,,
\end{equation}
where we have used \eqref{e:limit_measure} together with Proposition \ref{p:integral varifold limit}(1).\\

\smallskip

Next, we estimate the right-hand side of \eqref{Brakke2} from above. Setting $\pa \E_{j_\ell} = \pa \E_{j_\ell}(t)$ and $h_{\eps_{j_\ell}} = h_{\eps_{j_\ell}}(\cdot, \pa \E_{j_\ell})$, we proceed as in \eqref{monotonicity estimate basic} writing
\begin{equation} \label{Brakke5}
\delta(\pa \E_{j_\ell}, \hat \phi)(\eta_{j_\ell}\, h_{\eps_{j_\ell}}) = \delta (\pa \E_{j_\ell})(\eta_{j_\ell}\, \hat \phi \, h_{\eps_{j_\ell}}) + \int_{\bG_{n}(\R^{n+1})} \eta_{j_\ell}\, S^\perp (\nabla \phi) \cdot h_{\eps_{j_\ell}}\, d(\pa \E_{j_\ell})\,,
\end{equation}
where we have used that $\nabla \hat\phi = \nabla \phi$. Since $\eta_{j_\ell}\, \hat\phi \in \cA_{j_\ell}$, we can apply \eqref{e:fv along h vs h in L2} in order to obtain that
\begin{equation} \label{Brakke6}
\abs{\delta (\pa \E_{j_\ell})(\eta_{j_\ell}\, \hat \phi \, h_{\eps_{j_\ell}}) + b_{j_\ell}   } \leq \eps_{j_\ell}^{\sfrac14} \left( b_{j_\ell} + 1 \right)\,,
\end{equation}
where we have set for simplicity
\begin{equation} \label{Brakke7}
b_{j_\ell} := \int_{\R^{n+1}} \eta_{j_\ell}\, \hat \phi \, \frac{\abs{\Phi_{\eps_{j_\ell}}  \ast  \delta (\pa \E_{j_\ell})  }^2}{\Phi_{\eps_{j_\ell}}  \ast \| \pa \E_{j_\ell}  \| + \eps_{j_\ell}  } \, dx\,.
\end{equation}
Concerning the second summand in \eqref{Brakke5}, we use the Cauchy-Schwarz inequality to estimate
\begin{equation} \label{Brakke8}
\begin{split}
\Abs{  \int_{\bG_{n}(\R^{n+1})} \eta_{j_\ell}\, S^\perp (\nabla \phi) \cdot h_{\eps_{j_\ell}}\, d(\pa \E_{j_\ell})  } &\leq \left( \int_{\R^{n+1}}  \eta_{j_\ell} \, \frac{\abs{\nabla \phi}^2}{\hat\phi}  \right)^{\sfrac12} \left( \int_{\R^{n+1}} \eta_{j_\ell} \, \hat\phi\, \abs{h_{\eps_{j_\ell}}}^2  \right)^{\sfrac12} \\
&\leq c \, \| \pa \E_{j_\ell} \| (\R^{n+1})^{\sfrac12}\, \left( (1+\eps_{j_\ell}^{\sfrac14})\, b_{j_\ell} + \eps_{j_\ell}^{\sfrac14} \right)^{\sfrac12}\,,
\end{split}
\end{equation}
where $c$ depends only on $\| \phi \|_{C^2}$, and where we have used \eqref{e:L2 norm of h vs approx}. Using \eqref{Brakke6}, \eqref{Brakke8} and \eqref{induction:mass estimate}, we can then conclude that
\begin{equation} \label{Brakke9}
\sup_{t \in \left[ t_1, t_2 \right]} \delta (\pa \E_{j_\ell}(t), \hat \phi) (\eta_{j_\ell} \, h_{\eps_{j_\ell}}(\cdot, \pa \E_{j_\ell}(t))) \leq c\,, 
\end{equation}
where $c$ depends only on $\| \phi \|_{C^2}$ and $\| \pa \E_0 \|(\R^{n+1})$. Using \eqref{Brakke9} together with the definition of $\pa \E_{j_\ell}(t)$ and Fatou's lemma, one can readily show that, when we take the $\limsup$ as $\ell \to \infty$, the right-hand side of \eqref{Brakke2} can be bounded by 
\begin{equation} \label{Brakke10}
\begin{split}
\limsup_{\ell \to \infty}&\sum_{k=k_1}^{k_2} \Delta t_{j_\ell} \, \left(\delta (\pa \E_{j_\ell}(k\,\Delta t_{j_\ell}), \hat \phi)(\eta_{j_\ell} \, h_{\eps_{j_\ell}}(\cdot, \pa \E_{j_\ell}(k\,\Delta t_{j_\ell}))) + \eps_{j_\ell}^{\sfrac18}\right) \\ 
&= \limsup_{\ell \to \infty} \int_{t_1}^{t_2} \delta (\pa \E_{j_\ell}(t), \hat\phi) (\eta_{j_\ell} \, h_{\eps_{j_\ell}}(\cdot, \pa \E_{j_\ell}(t))) \, dt \\
&\leq \int_{t_1}^{t_2} \limsup_{\ell \to \infty}  \delta (\pa \E_{j_\ell}(t), \hat\phi) (\eta_{j_\ell} \, h_{\eps_{j_\ell}}(\cdot, \pa \E_{j_\ell}(t)))   \, dt \,.
\end{split}
\end{equation}

Now, fix $t \in \left[ t_1, t_2 \right]$ such that $\liminf_{\ell\to \infty}b_{j_\ell}<\infty$ 
(which holds for a.e.~$t$), and let $\{j_\ell'\} \subset \{j_\ell\}$ be a subsequence which realizes the $\limsup$, namely with
\begin{equation} \label{Brakke11}
\lim_{\ell \to \infty}  \delta (\pa \E_{j_\ell'}(t), \hat\phi) (\eta_{j_\ell'} \, h_{\eps_{j_\ell'}}(\cdot, \pa \E_{j_\ell'}(t))) = \limsup_{\ell \to \infty}  \delta (\pa \E_{j_\ell}(t), \hat\phi) (\eta_{j_\ell} \, h_{\eps_{j_\ell}}(\cdot, \pa \E_{j_\ell}(t)))\,.
\end{equation}
By the identity in \eqref{Brakke5}, we also have that along the same subsequence
\begin{equation} \label{Brakke12}
\begin{split}
\lim_{\ell\to \infty} \Big( - \delta (\pa \E_{j_\ell'})&(\eta_{j_\ell'}\, \hat \phi \, h_{\eps_{j_\ell'}}) - \int_{\bG_{n}(\R^{n+1})} \eta_{j_\ell'}\, S^\perp (\nabla \phi) \cdot h_{\eps_{j_\ell'}}\, d(\pa \E_{j_\ell'})  \Big) \\
&= \liminf_{\ell \to \infty} \Big( - \delta (\pa \E_{j_\ell})(\eta_{j_\ell}\, \hat \phi \, h_{\eps_{j_\ell}}) - \int_{\bG_{n}(\R^{n+1})} \eta_{j_\ell}\, S^\perp (\nabla \phi) \cdot h_{\eps_{j_\ell}}\, d(\pa \E_{j_\ell}) \Big) \,,
\end{split}
\end{equation}
where once again $\pa \E_{j_\ell} = \pa \E_{j_\ell}(t)$ and $h_{\eps_{j_\ell}} = h_{\eps_{j_\ell}}(\cdot, \pa \E_{j_\ell})$. Using \eqref{Brakke6} and \eqref{Brakke8}, we see that the right-hand side of \eqref{Brakke12} can be bounded from above by $\liminf_{\ell \to \infty} 2\, b_{j_\ell} + c$, whereas the left-hand side can be bounded from below by $\limsup_{\ell\to \infty} \frac12\, b_{j_\ell'} - c$, where $c$ depends on $\| \phi \|_{C^2}$ and $\| \pa \E_0 \|(\R^{n+1})$. As a consequence, along any subsequence $\{j_\ell'\}$ satisfying \eqref{Brakke11} one has that 
\begin{equation} \label{Brakke13}
\limsup_{\ell \to \infty} \int_{\R^{n+1}} \eta_{j_\ell'}\, \hat \phi \, \frac{\abs{\Phi_{\eps_{j_\ell'}}  \ast  \delta (\pa \E_{j_\ell'})  }^2}{\Phi_{\eps_{j_\ell'}}  \ast \| \pa \E_{j_\ell'}  \| + \eps_{j_\ell'}  } \, dx \leq 4\, \liminf_{\ell \to \infty} \int_{\R^{n+1}} \eta_{j_\ell}\, \hat \phi \, \frac{\abs{\Phi_{\eps_{j_\ell}}  \ast  \delta (\pa \E_{j_\ell})  }^2}{\Phi_{\eps_{j_\ell}}  \ast \| \pa \E_{j_\ell}  \| + \eps_{j_\ell}  } \, dx + c<\infty\,,
\end{equation}
where $\pa\E_{j_\ell'} = \pa\E_{j_\ell'}(t)$. Let us denote the right-hand side of \eqref{Brakke13} as $B(t)$. Since $\hat \phi \geq i^{-1}$, and thanks to \eqref{Brakke13}, if $B(t) < \infty$ then the assumption 
\eqref{hp:uniform bound on L2 mean curvature} of Proposition \ref{p:integral varifold limit} is satisfied along $j_\ell'$: hence, the whole sequence $\{\pa \E_{j_\ell'}(t)\}_{\ell=1}^{\infty}$ converges to $V_t
\in \IV_n(U)$ as varifolds in $U$. Furthermore, using one more time that $\hat \phi \geq i^{-1}$ we deduce that
\begin{equation} \label{Brakke14}
\limsup_{\ell \to \infty} \int_{\R^{n+1}} \eta_{j_\ell'} \, \frac{\abs{\Phi_{\eps_{j_\ell'}}  \ast  \delta (\pa \E_{j_\ell'})  }^2}{\Phi_{\eps_{j_\ell'}}  \ast \| \pa \E_{j_\ell'}  \| + \eps_{j_\ell'}  } \, dx \leq i \, B(t)\,.
\end{equation}  
Using \eqref{Brakke11}, \eqref{Brakke5}, \eqref{Brakke6}, $\hat \phi > \phi$, and Proposition \ref{p:integral varifold limit}(3) with $\phi$ (recalling $\phi\in C_c^{\infty}(U;\mathbb R^+)$), we have
\begin{equation} \label{Brakke key1}
\begin{split}
\limsup_{\ell \to \infty}  \delta (\pa \E_{j_\ell}(t), \hat\phi) &(\eta_{j_\ell} \, h_{\eps_{j_\ell}}(\cdot, \pa \E_{j_\ell}(t))) = \lim_{\ell \to \infty} \delta (\pa \E_{j_\ell'}(t), \hat\phi) (\eta_{j_\ell'} \, h_{\eps_{j_\ell'}}(\cdot, \pa \E_{j_\ell'}(t))) \\
\leq &- \int_{U} \abs{h(\cdot, V_t)}^2 \, \phi \, d\|V_t\| \\  &+ \limsup_{\ell \to \infty} \int_{\bG_n(U)} S^\perp(\nabla \phi) \cdot h_{\eps_{j_\ell'}}(\cdot, \pa \E_{j_\ell'}(t)) \, d(\pa \E_{j_\ell'}(t))\,,
\end{split}
\end{equation}
where we have also used that, as $\ell \to \infty$, $\eta_{j_\ell'} = 1$ on $\{ \nabla \phi \neq 0 \} \ssubset U$.\\

\smallskip

Now, recall that $V_t \in \IV_n(U)$. Therefore, there is an $\Ha^n$-rectifiable set $M_t \subset U$ such that 
\begin{equation} \label{Brakke15}
\int_{\bG_n(U)} S^\perp(\nabla \phi(x)) \, dV_t(x,S) = \int_{U} T_{x}M_t^\perp(\nabla\phi(x)) \, d\|V_t\|(x)\,. 
\end{equation}
Furthermore, since the map $x \mapsto T_{x}M_t^\perp(\nabla\phi(x))$ is in $L^2(\|V_t\|)$, for any $\eps > 0$ there are a vector field $g \in C^\infty_c(U; \R^{n+1})$ and a positive integer $m'$ such that $g \in \cB_{m'}$ and 
\begin{equation} \label{Brakke16}
\int_{U} \abs{T_{x}M_t^\perp(\nabla\phi(x)) - g(x)}^2 \, d\|V_t\|(x) \leq \eps^2\,.
\end{equation}

In order to estimate the $\limsup$ in the right-hand side of \eqref{Brakke key1}, we can now compute, for $\pa\E_{j_\ell'} = \pa\E_{j_\ell'}(t)$:
\begin{equation} \label{Brakke add and subtract}
\begin{split}
\int_{\bG_n(U)} & S^\perp(\nabla \phi) \cdot h_{\eps_{j_\ell'}}(\cdot, \pa \E_{j_\ell'}) \, d(\pa \E_{j_\ell'}) \\
= &\int_{\bG_n(U)} (S^\perp(\nabla\phi) - g) \cdot h_{\eps_{j_\ell'}} \, d(\pa\E_{j_\ell'})\\
&+ \left(  \int_{U} g \cdot h_{\eps_{j_\ell'}} \, d\|\pa \E_{j_\ell'}\| + \delta (\pa\E_{j_\ell'})(g)  \right) - \delta (\pa\E_{j_\ell'})(g) + \delta V_t(g) \\ 
&+ \int_{U} h(\cdot, V_t) \cdot \left( g - T_{\cdot}M_t^\perp(\nabla \phi)  \right) \, d\|V_t\| \\
&+ \int_{\bG_{n}(U)} h(\cdot, V_t) \cdot S^\perp(\nabla\phi) \, dV_t(\cdot,S)\,.
\end{split}
\end{equation}

We proceed estimating each term of \eqref{Brakke add and subtract}.
Using that $\eta_{j_\ell'} = 1$ on $\{ \nabla\phi \neq 0 \}$ for all $\ell$ sufficiently large, the Cauchy-Schwarz inequality gives that
\begin{equation} \label{Brakke17}
\begin{split}
\Big|\int_{\bG_n(U)} (S^\perp(\nabla\phi) - g) & \cdot h_{\eps_{j_\ell'}} \, d(\pa\E_{j_\ell'})\Big| \\ & \leq \left(  \int_{\bG_{n}(U)} \abs{S^\perp(\nabla\phi) - g}^2 \, d(\pa\E_{j_\ell'})  \right)^{\frac12} \, \left( \int_{\R^{n+1}} \eta_{j_\ell'} \, \abs{h_{\eps_{j_\ell'}}}^2 \, d\| \pa\E_{j_\ell'}\|  \right)^{\frac12}
\end{split}
\end{equation}
for all $\ell$ sufficiently large. Since $(x,S) \mapsto \abs{S^\perp(\nabla\phi(x)) - g(x)}^2 \in C_{c}(\bG_n(U))$, we have that
\begin{equation} \label{Brakke18}
\begin{split}
\lim_{\ell\to\infty} \int_{\bG_{n}(U)} \abs{S^\perp(\nabla\phi) - g}^2 \, d(\pa\E_{j_\ell'}) &= \int_{\bG_n(U)} \abs{S^\perp(\nabla\phi) - g}^2 \, dV_t \\
&= \int_{U} \abs{T_xM_t^\perp (\nabla\phi(x)) - g(x)}^2 \, d\|V_t\|(x) \overset{\eqref{Brakke16}}{\leq} \eps^2\,.
\end{split}
\end{equation} 
Using \eqref{e:L2 norm of h vs approx}, \eqref{Brakke14}, \eqref{Brakke17} and \eqref{Brakke18}, we then conclude that
\begin{equation} \label{Brakke19}
\limsup_{\ell\to\infty}\Big|\int_{\bG_n(U)} (S^\perp(\nabla\phi) - g) \cdot h_{\eps_{j_\ell'}} \, d(\pa\E_{j_\ell'})\Big| \leq \left( i\, B(t) \right)^{\frac12} \, \eps\,.
\end{equation}

Analogously, since $\eta_{j_\ell'} = 1$ on $\{g \neq 0\}$ for all $\ell$ sufficiently large, we have that
\begin{equation} \label{Brakke20}
\lim_{\ell\to\infty} \Abs{\int_{U} g \cdot h_{\eps_{j_\ell'}} \, d\|\pa \E_{j_\ell'}\| + \delta (\pa\E_{j_\ell'})(g)} = \lim_{\ell\to\infty} \Abs{\int_{\R^{n+1}} \eta_{j_\ell'} \,g \cdot h_{\eps_{j_\ell'}} \, d\|\pa \E_{j_\ell'}\| + \delta (\pa\E_{j_\ell'})(\eta_{j_\ell'}\,g)} = 0
\end{equation}
by \eqref{e:prop55} and \eqref{Brakke14}. 

Next, by varifold convergence of $\pa\E_{j_\ell'}$ to $V_t$ on $U$, given that $g$ has compact support in $U$, we also have
\begin{equation} \label{Brakke21}
\lim_{\ell\to\infty} \abs{\delta(\pa\E_{j_\ell'})(g) - \delta V_t (g)} = 0\,.
\end{equation}

Finally, letting $\psi$ be any function in $C_c(U; \R^+)$ such that $\psi = 1$ on $\{g \neq 0 \} \cup \{\nabla\phi \neq 0\}$ and $0\leq \psi\leq 1$, the Cauchy-Schwarz inequality allows us to estimate
\begin{equation} \label{Brakke22}
\begin{split}
\Big|\int_{U} h(x, V_t) \cdot &\left( g(x) - T_xM_t^\perp(\nabla\phi(x))  \right) \, d\|V_t\| \Big|\\
&\leq \left( \int_{U} \abs{h(x,V_t)}^2\,\psi(x) \, d\|V_t\|(x) \right)^{\frac12} \, \left( \int_{U} \abs{g(x) - T_xM_t^\perp (\nabla\phi(x))}^2 \, d\|V_t\|(x)  \right)^{\frac12} \\
&\leq \left( i \, B(t) \right)^{\frac12} \, \eps\,,
\end{split}
\end{equation}
where in the last inequality we have used \eqref{e:lsc of L2 norm mean curvature} with $\psi$ in place of $\phi$, \eqref{Brakke14} and \eqref{Brakke16}.

From \eqref{Brakke add and subtract}, combining \eqref{Brakke19}-\eqref{Brakke22} we conclude that
\begin{equation} \label{Brakke ci siamo quasi}
\limsup_{\ell\to\infty} \int_{\bG_n(U)}   S^\perp (\nabla\phi) \cdot h_{\eps_{j_\ell'}}(\cdot, \pa\E_{j_\ell'}) \, d(\pa\E_{j_\ell'}) \leq \int_{U} h(\cdot, V_t) \cdot \nabla\phi \, d\|V_t\| + 2\, \left( i \, B(t) \right)^{\frac12} \, \eps\,,
\end{equation}
where we have also used \eqref{BPT}.

We can now combine \eqref{Brakke2}, \eqref{Brakke4}, \eqref{Brakke10}, \eqref{Brakke key1}, and \eqref{Brakke ci siamo quasi} to deduce that
\begin{equation} \label{Brakke last effort}
\begin{split}
 \|V_t\|(\phi) \Big|_{t=t_1}^{t_2} \leq &- \int_{t_1}^{t_2} \int_{U} \left(\abs{h(\cdot, V_t)}^2\, \phi - h(\cdot, V_t) \cdot \nabla\phi\right) \, d\|V_t\| \, dt \\
 &+ i^{-1} \, \|\pa\E_0\|(\R^{n+1}) + 2 i^{\frac12} \eps \, \int_{t_1}^{t_2} B(t)^{\frac12} \, dt\,.
 \end{split}
\end{equation}

We use the Cauchy-Schwarz inequality one more time, and combine it with the definition of $B(t)$ as the right-hand side of \eqref{Brakke13} and with Fatou's lemma to obtain the bound
\begin{equation} \label{errore sotto controllo}
\int_{t_1}^{t_2} B(t)^{\frac12} \, dt \leq (t_2 - t_1) + c\, (t_2 - t_1) + 4\, \liminf_{\ell\to\infty} \int_{t_1}^{t_2} \int_{\R^{n+1}} \eta_{j_\ell} \, \hat\phi \frac{\abs{\Phi_{\eps_{j_\ell}} \ast \delta (\pa\E_{j_\ell})}^2}{\Phi_{\eps_{j_\ell}} \ast \| \pa\E_{j_\ell}  \| + \eps_{j_\ell}} \,,
\end{equation}
which is finite (depending on $t_2$) by \eqref{finite total mean curvature} (recall that $\hat \phi \leq 1$ everywhere). Brakke's inequality \eqref{e:Brakke} for a test function $\phi$ which does not depend on $t$ is then deduced from \eqref{Brakke last effort} after letting $\eps \downarrow 0$ and then $i \uparrow \infty$.

\smallskip

{\bf Step 2.} We consider now the general case of a time dependent test function $\phi \in C^{1}_{c}(U \times \R^+ ;\R^+)$. We can once again assume that $\phi$ is smooth, and then conclude by a density argument. The proof follows the same strategy of Step 1. We define $\hat\phi$ analogously, and then we apply \eqref{induction:mass variation} with $\phi = \hat\phi(\cdot, t)$. In place of \eqref{Brakke1}, we then obtain a formula with one extra term, namely
\begin{equation} \label{Brakke1bis}
\begin{split}
\| \pa \E_{j_\ell}(s)\|(\hat \phi(\cdot,s)) \Big|_{s=t-\Delta t_{j_\ell}}^{t} \leq \Delta t_{j_\ell} \, &\left(\delta (\pa \E_{j_\ell}(t), \hat \phi(\cdot, t))(\eta_{j_\ell} \, h_{\eps_{j_\ell}}(\cdot, \pa \E_{j_\ell}(t))) + \eps_{j_\ell}^{\sfrac18}\right)\\
&+ \| \pa \E_{j_\ell}(t-\Delta t_{j_\ell}) \| (\phi (\cdot, t) - \phi (\cdot, t-\Delta t_{j_\ell}))\,.
\end{split}
\end{equation}
Similarly, the inequality in \eqref{Brakke2} needs to be replaced with an analogous one containing, in the right-hand side, also the term
\begin{equation} \label{Brakke time}
\sum_{k=k_1}^{k_2}  \| \pa \E_{j_\ell}((k-1)\Delta t_{j_\ell}) \| (\phi (\cdot, k\, \Delta t_{j_\ell}) - \phi (\cdot, (k-1)\Delta t_{j_\ell}))\,.
\end{equation} 

Using the regularity of $\phi$ and the estimates in \eqref{induction:mass estimate} and \eqref{induction:mean curvature}, we may deduce that
\begin{equation} \label{Brakke time 2}
\begin{split}
\lim_{\ell\to\infty} \eqref{Brakke time} &= \lim_{\ell \to \infty} \sum_{k=k_1}^{k_2} \| \pa \E_{j_\ell}(k\,\Delta t_{j_\ell}) \| \left( \frac{\pa\phi}{\pa t}(\cdot, k\,\Delta t_{j_\ell}) \right) \Delta t_{j_\ell} \\
&= \lim_{\ell\to\infty} \int_{t_1}^{t_2} \| \pa \E_{j_\ell}(t) \| \left( \frac{\pa\phi}{\pa t}(\cdot, t) \right) \, dt \\
&= \int_{t_1}^{t_2} \|V_t\| \left(  \frac{\pa\phi}{\pa t}(\cdot, t) \right)\, dt\,,
\end{split}
\end{equation}
where the last identity is a consequence of \eqref{e:limit_measure}, Proposition \ref{p:integral varifold limit}(1), and Lebesgue's dominated convergence theorem. The remaining part of the argument stays the same, modulo the following variation. The identity in \eqref{Brakke10} remains true if $\hat \phi$ is replaced by the piecewise constant function $\hat\phi_{j_\ell}$ defined by
\[
\hat\phi_{j_\ell}(x,t) := \hat\phi(x,k\, \Delta t_{j_\ell}) \qquad \mbox{if $t \in \left( (k-1) \, \Delta t_{j_\ell}, k\,\Delta t_{j_\ell}  \right]$}\,.
\]
The error one makes in order to put $\hat \phi$ back into \eqref{Brakke10} in place of $\hat\phi_{j_\ell}$ is then given by the product of $\Delta t_{j_\ell}$ times some negative powers of $\eps_{j_\ell}$; nonetheless, this error converges to $0$ uniformly as $\ell \uparrow \infty$ by the choice of $\Delta t_{j_\ell}$, see \eqref{d:time step}. This allows us to conclude the proof of \eqref{e:Brakke} precisely as in the case of a time-independent $\phi$ whenever $\phi \in C^\infty_c(U \times \R^+ ; \R^+)$, and in turn, by approximation, also when $\phi \in C^1_c(U \times \R^+ ; \R^+)$.
\end{proof}

\section{Boundary behavior and proof of main results} \label{sec:bb}

\subsection{Vanishing of measure outside the convex hull of initial data}

First, we prove that the limit measures $\|V_t\|$ vanish uniformly 
in time near $\pa
U\setminus \pa\Gamma_0$. This is a preliminary result, and using the Brakke's inquality,
we eventually prove that
they actually vanish outside the convex hull of $\Gamma_0\cup\pa \Gamma_0$ in Proposition \ref{vvmp}.
\begin{proposition} 
\label{vmn}
For $\hat x\in \pa U\setminus \pa\Gamma_0$, suppose that 
an affine hyperplane $A\subset\R^{n+1}$ with $\hat x\notin A$ has the following property. Let $A^+$ and $A^-$ be defined as the open half-spaces separated by $A$, i.e., 
$\R^{n+1}$ is a disjoint union of $A^+$, $A$ and $A^-$, with $\hat x\in A^+$. 
Define $d_A(x):={\rm dist}\,(x,A^-)$, and suppose that 
\begin{enumerate}
\item $\Gamma_0\cup \pa \Gamma_0\subset A^-$,
\item $d_A$ is $\nu_{U}$-non decreasing in $A^+$. 
\end{enumerate}
Then for any compact set $C\subset A^+$, we have
\begin{equation}
\label{vmn1}
\lim_{j\rightarrow\infty}\sup_{t\in[0,j^{\sfrac12}]}\|\pa\E_j(t)\|(C)=0.
\end{equation}
\end{proposition}

\begin{remark} \label{rmk:hyperplane}
Due to the definition of $\pa\Gamma_0$ and the strict convexity of $U$,
note that there exists such an affine hyperplane $A$ for any given $\hat x\in \partial U\setminus
\pa \Gamma_0$. For example, we may choose a hyperplane $A$ which is parallel to the tangent
space of $\partial U$ at $\hat x$ and which passes through $\hat x- \nu_{U}(\hat x)c$. By the strict
convexity of $U$ and the $C^1$ regularity 
of $\nu_{U}$, for all sufficiently small $c>0$, one can show that such
$A$ satisfies the above (1) and (2). 
\end{remark}
\begin{remark}
In the following proof, we adapted a computation from \cite[p.60]{Ilm1}. There, the object 
is the Brakke flow, but the 
basic idea here is that a similar computation can be carried out for the approximate MCF
with suitable error estimates. 
\end{remark}
\begin{proof}
We may assume after a suitable change of coordinates that $A=\{x_{n+1}=0\}$
and $A^+=\{x_{n+1}>0\}$. With this, we have ${\rm clos}\,\Gamma_0\subset
\{x_{n+1}<0\}$ and $d_A(x)=\max\{x_{n+1},0\}$ is $\nu_{U}$-non decreasing in $\{x_{n+1} > 0\}$. 
Let $s>0$
be arbitrary, and define
\begin{equation}
\label{vmn2}
\phi(x):=s+ (d_A(x))^\beta
\end{equation}
for some $\beta\geq 3$ to be fixed later. Then $\phi\in C^2(\R^{n+1}; \R^+)$, and letting $\{e_1,\,\ldots,\,e_{n+1}\}$ denote the standard basis of $\R^{n+1}$, we have
\begin{equation}
\label{ilm11}
\nabla\phi = \beta \, d_A^{\beta - 1} \, e_{n+1}\,, \qquad \nabla^2\phi = \beta\,(\beta-1)\, d_A^{\beta-2} \, e_{n+1} \otimes e_{n+1}\,.
\end{equation}
With $s>0$ fixed, we choose sufficiently large $j$ 
so that $\phi\in \mathcal A_j$. Actually, the function $\phi$ as defined in \eqref{vmn2} is unbounded. Nonetheless, since we know that $\spt\,\|\pa\E_j(t)\|\subset (U)_{1/(4j^{\sfrac14})}$, we may modify $\phi$ suitably 
away from $U$ by multiplying it by a small number and truncating it, so that 
$\phi\leq 1$. We assume that we have done this modification if necessary. 
We also choose $j$ so large that $\eta_j=1$ on $\{x_{n+1}\geq 0\}$. This is 
possible due to Lemma \ref{l:etaj}(1). 
Additionally, since $d_A$ is $\nu_{U}$-non decreasing in $A^+$, and since $\phi$ is constant in $\R^{n+1} \setminus A^+$, we have $\phi
\in \mathcal R_j$. 
Thus, by \eqref{induction:mass variation}, 
we have for $\pa\E_{j,k}=:V$ and $\pa\E_{j,k-1}=:\hat V$ with $k\in
\{1,\ldots,j2^{p_j}\}$ 
\begin{equation}
\label{ilm1}
\frac{\|V\|(\phi)-\|\hat V\|(\phi)}{\Delta t_j}\leq \eps_j^{\sfrac18}+\delta
(V,\phi)(\eta_j\,h_{\eps_j}(\cdot,V)).
\end{equation}
For all sufficiently large $j$, we also have $\eta_j\phi\in \mathcal A_j$, 
thus we may proceed as in \eqref{monotonicity estimate basic} and estimate
\begin{equation}\label{ilm4}
\begin{split}
\delta(V,\phi)(\eta_j h_{\eps_j}(\cdot,V))&=\delta V(\phi\,\eta_j\, h_{\eps_j})+
\int_{\bG_n(\R^{n+1})}\eta_j\, h_{\eps_j}(I-S)(\nabla\phi)\,dV(x,S) \\
&\leq -(1-\eps_j^{\sfrac14})\int \eta_j\,\phi\,\frac{|\Phi_{\eps_j}\ast\delta V|^2}{\Phi_{\eps_j}\ast\|V\|
+\eps_j}\,dx+\eps_j^{\sfrac14}+\frac12\int \eta_j \, \phi\, |h_{\eps_j}|^2\,d\|V\|\\
&\qquad +\frac12\int\frac{|S(\nabla\phi)|^2}{\phi}\,dV+\int h_{\eps_j}\cdot\nabla\phi\,d\|V\|.
\end{split}
\end{equation}

Here we have used that $\eta_j=1$ when $\nabla\phi\neq 0$. In the present proof, we omit the
domains of integration, which are either $\R^{n+1}$ or $\bG_n(\R^{n+1})$ unless specified otherwise.
We use \eqref{e:L2 norm of h vs approx} to proceed as:
\[
\leq -\left(1-\frac12-\frac{3\eps_j^{\sfrac14}}{2}\right)\int\eta_j\,\phi\,\frac{|\Phi_{\eps_j}\ast\delta V|^2}{\Phi_{\eps_j}\ast\|V\|
+\eps_j}\,dx+2\eps_j^{\sfrac14}+\frac12\int\frac{|S(\nabla\phi)|^2}{\phi}\,dV+\int h_{\eps_j}\cdot\nabla\phi\,d\|V\|.
\]
 We prove that the last term gives a good negative contribution. 
 We have
 \begin{equation}
 \begin{split}
 \int & h_{\eps_j}\cdot \nabla\phi\,d\|V\|=-\int \Phi_{\eps_j}\ast\frac{\Phi_{\eps_j}
 \ast\delta V}{\Phi_{\eps_j}\ast\|V\|+\eps_j}\cdot\nabla\phi\, d\|V\| \\
 &
 =-\int \Big(\frac{\Phi_{\eps_j}
 \ast\delta V}{\Phi_{\eps_j}\ast\|V\|+\eps_j}\Big)(y)\cdot \int \Phi_{\eps_j}(x-y)\nabla\phi(x)
 \,d\|V\|(x)\,dy.
 \end{split}
 \label{ilm5}
 \end{equation}
 Here we replace $\nabla\phi(x)$ by $\nabla\phi(y)$ and estimate the error
 \begin{equation}\label{ilm2}
   \Big|
 \int \Big(\frac{\Phi_{\eps_j}
 \ast\delta V}{\Phi_{\eps_j}\ast\|V\|+\eps_j}\Big)(y)\cdot \int \Phi_{\eps_j}(x-y)(\nabla\phi(x)
 -\nabla\phi(y))\,d\|V\|(x)\,dy\Big|.
 \end{equation}
To estimate \eqref{ilm2}, since $\eta_j\phi\in \mathcal A_j$, \eqref{classA} and \eqref{e:Gronwall} imply
 \[|\nabla\phi(x)-\nabla\phi(y)|=|\nabla(\eta_j\phi)(x)-\nabla(\eta_j\phi)(y)|\leq j\,
 |x-y|\,\eta_j(y)\,\phi(y)\,\exp(j|x-y|)\,.\]
 By separating the integration to $B_{\sqrt{\eps_j}}(y)$ and $B_1(y)\setminus B_{\sqrt{\eps_j}}(y)$,
\begin{equation}
\label{ilm3}
\begin{split}
\int\Phi_{\eps_j}(x-y) &|\nabla\phi(x)-\nabla\phi(y)|\,d\|V\|(x)\leq j\,\sqrt{\eps_j}\,\exp(j\sqrt{
\eps_j})\,\eta_j(y)\,\phi(y)\,(\Phi_{\eps_j}\ast\|V\|)(y) \\
&+c(n)\,\eps_j^{-n-1}\,j\,\exp(j-(2\eps_j)^{-1})\,\eta_j(y)\,\phi(y)\,\|V\|(B_1(y)).
\end{split}
\end{equation}
Let us denote $c_{\eps_j}:=c(n)\eps_j^{-n-1}j\exp(j-(2\eps_j)^{-1})$ and note that 
it is exponentially small (say, $\leq \exp(-\eps_j^{-\sfrac12})$ for all large $j$) due to $j\leq \eps_j^{-1/6}/2$. Similarly we have $j\sqrt{\eps_j}
\exp(j\sqrt{\eps_j})\leq \eps_j^{\sfrac14}$, so that
\[
\int\Phi_{\eps_j}(x-y) |\nabla\phi(x)-\nabla\phi(y)|\,d\|V\|(x)
\leq (\eps_j^{\sfrac14}(\Phi_{\eps_j}\ast\|V\|)(y)+c_{\eps_j}\|V\|(B_1(y)))\eta_j(y)\phi(y).\]
Using this, we can estimate
\begin{equation}
\label{ilm9}
\begin{split}
|\eqref{ilm2}|&\leq \left(\int\eta_j \, \phi \,\frac{|\Phi_{\eps_j}\ast\delta V|^2}{\Phi_{\eps_j}\ast\|V\|+\eps_j}\right)^{\frac12}\,\left(2\,\int \eps_j^{\frac12}\,(\Phi_{\eps_j}\ast\|V\|)(y)+c_{\eps_j}^2\,\eps_j^{-1}\,\|V\|(B_1(y))^2\,dy\right)^{\frac12} \\
&\leq \eps_j^{\frac14}\, \int\eta_j\,\phi\,\frac{|\Phi_{\eps_j}\ast\delta V|^2}{\Phi_{\eps_j}\ast\|V\|+\eps_j}
+\int \eps_j^{\frac14}\,(\Phi_{\eps_j}\ast\|V\|)(y)+c_{\eps_j}^2\,\eps_j^{-\frac54}\,\|V\|(B_1(y))^2\,dy.
\end{split}
\end{equation}
In view of \eqref{ilm4}, this shows that \eqref{ilm2} can be absorbed as a small error term.
 Continuing from \eqref{ilm5} with $\nabla\phi(y)$ replacing $\nabla\phi(x)$, 
 \begin{equation}
 \begin{split}
 & -\int \Big(\frac{\Phi_{\eps_j}\ast\delta V}{\Phi_{\eps_j}\ast\|V\|+\eps_j}\Big)(y)\cdot \int \Phi_{\eps_j}(x-y)\nabla\phi(y)
 \,d\|V\|(x)\,dy \\
  =& -\int \Big(\frac{\Phi_{\eps_j}
 \ast\delta V}{\Phi_{\eps_j}\ast\|V\|+\eps_j}\Big)(y)\cdot \nabla\phi(y)\, ( \Phi_{\eps_j}\ast\|V\|)(y)\,dy \\
  =& -\int(\Phi_{\eps_j}\ast\delta V)\cdot\nabla\phi\,dy+\eps_j
 \int \Big(\frac{\Phi_{\eps_j}
 \ast\delta V}{\Phi_{\eps_j}\ast\|V\|+\eps_j}\Big)(y)\cdot \nabla\phi(y)\,dy\,.
 \end{split}\label{ilm6}\end{equation}
 The last term of \eqref{ilm6} may be estimated as
 \begin{equation}\label{ilm7}
 \begin{split}
 \eps_j \Big|\int \Big(\frac{\Phi_{\eps_j}
 \ast\delta V}{\Phi_{\eps_j}\ast\|V\|+\eps_j}\Big)(y)\cdot& \nabla\phi(y)\,dy\Big|
 \leq j\, \eps_j \int_{(U)_{2}} \eta_j \, \phi \,\frac{|\Phi_{\eps_j}
 \ast\delta V|}{\Phi_{\eps_j}\ast\|V\|+\eps_j} \\
 &\leq j\, \eps_j^{\frac12} \, \Big(\int \eta_j \, \phi\,\frac{|\Phi_{\eps_j}\ast\delta V|^2}{\Phi_{\eps_j}\ast\|V\|+
 \eps_j}\Big)^{\frac12}\Big(\int_{(U)_2}\eta_j\,\phi\Big)^{\frac12} \\
 &\leq  \eps_j^{\frac14}\int\eta_j \,\phi\,\frac{|\Phi_{\eps_j}\ast\delta V|^2}{\Phi_{\eps_j}\ast\|V\|+
 \eps_j}+j^2\,\eps_j^{\frac34}\,\int_{(U)_2}\eta_j\phi.
 \end{split}
\end{equation}
Here, we used the fact that the integrand is $0$ far away from $U$, for
example, outside of $(U)_2$. 
The last term of \eqref{ilm7} can be absorbed as a small error since $j\leq \eps_j^{-1/6}/2$ 
and $\int_{(U)_2}\eta_j\,\phi$ is bounded by a constant.
We can continue as 
 \begin{equation*}
 \begin{split}
-\int (\Phi_{\eps_j}\ast\delta V)\cdot\nabla\phi\,dy &=-\iint S(\nabla\Phi_{\eps_j}(x-y))\, dV(x,S)\nabla\phi(y)\,dy \\ 
 &
 =-\int S\cdot\Big(\int \nabla\Phi_{\eps_j}(x-y)\otimes\nabla\phi(y)\,dy\Big)\,dV(x,S)
 \\ &
 =-\int S\cdot \int\Phi_{\eps_j}(x-y)\,\nabla^2\phi(y)\, dy\, dV(x,S).
 \end{split}
 \end{equation*}
 We replace $\nabla^2\phi(y)$ by $\nabla^2\phi(x)$, with the resulting error being estimated, for instance, by $\leq M \eps_j^{\sfrac12}$ using standard methods as above. Then, we have
 \begin{equation}
 \label{ilm8}
 -\int(\Phi_{\eps_j}\ast\delta V)\cdot\nabla\phi\,dy\leq 
  -\int S\cdot \nabla^2\phi(x)\, dV(x,S) +M\eps_j^{\sfrac12}.
 \end{equation}
 Thus, combining \eqref{ilm1}-\eqref{ilm8} and recovering the notations,
 we obtain
\begin{equation}
 \label{ilm10}
\frac{\|\pa\E_{j,k}\|(\phi) - \|\pa\E_{j,k-1}\|(\phi)}{\Delta t_j}
\leq 2\eps_j^{\sfrac18}+\int\frac{|S(\nabla\phi)|^2}{2\phi}-S\cdot\nabla^2\phi\,dV
\end{equation}
for all sufficiently large $j$. By \eqref{ilm11}, 
we have 
 \begin{equation} \label{ilm12}
 \begin{split}
 \frac{|S(\nabla\phi)|^2}{2\phi}-S\cdot\nabla^2\phi &= \left( \frac{\beta^2}{2} \, \sum_{i=1}^{n+1} S_{i,n+1}^2 - \beta\, (\beta - 1) \, S_{n+1,n+1}   \right) \, d_A^{\beta - 2}\\ &= \left(\frac{\beta^2}{2}
- \beta\,(\beta-1)\right)\, \abs{S_{n+1,n+1}}\, d_A^{\beta-2}\,,
\end{split}
 \end{equation}
where in the last identity we have used that $S$ is the matrix representing an orthogonal projection operator, so that $S$ is symmetric and $S^2 = S$, whence
\[
S_{n+1,n+1} = (S^2)_{n+1,n+1} = \sum_{i=1}^{n+1} S_{i,n+1}^2 \geq 0\,.
\] 
 
In particular, the quantity in \eqref{ilm12} can be made negative if $\beta=4$, for example. This shows that \eqref{ilm10} is less than $2\eps_j^{\sfrac18}$. By summing
over $k=1,\ldots,j^{\sfrac12}/(\Delta t_j)$ and using that $\|\pa\E_{j,0}\|(\phi)
=s\,\Ha^n(\Gamma_0)$, we obtain
\begin{equation}
\label{ilm13}
\sup_{t\in[0,j^{\sfrac12}]}\|\pa\E_j(t)\|(\phi)\leq 2\eps_j^{\sfrac18} j^{\sfrac12}
+s\,\Ha^n(\Gamma_0).
\end{equation}
Fix $\rho>0$ so that 
$C\subset\{x_{n+1}>\rho\}$. Then we have $\phi\geq \rho^{\beta}$ on $C$. 
With this, we have $\|\pa\E_j(t)\|(C)
\leq \rho^{-\beta} \|\pa\E_j(t)\|(\phi)$. We use this in \eqref{ilm13}, 
and we let first $j\rightarrow\infty$ and then $s\rightarrow 0$ in order to obtain \eqref{vmn1}.
\end{proof}

\begin{proposition}\label{vvmp}
For all $t\geq 0$, we have ${\rm spt}\,\|V_t\|\subset{\rm conv}\,(\Gamma_0\cup\pa\Gamma_0)$. 
\end{proposition}
\begin{proof}
Suppose that $A\subset\R^{n+1}$ is a hyperplane such that, using the notation in the 
statement of Proposition \ref{vmn}, $\Gamma_0\cup\partial\Gamma_0\subset A^-$. If 
$d_A$ is $\nu_{U}$-non decreasing in $A^+$, then \eqref{vmn1} proves immediately that
$\|V_t\|(A^+)=0$ for all $t\geq 0$. Thus, suppose that $d_A$ does not satisfy this property.
Still, due to Proposition \ref{vmn}, for each $x\in\partial U\setminus\partial\Gamma_0$, 
there exists a neighborhood $B_r(x)$ such that $\|V_t\|(B_r(x)\cap U)=0$ for all $t\geq 0$. 
In particular, there exists some $r_0>0$ such that 
\begin{equation}
\|V_t\|((\partial U)_{r_0}\cap A^+)=0
\label{vvm1}
\end{equation}
for all $t\geq 0$. Let $\psi\in C^{\infty}_c(U;\R^+)$ be such that $\psi=1$ on $U\setminus 
(\partial U)_{r_0}$ and $\psi=0$ on $(\partial U)_{\frac{r_0}{2}}$. We next use $\phi
=\psi\, d_A^{4}$ in \eqref{e:Brakke} with $t_1=0$ and an arbitrary $t_2=t>0$ to obtain
\begin{equation}
\begin{split}
\label{vvm2}
\|V_s\|(\phi)\Big|_{s=0}^{t}&\leq \int_0^t \int_{U}(\nabla\phi-\phi\, h(\cdot,V_s))\cdot h(\cdot,V_s)\,d\|V_s\|\,ds \\
& \leq -\int_0^t\int_{U} S\cdot\nabla^2\phi\,dV_s(\cdot,S)\,ds.
\end{split}
\end{equation}
By \eqref{vvm1}, $\phi=d_A^4$ on the support of $\|V_s\|$. Since $S\cdot \nabla^2 d_A^4\geq 0$
for any $S\in {\bf G}(n+1,n)$ (see \eqref{ilm12}), the right-hand side of \eqref{vvm2}
is $\leq 0$. Since $\|V_0\|(\phi)=0$, we have $\|V_t\|(A^+)=0$ for all $t>0$. This proves
the claim. 
\end{proof}

In the following, we list results from \cite[Section 10]{KimTone}. The results are local in nature, thus
even if we are concerned with a Brakke flow in $U$ instead of $\R^{n+1}$, the 
proofs are the same. We recall the following (cf. Theorem \ref{thm:main2}(11)):
\begin{definition} \label{spacetime_measure}
Define a Radon measure $\mu$ on $U \times \R^+$ by setting $d\mu := d\|V_t\|\,dt$, namely
\begin{equation} \label{e:spacetime_measure}
\int_{U\times\R^+} \phi(x,t) \, d\mu(x,t) := \int_0^\infty \left(\int_U \phi(x,t) \, d\|V_t\|(x)\right)\,dt \qquad \mbox{for every $\phi \in C_c(U \times \R^+)$}\,. 
\end{equation}
\end{definition}

\begin{lemma} \label{sptfini}
We have the following properties for $\mu$ and $\{V_t\}_{t\in\R^+}$.
\begin{enumerate}
\item ${\rm spt}\,\|V_t\|\subset\{x\in U\,:\,(x,t)\in{\rm spt}\,\mu\}$ for all $t>0$.
\item For each $\tilde U\ssubset U$ and $t>0$, we have $\Ha^n(\{x\in\tilde  U\,:\,
(x,t)\in {\rm spt}\,\mu\})<\infty$. 
\end{enumerate}
\end{lemma}  
The next Lemma (see \cite[Lemma 10.10 and 10.11]{KimTone}) is used to prove the continuity of the labeling of partitions. 
\begin{lemma} \label{contidom}
Let $\{\mathcal E_{j_{\ell}}(t)\}_{\ell=1}^{\infty}$ be the sequence obtained in Proposition \ref{p:integral varifold limit}, and let $\{E_{j_\ell,i}(t)\}_{i=1}^N$ denote the open partitions
for each $j_\ell$ and $t\in\R^+$, i.e., $\mathcal E_{j_\ell}(t)=\{E_{j_\ell,i}(t)\}_{i=1}^N$.
\begin{enumerate}
\item
For fixed $i\in \{1,\ldots,N\}$, $B_{2r}(x)\ssubset U$, $t>0$ with $t-r^2>0$, suppose that
\[ \lim_{\ell\rightarrow\infty} \mathcal L^{n+1}(B_{2r}(x)\setminus E_{j_\ell,i}(t))=0
\hspace{.5cm}\mbox{and}\hspace{.5cm}\mu(B_{2r}(x)\times [t-r^2,t+r^2])=0.\]
Then for all $t'\in (t-r^2,t+r^2]$, we have 
\[ \lim_{\ell\rightarrow\infty} \mathcal L^{n+1}(B_r(x)\setminus E_{j_\ell,i}(t'))=0.\]
\item For fixed $i\in\{1,\ldots,N\}$, $B_{2r}(x)\ssubset U$ and $r>0$, suppose that
\[B_{2r}(x)\subset E_{j_\ell,i}(0)\hspace{.3cm}\mbox{for all $\ell\in\mathbb N$}\hspace{.5cm} 
\mbox{and}
\hspace{.5cm} \mu(B_{2r}(x)\times[0,r^2])=0.\]
Then for all $t'\in (0,r^2]$, we have
\[ \lim_{\ell\rightarrow\infty} \mathcal L^{n+1}(B_r(x)\setminus E_{j_\ell,i}(t'))=0.\]
\end{enumerate}
\end{lemma} 
The following is from \cite[3.7]{Brakke}. 
\begin{lemma}\label{lemma1012}
Suppose that $\|V_t\|(U_r(x)) = 0$ for some $t \in \R^+$ and $U_r(x)\ssubset U$. Then, for every $t' \in \left[t, t+\frac{r^2}{2n} \right]$ it holds $\|V_{t'}\|(U_{\sqrt{r^2 - 2n\,(t'-t)}}(x)) = 0$.
\end{lemma}

\begin{proof}[{\bf Proof of Theorem \ref{thm:main2}}]

Let $\{ \E_{j_\ell}(t) \}_{\ell=1}^{\infty}$ be a sequence as in Lemma \ref{contidom}, with $\E_{j_\ell}(t) = \{ E_{j_\ell,i}(t) \}_{i=1}^N$ for every $\ell \in \mathbb{N}$. Since $E_{j_\ell,i}(t) \subset \left( U \right)_1$, for each $t$ and $i$ the volumes $\mathcal{L}^{n+1}(E_{j_\ell,i}(t))$ are uniformly bounded in $\ell$. Furthermore, by the mass estimate in \eqref{e:precompactness} we also have that $\| \nabla \chi_{E_{j_\ell,i}(t)} \| (\R^{n+1})$ are uniformly bounded. Hence, we can use the compactness theorem for sets of finite perimeter in order to select a (not relabeled) subsequence with the property that, for each fixed $i \in \{1,\ldots,N\}$, 
\begin{equation} \label{e:convergence}
\chi_{E_{j_\ell,i}(t)} \to \chi_{E_i(t)} \quad \mbox{in $L^1_{loc}(\R^{n+1})$ $\qquad$ for every $t \in 2_\Q$}\,,
\end{equation} 
where $E_i(t)$ is a set of locally finite perimeter in $\R^{n+1}$. Moreover, using that $E_{j_\ell,i}(t) \subset \left( U \right)_{1/(4\,j_\ell^{\sfrac14})}$ (see Proposition \ref{p:induction} and \eqref{epsilon conditions}) we see that $\mathcal{L}^{n+1}(E_i(t) \setminus U) = 0$. Since sets of finite perimeter are defined up to measure zero sets, we can then assume without loss of generality that $E_i(t) \subset U$. Hence, since $\Ha^n(\pa U) < \infty$, $E_i(t)$ is in fact a set of finite perimeter in $\R^{n+1}$. 

\smallskip

Next, consider the complement of $\spt\,\mu \cup (\Gamma_0 \times \{0\})$ in $U \times \R^+$, which is relatively open in $U \times \R^+$, and let $S$ be one of its connected components. For any point $(x,t) \in S$ there exists $r > 0$ such that either $B_{2\,r}(x) \times \left[ t-r^2, t+r^2 \right] \subset S$ if $t > 0$, or $B_{2\,r}(x) \times \left[0, r^2 \right] \subset S$ if $t=0$. We first consider the case $t=0$. Since $B_{2\,r}(x)$ lies in the complement of $\Gamma_0$, there exists $i(x,0) \in \{1,\ldots,N\}$ such that $B_{2\,r}(x) \subset E_{0,i(x,0)}$, and thus $B_{2\,r}(x) \subset E_{j_\ell,i(x,0)}(0)$ for all $\ell \in \mathbb N$. Since also $\mu (B_{2\,r}(x) \times \left[0,r^2\right]) = 0$, we can apply Lemma \ref{contidom}(2) and conclude that 
\begin{equation} \label{e:conclusion t=0}
\lim_{\ell \to \infty}\mathcal{L}^{n+1}(B_r(x) \setminus E_{j_\ell,i(x,0)}(t')) = 0 \qquad \mbox{for all $t' \in \left(0, r^2 \right]$}\,.
\end{equation}
Similarly, if $t > 0$, since $\mu (B_{2\,r}(x) \times \left[t-r^2,t+r^2  \right]) = 0$, we can apply Lemma \ref{contidom}(1) to conclude that there is a unique $i(x,t)\in \{1,\ldots,N\}$ such that
\begin{equation} \label{e:conclusion t>0}
\lim_{\ell \to \infty}\mathcal{L}^{n+1}(B_r(x) \setminus E_{j_\ell,i(x,t)}(t')) = 0 \qquad \mbox{for all $t' \in \left(t-r^2, t+ r^2 \right]$}\,.
\end{equation}

\smallskip

Now, observe that if $S$ is any connected component of the complement of $\spt\,\mu \cup (\Gamma_0 \times \{0\})$ in $U \times \R^+$, then by \eqref{e:conclusion t=0} and \eqref{e:conclusion t>0}, and since $S$ is connected, for any two points $(x,t)$ and $(y,s)$ in $S$ it has to be $i(x,t) = i(y,s)$. For every $i \in \{1,\ldots,N\}$, we can then let $S(i)$ denote the union of all connected components $S$ such that $i(x,t) = i$ for every $(x,t) \in S$. It is clear that $S(i)$ are open sets, and that $E_{0,i} = \left\lbrace x \in U \, \colon \, (x,0) \in S(i) \right\rbrace$ (notice that if $x \in E_{0,i}$ then $(x,0) \notin \spt\,\mu$ as a consequence of Lemma \ref{lemma1012}), so that each $S(i)$ is not empty. Furthermore, we have that $\bigcup_{i=1}^N S(i) = (U \times \R^+) \setminus \left(  \spt\,\mu \cup (\Gamma_0 \times \{0\}) \right)$. For every $t \in \R^+$, we can thus define
\begin{equation} \label{def partition final}
E_i(t) := \left\lbrace x \in U \, \colon \, (x,t) \in S(i) \right\rbrace\,,\,\,
\Gamma(t):=U\setminus \cup_{i=1}^N E_i(t).
\end{equation}
By examining the definition, one obtains $\Gamma(t)=\{x\in U\,:\, (x,t)\in {\rm spt}\,\mu\}$ for all $t>0$.
Combined with Lemma \ref{sptfini}(1), we have (11). By Lemma \ref{sptfini}(2), we have (3), and this 
also proves that $\Gamma(t)$ has empty interior, which shows (4). The claims (1) and (2) hold true by construction. (5) is a consequence of Proposition \ref{vvmp} and the definition of $\mu$ being
the product measure. (6) is similar: if $x \in U \setminus {\rm conv}(\Gamma_0\cup\pa\Gamma_0)$ then the half-line $t \in \R^+ \mapsto \gamma_x(t) := \left( x,t \right) \in U \times \R^+$ must be contained in the same connected component of $(U \times \R^+) \setminus ( \spt \,\mu \cup (\Gamma_0 \times \{0\}) )$, for otherwise there would be $t > 0$ such that $(x,t) \in \spt\,\mu$, thus contradicting (5). 
For (7), by the strict convexity of $U$ and (5), we have $\pa\Gamma(t)\subset \pa\Gamma_0$ for all
$t>0$. Later in Proposition \ref{p:boundary data}, we prove $({\rm clos}\,({\rm spt}\,\|V_t\|))\setminus
U=\pa\Gamma_0$ and $\pa\Gamma_0\subset\pa\Gamma(t)$ follows from this and (11). Coming to (8), we use \eqref{e:conclusion t>0} together with the conclusions in Proposition \ref{p:induction}(1) to see that $\chi_{E_{j_\ell,i}(t)} \to \chi_{E_i(t)}$ in $L^1(\R^{n+1})$ as $\ell \uparrow\infty$, for every $t \in \R^+$. In particular, the lower semi-continuity of perimeter allows us to deduce that for any $\phi\in C_c(U;\R^+)$ \[\| \nabla \chi_{E_i(t)} \| (\phi) \leq \liminf_{\ell \to \infty} \| \nabla \chi_{E_{j_\ell,i}(t)} \| (\phi) \leq \liminf_{\ell\to\infty} \| \pa \E_{j_\ell}(t) \| (\phi) = \|V_t\|(\phi)\,, \]  
thus proving $\|\nabla\chi_{E_i(t)}\|\leq \|V_t\|$ of (8). Using the cluster structure of each $\pa \E_{j_\ell}(t)$ (see e.g. \cite[Proposition 29.4]{Maggi_book}), we have in fact that
\[
\frac12 \sum_{i=1}^N \| \nabla \chi_{E_{j_\ell,i}(t)} \| (\phi) =\Ha^n\mres_{(\cup_{i=1}^N \pa^*
E_{j_\ell,i}(t))}(\phi)\leq  \| \pa \E_{j_\ell}(t) \| (\phi) \qquad \mbox{for every $\phi$ as above}\,,
\]
which shows the other statement $\sum_{i=1}^N \| \nabla \chi_{E_i(t)} \| \leq 2\, \|V_t\|$ in (8). 
Since the claim of (9) is interior in nature, the proof is identical to the case without boundary
as in \cite[Theorem 3.5(6)]{KimTone}. For the proof of (10), for $\bar t\geq 0$, 
we prove that $\chi_{E_i(t)}\to \chi_{E_i(\bar t)}$ in $L^1(U)$ as $t\to \bar t$ for
each $i=1,\ldots,N$. Since $\|\nabla\chi_{E_i(t)}\|(U)\leq \|V_t\|(U)\leq \Ha^n(\Gamma_0)$, 
for any $t_k\to \bar t$, there exists a subsequence (denoted by the same index) and 
$\tilde E_i\subset U$ such that $\chi_{E_i(t_k)}\to \chi_{\tilde E_i}$ in $L^1(U)$ and 
$\mathcal L^{n+1}$ a.e.~by the
compactness theorem for sets of finite perimeter. We also have $\mathcal L^{n+1}(\tilde E_i\cap \tilde E_j)=0$
for $i\neq j$ and $\mathcal L^{n+1}(U\setminus\cup_{i=1}^N \tilde E_i)=0$. For a contradiction,
assume that $\mathcal L^{n+1}(E_i(\bar t)\setminus \tilde E_i)>0$ for some $i$. Then, there must be 
$U_r(x)\ssubset E_i(\bar t)$ such that $\mathcal L^{n+1}(U_r(x)\setminus \tilde E_i)>0$. We then use
Theorem \ref{thm:main2}(9) with $g(t)=\mathcal L^{n+1}(E_i(t)\cap U_r(x))$, which gives
$\lim_{t\to \bar t}g(t)=g(\bar t)=\mathcal L^{n+1}(E_i(\bar t)\cap U_r(x))=\mathcal L^{n+1}(U_r(x))$. On the 
other hand, $\chi_{E_i(t)}\to \chi_{\tilde E_i}$ in $L^1(U)$ implies $\lim_{t\to\bar t}g(t)=
\mathcal L^{n+1}(\tilde E_i\cap U_r(x))<\mathcal L^{n+1}(U_r(x))$ because of $\mathcal L^{n+1}(U_r(x)\setminus \tilde E_i)>0$. This is a contradiction. Thus, we have $\mathcal L^{n+1}(E_i(\bar t)\setminus
\tilde E_i)=0$ for all $i=1,\ldots, N$. Since $\{\tilde E_1,\ldots,\tilde E_N\}$ is a partion of $U$,
we have $\mathcal L^{n+1}(E_i(\bar t)\triangle \tilde E_i)=0$ for all $i$. This proves (9),
and finishes the proof of (1)-(11) except for (7), which 
is independent and is proved once we prove Proposition \ref{p:boundary data}.  
\end{proof}

\begin{proposition} \label{p:boundary data}
For all $t \geq 0$, it holds $({\rm clos}\,(\spt\|V_t\|) )\setminus U = \pa\Gamma_0$. 
\end{proposition}

\begin{proof}
Let $x \in ({\rm clos}\,(\spt\|V_t\|) )\setminus U$, and let $\{x_k\}_{k=1}^\infty$ be a sequence with $x_k \in \spt\,\|V_t\|$ such that $x_k \to x$ as $k \uparrow \infty$. If $x \notin \pa \Gamma_0$, then by Proposition \ref{vmn} there is $r > 0$ such that $\|V_t\| (B_r(x) \cap U) = 0$. For all suitably large $k$ so that $\abs{x-x_k} < r$ we then have $\|V_t\|(B_{r-\abs{x-x_k}}(x_k) \cap U) = 0$, which contradicts the fact that $x_k \in \spt\|V_t\|$.

\smallskip

Conversely, let $x \in \pa \Gamma_0$, and suppose for a contradiction that $x \notin {\rm clos}\,(\spt\|V_t\|)$, so that there is a radius $r > 0$ with the property that $B_{r}(x) \cap \spt\|V_t\| = \emptyset$. 
Then, Theorem \ref{thm:main2}(8) implies that $\|\nabla \chi_{E_i(t)}\|(B_{r}(x) \cap U) = 0$ for every $i\in \{1,\ldots,N\}$. Since $B_{r}(x) \cap U$ is connected by the convexity of $U$, every $\chi_{E_i(t)}$ is either constantly equal to $0$ or $1$ on $B_{r}(x) \cap U$, namely
\begin{equation} \label{contradiction}
B_{r}(x) \cap U \subset E_\ell(t) \qquad \mbox{for some $\ell \in \{1,\ldots,N\}$}\,.
\end{equation}

If $t=0$, since $E_i(0) = E_{0,i}$ for every $i =1,\ldots,N$, the conclusion in \eqref{contradiction} is evidently incompatible with $(A4)$, thus providing the desired contradiction. We can then assume $t > 0$. By $(A4)$, there are at least two indices $i \neq i' \in \{1,\ldots,N\}$ and sequences of balls $\{B_{r_j}(x_j)\}_{j=1}^\infty$, $\{B_{r_j'}(x_j')\}_{j=1}^\infty$ such that $x_j, x_j' \in \pa U$, $\lim_{j\to\infty} x_j = \lim_{j\to\infty} x_j' = x$ and $B_{r_j}(x_j) \cap U \subset E_{0,i}$ whereas $B_{r_j'} (x_j') \cap U \subset E_{0,i'}$. Let $z$ denote any of the points $x_j$ or $x_j'$, and observe that the above condition guarantees that $z \in \pa U \setminus \pa \Gamma_0$. In turn, by arguing as in Remark \ref{rmk:hyperplane} we deduce that there is a neighborhood $B_{\rho}(z) \cap U$ such that $\|V_t\|(B_\rho(z) \cap U) = 0$ for all $t \geq 0$, and thus also $\| \nabla \chi_{E_l(t)} \| (B_\rho(z) \cap U) = 0$ for every $t \geq 0$ and for every $l \in \{1,\ldots,N\}$. Since $B_\rho(z) \cap U$ is connected this implies that $B_\rho(z) \cap U \subset E_l(t)$ for some $l$. Applying this argument with $z=x_j$ and $z=x_j'$ we then find radii $\rho_j$ and $\rho_j'$ such that, necessarily, $B_{\rho_j}(x_j) \cap U \subset E_i(t)$ whereas $B_{\rho_j'}(x_j') \cap U \subset E_{i'}(t)$ for all $t \geq 0$. Since $x_j \to x$ and $x_j' \to x$ this conclusion is again incompatible with \eqref{contradiction}, thus completing the proof. 
\end{proof}
\begin{proposition}
\label{inidata}
We have for each $\phi\in C_c(U;\R^+)$
\[\Ha^n\mres_{(\cup_{i=1}^N \pa^*E_{0,i})}(\phi)\leq \liminf_{t\downarrow 0}\|V_t\|(\phi)=\limsup_{t\downarrow 0}
\|V_t\|(\phi)\leq \Ha^n\mres_{\Gamma_0}(\phi).\]
In particular, if $\Ha^n(\Gamma_0\setminus \cup_{i=1}^N\pa^*E_{0,i})=0$, then we have
\[\lim_{t\downarrow 0}\|V_t\|=\Ha^n\mres_{\Gamma_0} \qquad \mbox{as Radon measures in $U$}\,.\]
\end{proposition}
\begin{proof}

By \cite[Proposition 29.4]{Maggi_book}, we have for each $\phi\in C_c(U;\R^+)$
\begin{equation*}
\begin{split}
&2\Ha^n\mres_{(\cup_{i=1}^N \pa^*E_{0,i})}(\phi)=\sum_{i=1}^N \|\nabla\chi_{E_{0,i}}\|(\phi)
\leq\sum_{i=1}^N \liminf_{t\downarrow 0} \|\nabla\chi_{E_i(t)}\|(\phi) \\
&\leq \liminf_{t\downarrow 0}\sum_{i=1}^N \|\nabla\chi_{E_i(t)}\|(\phi) \leq 2\liminf_{t\downarrow 0}
\|V_t\|(\phi)
\end{split}
\end{equation*}
where we also used Theorem \ref{thm:main2}(8) and (10). This proves the first inequality. 
The second equality and the third inequality follow from \eqref{muconti}, $\mu_t=\|V_t\|$ and 
$\|V_0\|=\Ha^n\mres_{\Gamma_0}$. 
\end{proof}
The proof of Theorem \ref{thm:main} is now complete: $\{V_t\}_{t\geq 0}$ is a Brakke flow with fixed
boundary $\pa\Gamma_0$ due to Proposition \ref{p:integral varifold limit}(1), Theorem \ref{t:Brakke inequality}
and Proposition \ref{p:boundary data}. Proposition \ref{inidata} proves the claim on the continuity
of measure at $t=0$. 
\section{Applications to the problem of Plateau}
\label{propla}

As anticipated in the introduction, an interesting byproduct of our global existence result for Brakke flow is the existence of a stationary integral varifold $V$ in $U$ satisfying the topological boundary constraint ${\rm clos}(\spt \|V\|) \setminus U = \pa \Gamma_0$. This is the content of Corollary \ref{main:cor}, which we prove next.

\begin{proof}[Proof of Corollary \ref{main:cor}]
By the estimate in \eqref{e:mean curvature bound}, the function
\[
H(t) := \int_U \abs{h(x,V_t)}^2 \, d\|V_t\|(x)
\]
is in $L^1(\left(0,\infty\right))$. Hence, there exists a sequence $\{t_k\}_{k=1}^\infty$ such that
\begin{equation} \label{vanishing sequence}
\lim_{k \to \infty} t_k = \infty\,, \qquad \lim_{k\to \infty} H(t_k) = 0\,.
\end{equation}
Let $V_k := V_{t_k}$. Again by \eqref{e:mean curvature bound}, we have that
\begin{equation} \label{mass bound}
\sup_{k} \|V_k\|(U) \leq \Ha^n(\Gamma_0)\,.
\end{equation}
Furthermore, combining \eqref{def:generalized mean curvature} with \eqref{mass bound} yields, via the Cauchy-Schwarz inequality, that
\begin{equation}
\abs{\delta V_k (g)} \leq \|g\|_{C^0} \, \left( \Ha^n(\Gamma_0) \right)^{\frac12} \, \left( H(t_k) \right)^{\frac12} \qquad \mbox{for every $g \in C_c(U;\R^{n+1})$}\,,
\end{equation} 
so that
\begin{equation} \label{first variation limit}
\lim_{k \to \infty} \| \delta V_k \| (U) = 0\,.
\end{equation}
Hence, we can apply Allard's compactness theorem for integral varifolds, see \cite[Theorem 42.7]{Simon}, in order to conclude the existence of a stationary integral varifold $V \in \IV_n(U)$ such that $V_k \to V$ in the sense of varifolds.

\smallskip

Next, we prove the existence of the family $\{E_i\}_{i=1}^N$. Fix $i \in \{1,\ldots,N\}$, and consider the sequence $\{E_i^k\}_{k=1}^\infty$, where $E_i^k := E_i(t_k)$. By Theorem \ref{thm:main2}(8) and \eqref{e:mean curvature bound} we have, along a (not relabeled) subsequence, the convergence
\begin{equation} \label{long_time_limit_sets}
\chi_{E_i^k} \to \chi_{E_i} \qquad \mbox{in $L^1(U)$ and pointwise $\mathcal{L}^{n+1}$-a.e. as $k \to \infty$}\,,
\end{equation}
where $E_i \subset U$ are sets of finite perimeter. Since, by Theorem \ref{thm:main2}(3), $\sum_{i=1}^N \chi_{E_i^k} = \chi_U$ as $L^1$ functions, we conclude that 
\[
\mathcal{L}^{n+1}\left(U \setminus \bigcup_{i=1}^N E_i \right) =0\,, \qquad \mbox{and} \qquad \mathcal{L}^{n+1}(E_i \cap E_j) = 0 \quad \mbox{if $i \neq j$}\,,
\]
so that $\bigcup_{i=1}^N E_i$ is an $\mathcal{L}^{n+1}$-partition of $U$. The validity of Theorem \ref{thm:main2}(8) implies conclusion (1), namely that 
\begin{equation} \label{measure_inclusion}
\| \nabla \chi_{E_i} \| \leq \|V\| \quad \mbox{for every $i \in \{1,\ldots, N\}$} \qquad \mbox{and} \qquad \sum_{i=1}^N \| \nabla \chi_{E_i} \| \leq 2\, \|V\|
\end{equation}
in the sense of Radon measures in $U$. As a consequence of \eqref{measure_inclusion}, we have that $\spt\, \| \nabla \chi_{E_i} \| \subset \spt\, \|V\|$ for every $i=1,\ldots,N$. Since $V$ is a stationary integral varifold, the monotonicity formula implies that $\spt\|V\|$ is $\Ha^n$-rectifiable, and $V=\var(\spt\,\|V\|,\theta)$ for some upper semi-continuous $\theta\,\colon\,U \to \mathbb R^+$ with $\theta(x) \ge 1$ at each $x \in \spt\|V\|$. In particular, setting $\Gamma := \spt\,\|V\|$, we have
\begin{equation}\label{gamma1}
\Ha^n(\Gamma) = \| \var(\Gamma,1) \| (U) \leq \| V \| (U) \leq \Ha^n(\Gamma_0)\,,
\end{equation}
where the last inequality is a consequence of \eqref{e:mean curvature bound} and the lower semicontinuity of the weight with respect to varifold convergence. 

\smallskip

Next, we observe that, since $\spt\,\|\nabla
\chi_{E_i}\|\subset\Gamma$, on each connected component of $U\setminus \Gamma$ each $\chi_{E_i}$ is almost everywhere constant. Denoting $\{O_h\}_{h \in \mathbb{N}}$ the connected components of the open set $U \setminus \Gamma$, we may then modify each set $E_i$ ($i \in \{1,\ldots,N\}$) by setting
\[
E_{i}^* := \bigcup_{ \{ O_h \, \colon \, \chi_{E_i} = 1 \quad \mbox{a.e. on }O_h \}} O_h.
\]
By definition, each set $E_i^*$ is open; furthermore, the sets $E_i^*$ are pairwise disjoint, and $\bigcup_{i=1}^N E_i^* = U \setminus \Gamma$. Since for each $i$ we have $\mathcal{L}^{n+1}(E_i \Delta E_i^*) = 0$, and since sets of finite perimeter are defined up to $\mathcal{L}^{n+1}$-negligible sets, we can thus replace the family $\{E_i\}$ with $\{E_i^*\}$, and drop the superscript $\,^*$ to ease the notation. 

\smallskip

Property (2) is a consequence of Theorem \ref{thm:main2}(6), since the convergence $\chi_{E_i^k} \to \chi_{E_i}$ now holds pointwise on $U\setminus{\rm conv}(\Gamma_0\cup\pa\Gamma_0)$. 
We have not excluded the possibility that $\Ha^n(\Gamma)=0$. But this should imply $\|V\|=0$ 
by \eqref{gamma1}, and $\|\nabla\chi_{E_i}\|=0$ for every $i\in\{1,\ldots,N\}$ by \eqref{measure_inclusion},
which is a contradiction to (2). Thus we have necessarily 
$\Ha^n(\Gamma)>0$ and this completes the proof of (3).
In order to conclude the proof, we are just left with the boundary condition (4), namely
\begin{equation} \label{final_bc}
({\rm clos}\,(\spt\,\|V\|)) \setminus U = \pa \Gamma_0\,.
\end{equation}
Towards the first inclusion, suppose that $x \in (\clos\,(\spt\,\|V\|)) \setminus U$, and let $\{x_h\}_{h=1}^\infty$ be a sequence with $x_h \in \spt\|V\|$ such that $x_h \to x$ as $h \to \infty$. If $x \notin \pa \Gamma_0$ then Proposition \ref{vmn} implies that there exists $r > 0$ such that 
\[
\limsup_{k \to \infty} \|V_k\|( U \cap B_r(x)) = 0\,.
\]
By the lower semi-continuity of the weight with respect to varifold convergence, we deduce then that $\|V\|(U \cap U_r(x)) = 0$. For $h$ large enough so that $\abs{x-x_h} < r$ we then have $\|V\| (U \cap U_{r - \abs{x-x_h}}(x_h)) = 0$, thus contradicting that $x_h \in \spt\|V\|$. For the second inclusion, let $x \in \pa \Gamma_0$, and suppose towards a contradiction that $x \notin {\rm clos}(\spt\,\|V\|) \setminus U$. Then, there exists a radius $r > 0$ such that $U_r(x) \cap \spt \, \|V\| = \emptyset$. In particular, $\|\nabla \chi_{E_i}\| (U \cap U_r(x)) = 0$ for every $i \in \{1, \ldots, N \}$. Since $U$ is convex, $U \cap U_r(x)$ is connected, and thus every $\chi_{E_i}$ is either identically $0$ or $1$ in $U_r(x) \cap U$, namely
\begin{equation} \label{one_domain_only}
U_r(x) \cap U \subset E_{\ell} \qquad \mbox{for some $\ell \in \{ 1, \ldots, N \}$}\,.
\end{equation}
Because $x \in \pa \Gamma_0$, by property $(A4)$ in Assumption \ref{ass:main} there are two indices $i \neq i' \in \{1,\ldots,N\}$ and sequences $\{x_j\}_{j=1}^\infty\,, \{x'_j\}_{j=1}^\infty$ with $\lim_{j \to \infty} x_j = x = \lim_{j \to \infty} x_j'$ such that $x_j, x_j' \in \pa U \setminus \pa \Gamma_0$ and $U_{r_j}(x_j) \cap U\subset E_{0,i}$, $U_{r_j'}(x_j') \cap U\subset E_{0,i'}$ for some $r_j, r_j' > 0$. If $z$ denotes any of the points $x_j$ or $x_j'$, Proposition \ref{vmn} and Remark \ref{rmk:hyperplane} ensure the existence of $\rho$ such that $\|V_t\|(B_\rho(z) \cap U) = 0$ for all $t \geq 0$. Again by lower semicontinuity of the weight with respect to varifold convergence, $\| V \| (U_\rho(z) \cap U) = 0$. Since each $U_\rho(z) \cap U$ is connected and $\spt\|\nabla \chi_{E_i}\| \subset \spt\|V\|$ for all $i$, we deduce that $U_{\rho_j}(x_j) \cap U \subset E_{i}$ and $U_{\rho_j'}(x_j') \cap U \subset E_{i'}$ for some $i \neq i'$. Since both $x_j \to x$ and $x_j' \to x$, this conclusion is incompatible with \eqref{one_domain_only}. This completes the proof.
\end{proof}

The stationary varifold $V$ from Corollary \ref{main:cor} is a generalized minimal surface in $U$, and for this reason it can be thought of as a solution to Plateau's problem in $U$ with the prescribed boundary $\pa \Gamma_0$. Brakke flow provides, therefore, an interesting alternative approach to the existence theory for Plateau's problem compared to more classical methods based on mass (or area) minimization. Another novelty of this approach is that the structure of partitions allows to prescribe the boundary datum in the purely \emph{topological} sense, by means of the constraint $({\rm clos}\,(\spt\|V\|)) \setminus U = \pa \Gamma_0$. This adds to the several other possible interpretations of the spanning conditions that have been proposed in the literature: among them, let us mention the \emph{homological} boundary conditions in Federer and Fleming's theory of integral currents \cite{FF60} or of integral currents ${\rm mod}(p)$ \cite{Federer_book} (see also Brakke's covering space model for soap films \cite{Brakke_covering}); the \emph{sliding} boundary conditions in David's sliding minimizers \cite{David_Plateau,David_taylorsthm}; and the \emph{homotopic} spanning condition of Harrison \cite{Harrison14}, Harrison-Pugh \cite{HP16} and De Lellis-Ghiraldin-Maggi \cite{DGM}.

Concerning the latter, we can actually show that, under a suitable extra assumption on the initial partition $\E_0$, a homotopic spanning condition is satisfied at all times along the flow. Before stating and proving this result, which is Proposition \ref{final spanning} below, let us first record the definition of homotopic spanning condition after \cite{DGM}.

\begin{definition}[{see \cite[Definition 3]{DGM}}]
Let $n \geq 2$, and let $\Sigma$ be a closed subset of $\R^{n+1}$. Consider the family
\begin{equation} \label{spanning class}
\mathcal{C}_\Sigma := \left\lbrace \gamma \colon \Sf^1 \to \R^{n+1} \setminus \Sigma \, \colon \, \gamma \mbox{ is a smooth embedding of $\Sf^1$ into $\R^{n+1} \setminus \Sigma$} \right\rbrace\,.
\end{equation}
A subfamily $\mathcal C \subset \mathcal C_\Sigma$ is said to be \emph{homotopically closed} if $\gamma \in \mathcal C$ implies that $\tilde \gamma \in \mathcal C$ for every $\tilde \gamma \in \left[ \gamma \right]$, where $\left[ \gamma \right]$ is the equivalence class of $\gamma$ modulo homotopies in $\R^{n+1} \setminus \Sigma$. Given a homotopically closed $\mathcal{C} \subset \mathcal{C}_\Sigma$, a relatively closed subset $K \subset \R^{n+1} \setminus \Sigma$ is $\mathcal{C}$-\emph{spanning $\Sigma$} if \footnote{With a slight abuse of notation, in what follows we will always identify the map $\gamma$ with its image $\gamma(\mathbb{S}^1) \subset \R^{n+1} \setminus \Sigma$.}
\begin{equation} \label{C-spanning}
K \cap \gamma \neq \emptyset \qquad \mbox{for every $\gamma \in \mathcal{C}$}\,. 
\end{equation}
\end{definition}

\begin{remark}
If $\mathcal C \subset \mathcal{C}_\Sigma$ contains a homotopically trivial curve, then any $\mathcal C$-spanning set $K$ will necessarily have non-empty interior (and therefore infinite $\Ha^n$ measure). For this reason, we are only interested in subfamilies $\mathcal C$ with $\left[ \gamma \right] \neq 0$ for every $\gamma \in \mathcal C$.
\end{remark}

\begin{definition}
We will say that a relatively closed subset $K \subset \R^{n+1} \setminus \Sigma$ \emph{strongly homotopically spans $\Sigma$} if it $\mathcal{C}$-spans $\Sigma$ for \emph{every} homotopically closed family $\mathcal{C} \subset \mathcal{C}_\Sigma$ which does not contain any homotopically trivial curve. Namely, if $K \cap \gamma \neq \emptyset$ for every $\gamma \in \mathcal{C}_\Sigma$ such that $\left[ \gamma \right] \neq 0$ in $\pi_1(\R^{n+1} \setminus \Sigma)$.
\end{definition}

We can prove the following proposition, whose proof is a suitable adaptation of the argument in \cite[Lemma 10]{DGM}.

\begin{proposition} \label{final spanning}

Let $n \geq 2$, and let $U,\Gamma_0,\E_0$ be as in Assumption \ref{ass:main}. Suppose that the initial partition $\E_0$ satisfies the following additional property:
\begin{equation} \label{disconnected components} \tag{$\diamond$}
\begin{split}
&\mbox{Given any two connected components $S_1$ and $S_2$ of $\pa U \setminus \pa \Gamma_0$}\,,\\ &\mbox{there are two indices $i,j \in \{1,\ldots,N\}$ with $i \neq j$}\\&\mbox{such that $S_1 \subset {\rm clos}\,E_{0,i}$ and $S_2 \subset {\rm clos}\,E_{0,j}$}\,. 
\end{split}
\end{equation}
Then, the set $\Gamma(t)$ strongly homotopically spans $\pa \Gamma_0$ for every $t \in \left[0,\infty \right]$.

\end{proposition}

\begin{proof}

Let $\gamma \colon \mathbb{S}^1 \to \R^{n+1} \setminus \pa \Gamma_0$ be a smooth embedding that is not homotopically trivial in $\R^{n+1} \setminus \pa \Gamma_0$. The goal is to prove that, for every $t \in \left[ 0, \infty \right]$, $\Gamma(t) \cap \gamma \neq \emptyset$. First observe that it cannot be $\gamma \subset U$, for otherwise $\gamma$ would be homotopically trivial. For the same reason, since the ambient dimension is $n+1 \geq 3$ also $\gamma \subset \R^{n+1} \setminus {\rm clos}\,U$ is incompatible with the properties of $\gamma$. Hence, we conclude that $\gamma$ must necessarily intersect $\pa U$. We first prove the result under the additional assumption that $\gamma$ and $\pa U$ intersect transversally. We can then find finitely many closed arcs $I_h = \left[ a_h, b_h \right] \subset \Sf^1$ with the property that $\gamma \cap U = \bigcup_{h} \gamma(\left( a_h, b_h \right))$, and $\gamma \cap ( \pa U \setminus {\pa\Gamma_0} ) = \bigcup_{h} \{ \gamma(a_h), \gamma(b_h)  \}$. If there is $h$ such that $\gamma(a_h)$ and $\gamma(b_h)$ belong to two distinct connected components of $\pa U \setminus {\pa\Gamma_0}$, then \eqref{disconnected components} implies that the arc $\sigma_h := \left. \gamma \right|_{\left( a_h, b_h \right)}$ must intersect $U \cap \pa E_i(0)$ for some $i=1,\ldots,N$. In fact, since the labeling of the open partition at the boundary of $U$ is invariant along the flow, the same conclusion holds for every $t \in \left[ 0, \infty \right]$. In particular, in this case $\gamma$ intersects $\bigcup_{i} (\pa E_i(t) \cap U) = \Gamma(t)$ for every $t \in \left[0,\infty\right]$. Hence, if by contradiction $\gamma$ has empty intersection with $\Gamma(t)$, then necessarily for every $h$ there is a connected component $S_h$ of $\pa U \setminus {\pa\Gamma_0}$ such that $\gamma(a_h), \gamma(b_h) \in S_h$ (note that it may be $S_h = S_{h'}$ for $h \neq h'$). Since each $S_h$ is connected, for every $h$ we can find a smooth embedding $\tau_h \colon I_h \to S_h$ with the property that $\tau_h(a_h) = \gamma(a_h)$ and $\tau_h(b_h) = \gamma(b_h)$. Furthermore, this can be achieved under the additional condition that $\tau_h(I_h) \cap \tau_{h'} (I_{h'}) = \emptyset$ for every $h \neq h'$. We can then define a piecewise smooth embedding $\tilde \gamma$ of $\Sf^1$ into $\R^{n+1} \setminus {\pa\Gamma_0}$ such that $\left.\tilde \gamma \right|_{I_h} := \left. \tau_h \right|_{I_h}$ for every $h$, and $\tilde \gamma = \gamma$ on the open set $\Sf^1 \setminus \bigcup_{h} I_h$. We have $\left[ \tilde \gamma \right] = \left[ \gamma \right]$ in $\pi_1(\R^{n+1} \setminus {\pa\Gamma_0})$. We can then construct a smooth embedding $\hat \gamma \colon \Sf^1 \to \R^{n+1} \setminus {\pa\Gamma_0}$ such that $\left[ \hat \gamma \right] = \left[ \gamma \right]$ in $\pi_1(\R^{n+1} \setminus {\pa\Gamma_0})$, and with $\hat \gamma \subset \R^{n+1} \setminus \pa U$. Since $n+1 \geq 3$ this contradicts the assumption that $\left[ \gamma \right] \neq 0$ and completes the proof if $\gamma$ and $\pa U$ intersect transversally.

\smallskip

Finally, we remove the transversality assumption. Let $\delta = \delta (\pa U) > 0$ be such that the tubular neighborhood $(\pa U)_{2\delta}$ has a well-defined smooth nearest point projection $\Pi$, and consider, for $\abs{s} < \delta$, the open sets $U_s$ having boundary $\pa U_s = \left\lbrace x - s\, \nu_{U}(x) \, \colon \, x \in \pa U \right \rbrace$, where $\nu_{U}$ is the exterior normal unit vector field to $\pa U$. Since $\gamma$ is smooth, by Sard's theorem $\gamma$ intersects $\pa U_s$ transversally for a.e. $\abs{s} < \delta$. Fix such an $s \in \left( 0, \delta \right)$, and let $\Phi_s \colon \R^{n+1} \to \R^{n+1}$ be the smooth diffeomorphism of $\R^{n+1}$ defined by
\begin{equation} \label{diffeo}
\Phi_s(x) := x + \varphi_s(\rho_U(x)) \, \nu_U(\Pi(x))\,,
\end{equation}
where 
\[
\rho_U(x) := 
\begin{cases}
\abs{x - \Pi(x)} & \mbox{if } x \in (\pa U)_{2\delta} \cap U \\
- \abs{x - \Pi(x)} & \mbox{if } x \in (\pa U)_{2\delta} \setminus U
\end{cases}
\]
is the signed distance function from $\pa U$, and $\varphi_s = \varphi_s(t)$ is a smooth function such that
\[
\varphi_s(t) = 0 \quad \mbox{for all $|t| \geq 2s$}\,, \qquad \mbox{and} \qquad \varphi_s(s)=s\,.
\]

In particular, $\Phi_s$ maps $\pa U_s$ diffeomorphically onto $\pa U$, and furthermore
\begin{equation} \label{convergence diffeo}
\Phi_{s} \to {\rm id}  \qquad \mbox{uniformly on $\R^{n+1}$ as $s \to 0+$}\,.
\end{equation}
Since $\gamma$ intersects $\pa U_s$ transversally, the curve $\Phi_s \circ \gamma$ intersects $\pa U$ transversally. Furthermore, since $\gamma$ and ${\pa\Gamma_0}$ are two compact sets with empty intersection, \eqref{convergence diffeo} implies that if we choose $s$ sufficiently small then also $(\Phi_s \circ \gamma) \cap {\pa\Gamma_0} = \emptyset$. Since $\left[ \Phi_s \circ \gamma \right] = \left[ \gamma \right] \neq 0$ in $\pi_1(\R^{n+1} \setminus \pa\Gamma_0)$, the first part of the proof guarantees that for every $t \in \left[ 0, \infty \right]$ we have $\Gamma(t) \cap (\Phi_s \circ \gamma) \neq \emptyset$. For every $t$ we then have points $z_s(t) \in \Gamma(t) \cap \Phi_s \circ \gamma$. Along a sequence $s_h \to 0+$, then, by compactness, \eqref{convergence diffeo}, and the fact that each set $\Gamma(t)$ is closed, we have that the points $z_{s_h}(t)$ converge to a point $z_0(t) \in \Gamma(t) \cap \gamma$. The proof is now complete.
\end{proof}

\begin{example} \label{ex:two circles}
Suppose that $U = U_1(0) \subset \R^3$, and $\pa \Gamma_0$ is the union of two parallel circles contained in $\mathbb{S}^2 = \pa U$ at distance $2h$ from one another, with $h \in \left(0,1\right)$. Then, $\pa U \setminus \pa \Gamma_0$ consists of the union of three connected components $S_{u} \cup S_{l} \cup S_{d}$ (here $u,l,d$ stand for \emph{up, lateral}, and \emph{down}, respectively). If $h$ is suitably small, then there are two smooth minimal catenoidal surfaces $C_1 \subset U$ and $C_2 \subset U$, one stable and the other unstable, satisfying ${\rm clos}(C_j) \setminus U  = \pa \Gamma_0$. Nonetheless if the initial partition $\{E_{0,i}\}_{i}$ satisfies \eqref{disconnected components}, then, as a consequence of Proposition \ref{final spanning}, both $C_1$ and $C_2$ are \emph{not} admissible limits of Brakke flow as in Corollary \ref{main:cor}, since there exists a smooth and homotopically non-trivial embedding $\gamma \colon \mathbb{S}^1 \to \R^3 \setminus \pa\Gamma_0$ having empty intersection with each of them. For instances, if $N=3$ and the initial partition is such that $S_u \subset {\rm clos}\,E_{0,1}$, $S_l \subset {\rm clos}\,E_{0,2}$, and $S_d \subset {\rm clos}\,E_{0,3}$, then the corresponding Brakke flows will converge, instead, to a \emph{singular} minimal surface $\Gamma$ in $U$ consisting of the union $\Gamma = \tilde C_1 \cup \tilde C_2 \cup D$, where $\tilde C_j$ are pieces of catenoids, and $D$ is a disc contained in the plane $\{z=0\}$, which join together forming $120^{\circ}$ angles along the ``free boundary'' circle $\Sigma = \partial D$; see Figure \ref{singular_cat}.

\begin{figure}[h]
\includegraphics[scale=0.6]{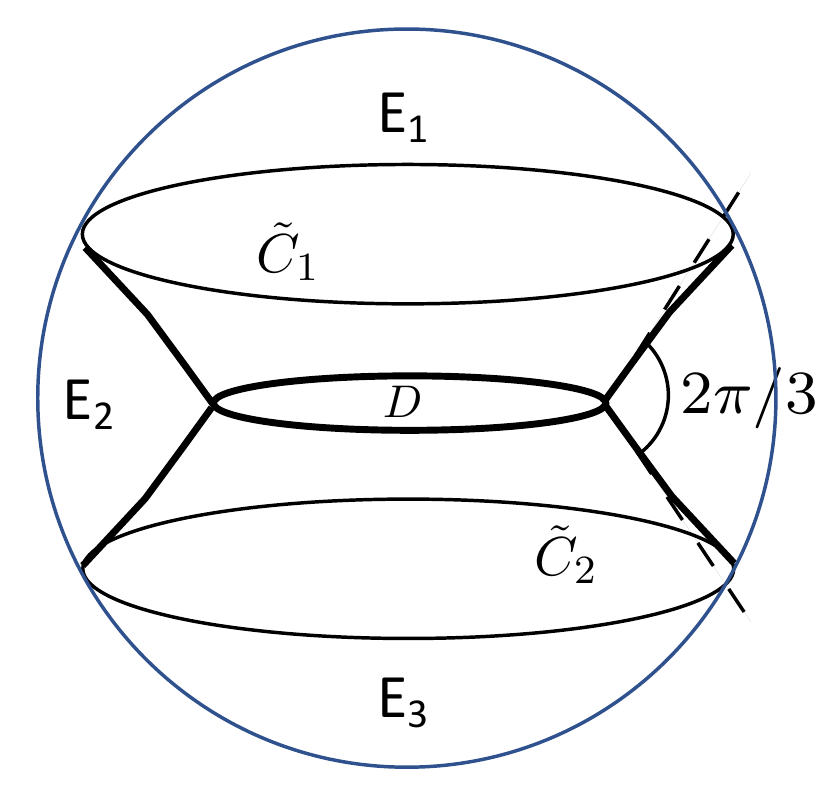}
\caption{The singular limit varifold detailed in Example \ref{ex:two circles}.} \label{singular_cat}
\end{figure}

\end{example}

We will conclude the section with three remarks containing some interesting possible future research directions.

\begin{remark}
First, we stress that the requirements on $\pa \Gamma_0$ are rather flexible, above all in terms of regularity. It would be interesting to characterize, for a given strictly convex domain $U \subset \R^{n+1}$, all its \emph{admissible boundaries}, namely all subsets $\Sigma \subset \pa U$ such that there are $N \geq 2$ and $\E_0$, $\Gamma_0$ as in Assumption \ref{ass:main} such that $\Sigma = \pa \Gamma_0$. A first observation is that admissible boundaries do not need to be countably $(n-1)$-rectifiable, or to have finite $(n-1)$-dimensional Hausdorff measure: for example, it is not difficult to construct an admissible $\Sigma \subset \pa U_1(0)$ in $\R^2$ with $\Ha^1(\Sigma) > 0$, essentially a ``fat'' Cantor set in $\mathbb{S}^1$. The assumption $(A4)$ requires any admissible boundary to have empty interior. It is unclear whether this condition is also sufficient for a subset $\Sigma$ to be admissible.
\end{remark}

\begin{remark}
Let us explicitly observe that, even in the case when $\Gamma_0$ (or more precisely $V_0 := \var(\Gamma_0,1)$) is stationary, it is false in general that $V_{t} = V_0$ for $t > 0$. In other words, the approximation scheme which produces the Brakke flow $V_t$ may move the initial datum $V_0$ even when the latter is stationary. A simple example is a set consisting of two line segments with a crossing, for which multiple non-trivial solutions (depending on the choice of the initial partition) are possible; see Figure \ref{fig:multiple_sols}. In fact, one can
prove that such one-dimensional configuration \emph{cannot} stay time-independent with respect to the Brakke flow constructed in the present paper: \cite[Theorem 2.2]{KiTo20}, indeed, shows that one-dimensional Brakke flows obtained in the present paper and in \cite{KimTone} necessarily satisfy a specific angle condition at junctions
for a.e. time, with the only admissible angles being $0$, $60$, or $120$ degrees. Thus, depending
on the initial labeling of domains, one of the two evolutions depicted in Figure \ref{fig:multiple_sols} has to occur instantly. 

\begin{figure}[h]
\includegraphics[scale=0.65]{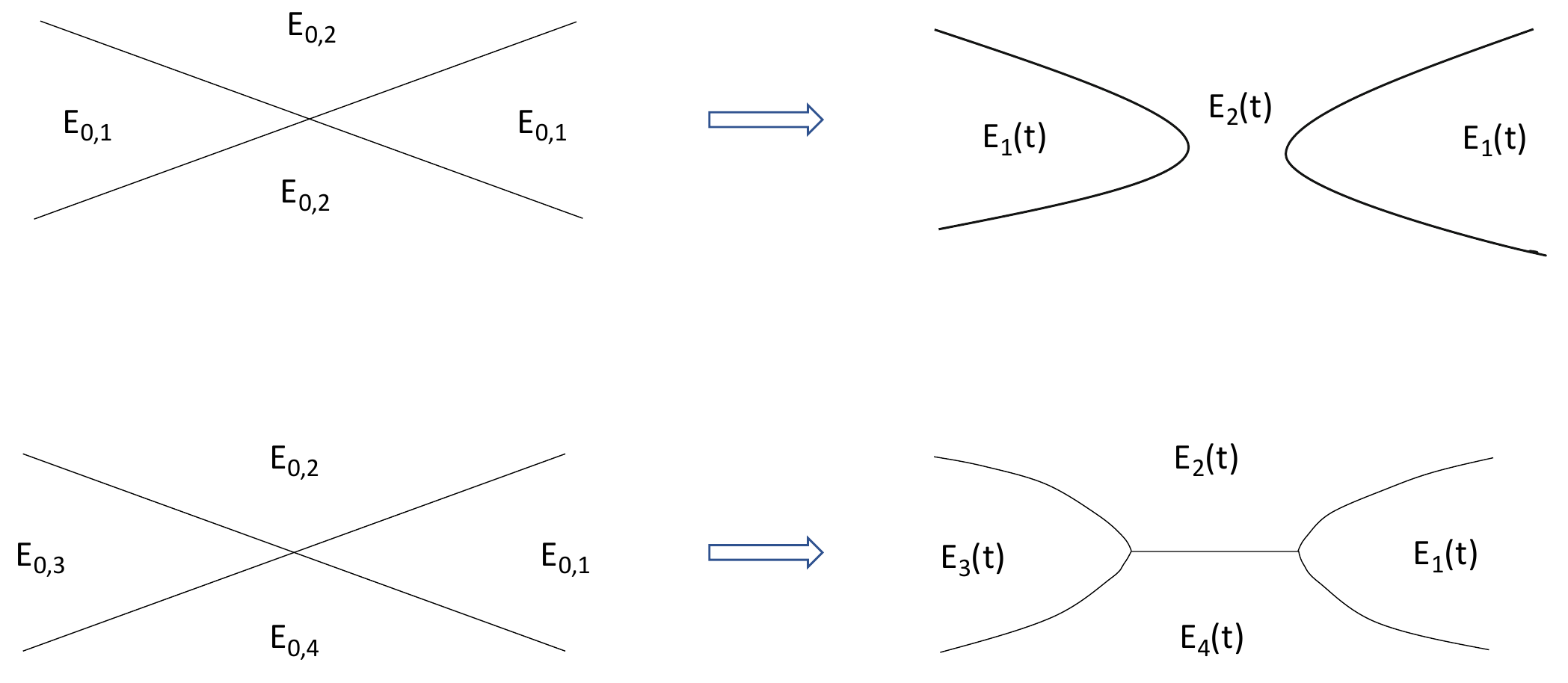}
\caption{Non-uniqueness without loss of mesure when $N=2$ (top) or $N=4$ (bottom).} \label{fig:multiple_sols}
\end{figure}

If $\Gamma_0$ is a smooth minimal surface
with smooth boundary $\pa\Gamma_0$, the uniqueness theorem for classical MCF should allow $\Gamma_t\equiv\Gamma_0$ as the unique solution, even if
the latter is unstable (i.e.~the second variation is negative for some direction). In other words, in the smooth case we expect that there is no other Brakke flow starting from $\Gamma_0$ other than the time-independent
solution (notice, in passing, that both the area-reducing Lipschitz deformation step and the motion by smoothed mean curvature step in our time-discrete approximation of Brakke flow trivialize in this case - at least locally -, because smooth minimal surfaces are already locally area minimizing at suitably small scales around each point). 

On the other hand, in \cite{stu-tone2} we show that time-dependent solutions may arise even from the existence, on $\Gamma_0$, of singular points at which $V_0$ has a \emph{flat} tangent cone, that is a tangent cone which is a plane $T$ with multiplicity $Q \ge 2$. It would be interesting to characterize the regularity properties of those stationary $\Gamma_0$ with
$E_{0,1},\ldots,E_{0,N}$ satisfying Assumption \ref{ass:main} and $\Ha^n(\Gamma_0\setminus \cup_{i=1}^N
\pa^* E_{0,i})=0$ 
which do not allow any non-trivial Brakke flows (\emph{dynamically stable} stationary varifolds, in the terminology introduced in \cite{stu-tone2}). We expect that such a $\Gamma_0$ 
should have some local measure minimizing properties. 
\end{remark}

\begin{remark}
Let $V$, $\{E_i\}_{i=1}^N$ and $\Gamma$ be as in Corollary \ref{main:cor} obtained as $t_k\to \infty$ along a Brakke flow. 
Since $V$ is integral and stationary, $V=\var (\Gamma,\theta)$ for some 
$\Ha^n$-measurable function $\theta:
\Gamma\to \mathbb N$. One can check that $\Gamma$ and $\{E_i\}_{i=1}^N$ (after removing empty
$E_i$'s if necessary) again satisfy the Assumption \ref{ass:main}, thus 
we may apply Theorem \ref{thm:main} and obtain another Brakke flow with the same
fixed boundary. Note that if we have $\|V\|(\{x\,:\,\theta(x)\geq 2\})>0$, then $\var(\Gamma,1)$ may 
not be stationary, and the Brakke flow starting from non-stationary $\var(\Gamma_,1)$ 
is genuinely time-dependent. 
We then obtain another stationary varifold as $t\to\infty$ by Corollary \ref{main:cor}. 
It is likely that,
after a finite number of iterations, this process produces a unit density stationary varifold which does not move anymore. The other possibility is 
also interesting, in that we would have 
infinitely many different integral stationary varifolds with the same boundary condition,
each having strictly smaller $\Ha^n$ measure than the previous one. 
\end{remark}

\newpage

\bibliographystyle{plain}
\bibliography{MCF_Plateau_biblio}

\end{document}